\documentclass[11pt]{amsart}

\usepackage[margin=1in]{geometry}

\usepackage{setspace}
\usepackage{amsthm}
\usepackage{breqn}
\usepackage{amsmath}
\usepackage{amsfonts}
\usepackage{amssymb}
\usepackage{thmtools} 
\usepackage{tikz-cd} 
\usetikzlibrary{shapes.geometric, positioning, patterns, arrows.meta, backgrounds, knots}

\usepackage{microtype} 
\usepackage{mathtools}
\usepackage{graphicx}
\usepackage{stmaryrd}

\usepackage{enumerate} 
\usepackage{bbm} 
\usepackage{xfrac} 
\usepackage{resmes} 
\usepackage{comment} 

\usepackage[hypertexnames=false, hidelinks]{hyperref}
\usepackage[capitalise]{cleveref}

\newtheorem{theorem}{Theorem}[section]
\newtheorem{lemma}[theorem]{Lemma}
\newtheorem{proposition}[theorem]{Proposition}
\newtheorem{corollary}[theorem]{Corollary}
\newtheorem{remark}[theorem]{Remark}

\newtheorem{definition}[theorem]{Definition}

\crefname{theorem}{Theorem}{Theorems}
\crefname{lemma}{Lemma}{Lemmas}
\crefname{proposition}{Proposition}{Propositions}
\crefname{corollary}{Corollary}{Corollaries}
\crefname{remark}{Remark}{Remarks}
\crefname{remarks}{Remarks}{Remarks}
\crefname{definition}{Definition}{Definitions}
\crefname{question}{Question}{Questions}

\newcommand{\R}{\mathbb{R}} 
\newcommand{\N}{\mathbb N}
\newcommand{\Z}{\mathbb Z}

\newcommand{\e}{\varepsilon}

\newcommand{\I}{\mathcal{I}}
\newcommand{\cB}{\overline{B}}
\newcommand{\deck}{\mathrm {deck}}
\newcommand{\lift}{\mathrm{lift}}
\newcommand{\Haus}{\mathcal{H}}
\newcommand{\one}{\mathbbm{1}}
\newcommand{\diff}{\mathop{}\!\mathrm{d}} 

\newcommand{\id}{\mathrm{id}}
\newcommand{\RE}{R_{\mathrm{Eucl}}}
\newcommand{\link}{\mathrm{link}}

\newcommand{\M}{\mathrm{\mathbf{M}}} 

\usepackage[
backend=biber,
style=alphabetic,
sorting=ynt,
maxalphanames=50,
maxnames=50
]{biblatex}
\allowdisplaybreaks

\title{Plateau's Problem via covering spaces}

\author{James Tissot}

\begin{document}

\begin{abstract}
    In 1995, Brakke proposed a formulation of the Plateau problem for a given boundary $\Gamma$ that allows for triple junctions and tetrahedral singularities. Let $\Gamma$ be a smooth closed curve and let $G=\pi_1(\mathbb{R}^3\setminus \Gamma)$. For each proper finite index subgroup $N$ of $G$, Brakke constructs a $(\M, 0, \infty)$-minimal surface $\Sigma_N$, which is obtained as the projection of the boundary of a perimeter-minimising fundamental domain in the covering space associated to $N$. We advance the theory in two ways. Firstly, we extend Brakke's construction to include all normal subgroups $N\triangleleft G$, i.e. possibly with infinite index. Secondly, we prove a compactness result which implies that there exists a proper normal subgroup $N_0\triangleleft G$ such that
    \[
        \mathrm{Area}(\Sigma_{N_0})=\inf_{\{N\triangleleft G, N \neq G\}}\ \mathrm{Area}(\Sigma_N).
    \]
    A similar result holds when $\Gamma$ has many connected components. Furthermore, we study the spanning and minimising properties of the $\Sigma_N$s and $\Sigma_{N_0}$.
\end{abstract}

\maketitle
\section{Introduction}
Let $\Gamma$ be a disjoint union of smooth closed curves\footnote{See \cref{subsec: setup} for technical remarks about the necessary regularity of $\Gamma$.}. In \cite{almgren1975elliptic, almgren1976existence}, Almgren introduced the following notion of area minimality: a non-empty closed set $\Sigma$ of finite area is \textbf{$(\M, 0, \infty)$-minimal with respect to the boundary $\Gamma$} if the area of $\Sigma$ cannot be decreased by a Lipschitz deformation whose support is compact and disjoint from $\Gamma$ (see \cref{def: Almgren minimal sets}). Remarkably, Taylor proved shortly after in \cite{taylor1976structure} that such minimal sets have the same interior regularity as what is experimentally observed of soap films \cite{plateau1873statique}; namely the sets are smooth $2$-dimensional surfaces away from line singularities (called \textbf{triple junctions}), where three sheets of surface come together at 120 degrees, and point singularities (called \textbf{tetrahedral singularities}), where four such lines come together as the cone over the vertices of a regular tetrahedron. Hence, it is natural to look for soap films spanning $\Gamma$ in the space of sets that are $(\M, 0, \infty)$-minimal with respect to $\Gamma$.

In \cite{brakke1995soap}, Brakke introduced a very interesting variational problem whose solutions were $(\M, 0, \infty)$-minimal sets. He realised that one can model many surfaces spanning $\Gamma$ as boundaries of fundamental domains in suitably chosen covering spaces $p: \widetilde M \rightarrow M$ of $M := \R^3 \setminus \Gamma$ (see Example \ref{ex: simple example}). Following this intuition, he proved, among many other things, that one can find a perimeter minimising fundamental domain $D$ in any finite degree covering space $p: \widetilde M \rightarrow M$, and that the boundary of $D$ projects to a $(\M, 0, \infty)$-minimal set in $M$.

We start by extending this result to the case where $\widetilde M$ is a \textit{normal} cover of possibly \textit{infinite degree}. By assuming normality, a fundamental domain can be defined as a set $D$ whose orbit under the action of deck transformations tessellate $\widetilde M$ (see \cref{def: fundamental domains}). This turns out to be a very nice definition to work with.

We use the notation $p: \widetilde M(H) \rightarrow M$ to denote the covering space corresponding to the subgroup $H \leq \pi_1(M)$ via the standard Galois correspondence (see \cref{subsec: relevant covering space theory}). Recall that for a set $E$, the \textbf{perimeter} is the area of the \textbf{reduced boundary} $\partial^* E$ (see \cref{subsec: relevant theory of sets of finite perimeter}). Furthermore, we recall the spanning condition from \cite{parks1992soap}: a surface $\Sigma$ \textbf{(homotopically) spans $\Gamma$ modulo $H$} if the image of the map on fundamental groups $\pi_1(M \setminus \Sigma) \rightarrow \pi_1(M)$, induced by the inclusion $M \setminus \Sigma \rightarrow M$, is contained in $H$.

\begin{restatable}{theorem}{first} \label{ithm: minimising fundamental domains and projected boundary}
    Let $p: \widetilde M(H) \rightarrow M$ be a non-trivial normal covering space. There exists a fundamental domain $D$ in $\widetilde M(H)$ minimising perimeter among all fundamental domains in $\widetilde M(H)$. Furthermore, the projected boundary $\overline{p(\partial^* D)}$ has positive area, homotopically spans $\Gamma$ modulo $H$, is a $(\M, 0, \infty)$-minimal set with respect to $\Gamma$, and minimises area among all $(\M, 0, \infty)$-minimal sets\footnote{With no extra work, one shows that $\overline{p(\partial^* D)}$ minimises area among all $(\M, 0, \delta)$-minimal sets (see \cite{taylor1976structure} for a definition) homotopically spanning $\Gamma$ modulo $H$, but we stick with $(\M, 0, \infty)$-minimal sets for simplicity.} homotopically spanning $\Gamma$ modulo $H$.
\end{restatable}

To motivate the above result, note that there are conjectural examples of $(\M, 0, \infty)$-minimal sets which can only arise from an \textit{infinite degree} cover via \cref{ithm: minimising fundamental domains and projected boundary} (see Example \ref{ex: spiral boundary}).

Given a perimeter minimising fundamental domain $D$ in $\widetilde M(H)$, we define the \textbf{projected boundary} of $D$ as $\Sigma_H(D) := \overline{p(\partial^* D)}$. A natural question is: how different are the $\Sigma_H(D)$s obtained in \cref{ithm: minimising fundamental domains and projected boundary}? If $\Sigma$ spans $\Gamma$ modulo $H$, it also spans $\Gamma$ modulo $H'$ for $H \leq H'$. Therefore, it is possible that many of the $\Sigma_H(D)$s are the same for different $H$s. For instance, this is the case when $\Gamma$ is a circle since the unique $(\M, 0, \infty)$-minimal set with respect to $\Gamma$ is the disk. On the other hand, Example \ref{ex: spiral boundary} suggests that it is also possible for all the $\Sigma_H(D)$s to be distinct.

Since the $\Sigma_H(D)$s can be quite different, a second natural question is which one, if any, of the $\Sigma_H(D)$s achieves the least possible area? For this question to be meaningful, some subgroups cannot be considered as they produce $\Sigma_H(D)$s which have small area because they do not span all of $\Gamma$. Indeed, when $\Gamma$ consists of multiple closed curves, some $\Sigma_H(D)$s may `attach' to only \textit{some} of these curves (see Example \ref{ex: two circles}). These $\Sigma_H(D)$s arise when the subgroup $H$ contains a \textbf{meridian} of $\Gamma$, that is a closed curve in $M$ that has linking number 1 with a single component of $\Gamma$, and linking number 0 with all the others. To rule such behaviour out, we say $H$ is \textbf{fully spanning} if it contains no meridians, and that $\widetilde M(H)$ is \textbf{fully spanning} whenever $H$ is. Our next two results show that among $\Sigma_H(D)$s with $H$ fully spanning, we can find one of least area.

\begin{restatable}{theorem}{second} \label{ithm: comparison of different covers}
    There exists a normal fully spanning cover whose perimeter minimising fundamental domains have least perimeter among any fundamental domain from any normal fully spanning cover.
\end{restatable}

Combining \cref{ithm: minimising fundamental domains and projected boundary,ithm: comparison of different covers}, we get the sought after $\Sigma_H(D)$ with least area:

\begin{restatable}{corollary}{third} \label{icor: comparison of different surfaces}
    Among sets which are $(\M, 0, \infty)$-minimal with respect to $\Gamma$ and homotopically span $\Gamma$ modulo some normal fully spanning $H'$, there is a set $\Sigma$ of least area.
\end{restatable}

To conclude, we wish to note that the topology of the cover can, in specific cases, provide additional information about the regularity and minimising properties of the $\Sigma_H(D)$s. We restrict to the case $\Gamma$ has a single connected component for simplicity. Let $\widetilde M(K_2)$ be the unique double cover of $M$. Let $\widetilde M(K_3)$ be the normal triple cover where $K_3$ is the kernel of the map $\pi_1(M) \rightarrow \Z \rightarrow \Z_3$ given by the composition of the abelianisation of $\pi_1(M)$ and the homomorphism sending $1 \in \Z \mapsto 1 \in \Z_3$ (note that $\widetilde M(K_3)$ is just the standard cyclic triple cover from knot theory).
\begin{restatable}{proposition}{fourth} \label{iprop: projected boundary from specific covers}
    Throughout, let $\Gamma$ have a single connected component. For any perimeter minimising fundamental domain $D$ in $\widetilde M(K_2)$, $\Sigma_{K_2}(D)$ is smooth with smooth boundary $\Gamma$. Furthermore, $\Sigma_{K_2}(D)$ minimises area among all compact smooth surfaces with smooth manifold boundary $\Gamma$.
    
    For any perimeter minimising fundamental domain $D$ in $\widetilde M(K_3)$, $\Sigma_{K_3}(D)$ has no interior tetrahedral singularities. Furthermore, $\Sigma_{K_3}(D)$ has area no greater than any compact surface that has smooth manifold boundary $\Gamma$, and that has smooth and oriented interior away from finitely many smooth triple junction line singularities, where the orientation induced on the triple junction by each smooth piece agrees.
\end{restatable}

\subsection{Unresolved questions} \label{subsec: unresolve questions}
Recall that an important simplifying assumption in this work is that the covering spaces we consider are normal. This property is used repeatedly in the arguments, but morally it is useful because it allows for a simple and robust definition of a fundamental domain. It would be interesting to study the case of non-normal covers more closely. This has already been considered in the finite degree case in \cite{brakke1995soap, amato2017constrained, bellettini2018triple, bellettini2018covers}. The following elementary proposition demonstrates that the non-normal case is essential to producing some specific $(\M, 0, \infty)$-minimal sets (see Example \ref{ex: film not fully wetting}):
\begin{restatable}{proposition}{fifth} \label{iprop: no partial wetting in normal setting}
    Suppose $\Gamma$ consists of a single possibly knotted closed curve. It is impossible to find a non-trivial normal cover $p: \widetilde M(H) \rightarrow M$ such that, for some $D$, $\Sigma_H(D)$ only partially wets $\Gamma$, meaning that not all the Wirtinger generators of the knot intersect $\Sigma_H(D)$ (see \cref{subsec: knot group}). However, this is possible in a non-normal cover of $M$ - see \cite[Theorem 11.3]{brakke1995soap}.
\end{restatable}

On a different note, it is plausible, but unclear, whether the infimum of \cref{ithm: comparison of different covers} is achieved in a finite degree normal covering space of $M$. Answering this would be a first step to a better understanding of the properties of the minimisers of \cref{ithm: comparison of different covers}.

\subsection{Examples}
We collect examples that help illustrate previous considerations.

\subsubsection{$\Gamma$ is a circle, $\widetilde M$ is the double cover} \label{ex: simple example}
We consider this simple boundary curve and cover to illustrate how a surface $\Sigma$ corresponds to the boundary of a fundamental domain $D$, and vice versa. See \cref{fig: simple example}. The left hand side of the figure depicts the double cover $\widetilde M$ of $M = \R^3 \setminus \Gamma$. It is obtained by taking two copies of $M$ (depicted on either side of the dashed line in the top left of \cref{fig: simple example}) and gluing across the disk as indicated by the arrows: a path descending onto the left disk emerges out of the right disk, and a path descending onto the right disk emerges out of the left disk. The right hand side of \cref{fig: simple example} shows how a surface $\Sigma$ corresponds to $\partial D$. One obtains $D$ from $\Sigma$ as follows: fix a basepoint $x_0$ and a lift $\widetilde x_0$ of $x_0$, then lift each path in $M \setminus \Sigma$ starting at $x_0$ to a path in $\widetilde M$ starting at $\widetilde x_0$, and finally let $D$ be the collection of terminal endpoints of all such lifts (some paths, their lifts, and their terminal endpoints are depicted in yellow in \cref{fig: simple example}). For more details, see \cref{prop: lifting thick surface to fundamental domain}. One obtains $\Sigma$ from $D$ by projecting $\partial D$ into $M$. For more details, see \cref{prop: area of projection of reduced boundary}.

\begin{figure}
    \centering
    \resizebox{0.6\textwidth}{!}{%
        \begin{tikzpicture}[
            scale=1.0,
            every node/.style={transform shape},
            mathnode/.style={font=\fontfamily{cmr}\selectfont\itshape}, 
            topo/.style={blue!50!black, thick, smooth},
            network/.style={yellow!70!black, thick, smooth},
            dot/.style={circle, fill=yellow!70!black, inner sep=0pt, minimum size=3pt},
            vector/.style={->, >=Stealth, ultra thick}
        ]
        
            \draw[thick] (0, 0) -- (16, 0);
            
            \draw[ultra thick] (8, -4.5) -- (8, 4.5);
            
            \draw[dashed] (4, 0.4) -- (4, 3.5);
            \draw[dashed] (12, 0.4) -- (12, 3.5);

            \node[mathnode, anchor=north west] at (0.3, -0.3) {$M = \mathbb{R}^3 \setminus \Gamma$};
            \node[mathnode, anchor=north west] at (8.3, -0.3) {$M = \mathbb{R}^3 \setminus \Gamma$};
            \node[mathnode, anchor=north west] at (0.3, 4.2) {$\widetilde{\text{M}}$};
            \node[mathnode, anchor=north west] at (8.3, 4.2) {$\widetilde{\text{M}}$};
        
            \node[mathnode, anchor=north west] at (2.3, -2.3) {$\Gamma$};
            \node[mathnode, anchor=north west] at (10.3, -2.3) {$\Gamma$};
        
            \node[mathnode, anchor=north west, blue] at (10.9, -0.95) {$\Sigma$};

            \node[mathnode, anchor=north west, blue!80] at (9, 1.0) {$D$};
            \node[mathnode, anchor=north west, blue!80] at (12.7, 3.05) {$D$};

            \tikzset{
                standard ellipse/.style={
                    draw, 
                    thick, 
                    ellipse, 
                    minimum width=2.8cm, 
                    minimum height=0.8cm
                },
                shaded ellipse/.style={
                    standard ellipse,
                    pattern=north west lines,
                    pattern color=orange!50!black
                }
            }
        
            
            \fill[blue!20, opacity=0.4, even odd rule]
                (8.1, 0.1) rectangle (11.9, 4.4)
                (8.6, 2.25) 
                .. controls (8.6, 3.2) and (9.2, 3.2) .. (9.6, 2.9)  
                .. controls (9.8, 2.7) and (10, 2.7) .. (10.2, 2.9) 
                .. controls (10.4, 3.3) and (10.6, 4.3) .. (11.1, 4.15) 
                .. controls (11.4, 3.9) and (11.4, 2.8) .. (11.4, 2.25)
                arc[start angle=0, end angle=-180, x radius=1.4cm, y radius=0.4cm] -- cycle
                (10.9, 3.6) .. controls (10.8, 3.3) and (10.8, 3.1) .. (11.1, 2.7)
                .. controls (11.0, 2.85) and (10.95, 3.3) .. (10.85, 3.45) -- cycle;
        
            \fill[blue!20, opacity=0.4] 
                (12.6, 2.25) 
                .. controls (12.6, 3.2) and (13.2, 3.2) .. (13.6, 2.9)  
                .. controls (13.8, 2.7) and (14, 2.7) .. (14.2, 2.9) 
                .. controls (14.4, 3.3) and (14.6, 4.3) .. (15.1, 4.15) 
                .. controls (15.4, 3.9) and (15.4, 2.8) .. (15.4, 2.25)
                arc[start angle=0, end angle=-180, x radius=1.4cm, y radius=0.4cm] -- cycle
                [even odd rule]
                (14.9, 3.6) .. controls (14.8, 3.3) and (14.8, 3.1) .. (15.1, 2.7)
                .. controls (15.0, 2.85) and (14.95, 3.3) .. (14.85, 3.45) -- cycle;
        
            \node[shaded ellipse] at (2, 2.25) {};
            \node[shaded ellipse] at (6, 2.25) {};
            \node[shaded ellipse] at (10, 2.25) {};
            \node[shaded ellipse] at (14, 2.25) {};
        
            \node[standard ellipse] at (4, -2.25) {};
            \node[standard ellipse] at (12, -2.25) {};

            \draw[topo] (8.6, 2.25) 
                .. controls (8.6, 3.2) and (9.2, 3.2) .. (9.6, 2.9)
                .. controls (9.8, 2.7) and (10, 2.7) .. (10.2, 2.9)
                .. controls (10.4, 3.3) and (10.6, 4.3) .. (11.1, 4.15)
                .. controls (11.4, 3.9) and (11.4, 2.8) .. (11.4, 2.25);
            \draw[topo] (10.9, 3.6) .. controls (10.8, 3.3) and (10.8, 3.1) .. (11.1, 2.7);
            \draw[topo] (10.85, 3.45) .. controls (10.95, 3.3) and (11.0, 3.2) .. (11.0, 2.85);
        
            \draw[topo] (10.6, -2.25) 
                .. controls (10.6, -1.3) and (11.2, -1.3) .. (11.6, -1.6)
                .. controls (11.8, -1.8) and (12, -1.8) .. (12.2, -1.6)
                .. controls (12.4, -1.2) and (12.6, -0.2) .. (13.1, -0.35)
                .. controls (13.4, -0.6) and (13.4, -1.7) .. (13.4, -2.25);
            \draw[topo] (12.9, -0.9) .. controls (12.8, -1.2) and (12.8, -1.4) .. (13.1, -1.8);
            \draw[topo] (12.85, -1.05) .. controls (12.95, -1.2) and (13.0, -1.3) .. (13.0, -1.65);
        
            \draw[topo] (12.6, 2.25) 
                .. controls (12.6, 3.2) and (13.2, 3.2) .. (13.6, 2.9)
                .. controls (13.8, 2.7) and (14, 2.7) .. (14.2, 2.9)
                .. controls (14.4, 3.3) and (14.6, 4.3) .. (15.1, 4.15)
                .. controls (15.4, 3.9) and (15.4, 2.8) .. (15.4, 2.25);
            \draw[topo] (14.9, 3.6) .. controls (14.8, 3.3) and (14.8, 3.1) .. (15.1, 2.7);
            \draw[topo] (14.85, 3.45) .. controls (14.95, 3.3) and (15.0, 3.2) .. (15.0, 2.85);

            \coordinate (x0_coord) at (12, -3.4);
            \node[dot] (x0) at (x0_coord) {};
            \node[mathnode, anchor=north west] at (12, -3.4) {$x_0$};
        
            \draw[network] (x0) .. controls (11, -2.8) and (11, -2.8) .. (10.2, -3.0) node[dot] {};
        
            \draw[network] (x0) .. controls (12.1, -3.2) and (12.2, -2.9) .. (12.2, -2.8);
            \draw[network] (12.2, -2.5) .. controls (12.2, -2.0) and (12.3, -2.0) .. (12.6, -1.1) node[dot] {};
        
            \draw[network] (x0) .. controls (14.55, -2.7) and (14.3, 1.2) .. (12.4, -0.6) node[dot] {};
        
            \coordinate (tilde x0_coord) at (10, 1.15);
            \node[dot] (tilde x0) at (tilde x0_coord) {};
            \node[mathnode, anchor=north west] at (10, 1.15) {$\widetilde{x}_0$};
        
            \begin{scope}[xshift=-2cm, yshift=4.55cm]
                \draw[network] (tilde x0) .. controls (11, -2.8) and (11, -2.8) .. (10.2, -3.0) node[dot] {};
                \draw[network] (tilde x0) .. controls (12.1, -3.2) and (12.2, -2.9) .. (12.2, -2.8);
                \draw[network] (tilde x0) .. controls (14.55, -2.7) and (14.3, 1.0) .. (12.4, -0.6) node[dot] {};
            \end{scope}
            
            \begin{scope}[xshift=2cm, yshift=4.55cm]
                \draw[network] (12.2, -2.5) .. controls (12.2, -2.0) and (12.3, -2.0) .. (12.6, -1.1) node[dot] {};
            \end{scope}

            \draw[vector, green!60!black] (1.4, 3.0) -- (1.4, 2.25);
            \draw[vector, red!80!black] (6.6, 3.0) -- (6.6, 2.25);
        
            \begin{pgfonlayer}{background}
                \draw[vector, red!80!black] (2.6, 2.25) -- (2.6, 1.5);
                \draw[vector, green!60!black] (5.4, 2.25) -- (5.4, 1.5);
            \end{pgfonlayer}
        
        \end{tikzpicture}
    }
    \caption{Depiction of Example \ref{ex: simple example}.}
    \label{fig: simple example}
\end{figure}
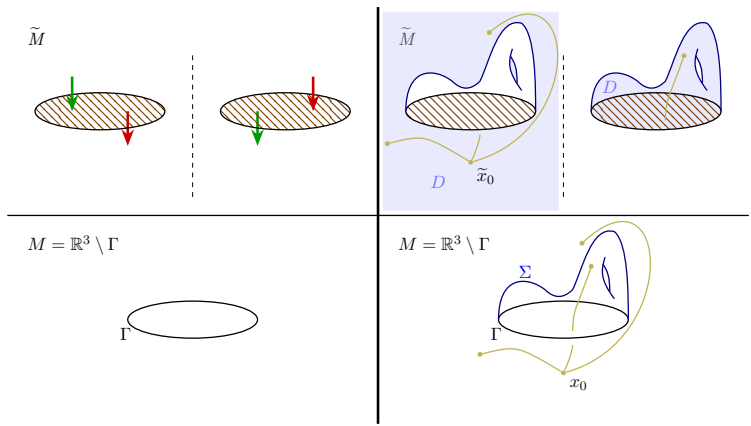

\subsubsection{$\Gamma$ consists of two circles} \label{ex: two circles}
Here $\Gamma = \gamma_1 \cup \gamma_2$ is the union of two unlinked circles. We have that $\pi_1(M)$ is the free group of rank 2 generated by two circles $a$ and $b$ circling $\gamma_1$ and $\gamma_2$ respectively. Note that the disk $\Sigma$ with $\partial \Sigma = \gamma_1$ is $(\M, 0, \infty)$-minimal with respect to $\Gamma$ and spans $\Gamma$ modulo $\langle b \rangle$, the subgroup generated by $b$. Hence, $\Sigma$ could be (and in fact is) obtained as $\Sigma_{\langle b \rangle}(D)$. This shows that when $H$ contains a meridian of $\pi_1(M)$ (in this case $b$), $\Sigma_H(D)$ can `attach' to only some boundary components.

\subsubsection{$\Gamma$ is an infinite spiral} \label{ex: spiral boundary}
See \cite[Figure 15]{brakke1995soap}. This example is due to Brakke \cite[Example 7.9]{brakke1995soap} and illustrates that different covering spaces can yield different minimal surfaces. We present it very informally. Note that here $\pi_1(M) \cong \Z$ and so any subgroup has the form $n \Z$. Consider a surface $\Sigma_n$ which consists of a pattern of $n-1$ filled strips followed by one empty strip, and all strips attach to the slant edge in appropriate ways (see \cite[Figure 15]{brakke1995soap} for an example in the case $n=3$). The discussion in \cite{brakke1995soap} suggests that these are $(\M, 0, \infty)$-minimal with respect to $\Gamma$. $\Sigma_n$ homotopically spans $\Gamma$ modulo $n \Z$ and $\mathrm{Area}(\Sigma_n) \leq \mathrm{Area}(\Sigma_{n+1})$. The first fact suggests that for a given cover $\widetilde M(q \Z)$, all $\Sigma_n$ with $n \Z \leq q \Z$ (equivalently, $n$ is a multiple of $q$) are possible candidates for $\Sigma_{q \Z}(D)$ (as they all span modulo $q \Z$). The second fact says that of these candidates for $\Sigma_{q \Z}(D)$, $\Sigma_q$ has the least area which suggests $\Sigma_{q \Z}(D) = \Sigma_q$. Therefore, each cover produces a different minimal surface. Note, however, that while this example is informative, it does not exactly fit within our setup since $\Gamma$ is not smooth.

\subsubsection{A minimal film that does not homotopically span}
See \cite[Figure 3]{harrison2015plateau}. This is an example of a film that cannot be obtained via \cref{ithm: minimising fundamental domains and projected boundary}. Indeed, consider a loop that passes through the centre of both `coils' and that does not touch the film. This loop is a generator of $\pi_1(M) \cong \Z$ and so shows that the map on fundamental groups $\pi_1(M \setminus \Sigma) \rightarrow \pi_1(M)$, induced by the inclusion, has full image. Therefore, $\Sigma$ does not homotopically span $\Gamma$ modulo $H$ for any proper subgroup $H \leq \pi_1(M)$. It follows that $\Sigma$ can only be obtained from the trivial cover, but fundamental domains in the trivial cover have zero perimeter. We conclude that $\Sigma$ cannot be obtained via \cref{ithm: minimising fundamental domains and projected boundary}. 

\subsubsection{A minimal film that partially wets the boundary} \label{ex: film not fully wetting}
See the figure entitled “Film partially touching wire" in \cite{Brakkeweb}. This is an example of a $(\M, 0, \infty)$-minimal set which does not `touch' the whole boundary. See \cite{parks1992soap, drachman1998soap}. This example was briefly discussed in \cref{subsec: unresolve questions}.

\subsection{Comparison with the literature and main ideas}
Let us start by contrasting the first part of \cref{ithm: minimising fundamental domains and projected boundary} with the existing literature and explain the main ideas of the proof. In \cite{choe1989existence}, Choe proves, when the base manifold is \textit{compact}, that fundamental domains minimising perimeter exist in all covering spaces. He is interested in controlling the topology of the limit fundamental domain, and therefore has to regularise the minimising sequence very delicately via cut-and-paste arguments. In doing so, he uses the compactness of the base to control the number of cut-and-paste procedures required. In our context, compactness of the base $\R^3 \setminus \Gamma$ fails significantly around $\Gamma$ and so it seems such arguments are unavailable.

Brakke proves the result in the case where $\widetilde M$ has finite degree \cite{brakke1995soap}. There are specific finite covers of $M$ for which the result was proved in \cite{amato2017constrained, bellettini2018triple, bellettini2018covers}, using the theory of BV functions. However, one cannot simply extend the finite degree methods to the case where $\widetilde M$ has infinite degree due to possible loss of compactness along the infinite fibres. Indeed, a given minimising sequence of fundamental domains might not be compact because the fundamental domains are drifting off to infinity along the fibres of $\widetilde M$. Worse yet, the fundamental domains might look like many sets drifting off to infinity along different directions and at different speeds, with all the regions connected together by thin tentacles.

In \cite{cesaroni2024periodic}, the authors handle this loss of compactness by making use of the beautiful concentration compactness result of \cite{novaga2022isoperimetricclusters}. In particular, they prove, when the base manifold is \textit{bounded}, that fundamental domains minimising perimeter exist in all regular covering spaces. However, unlike Choe's argument, their method can be straightforwardly modified to apply to the case of an \textit{unbounded} base, which is what we do to prove the first part of \cref{ithm: minimising fundamental domains and projected boundary}. We now explain this compactness argument. The first step is to take the limit of each of these drifting pieces to get some set, albeit this set might live `at infinity'. The question now is how to reassemble these sets `at infinity' into a fundamental domain in $\widetilde M$. In many other geometric minimisation problems, this turns out to be a delicate task as the sets might overlap when one brings them back from infinity into $\widetilde M$, making the resulting object qualitatively different to the minimising sequence. But in our case, because of the compatibility of fundamental domains with the action of deck transformation, we are able to reassemble the pieces without any overlap, making the resulting set a fundamental domain.

The proof that $\Sigma_H(D)$ spans $\Gamma$ modulo $H$ is direct from a topological argument. The proof that $\Sigma_H(D)$ is $(\M, 0, \infty)$-minimal is well known in the literature (see \cite{choe1989existence, brakke1995soap, amato2017constrained}); we carry it out in detail for the convenience of the reader. The proof of the minimality of $\Sigma_H(D)$ among $(\M, 0, \infty)$-minimal sets spanning $\Gamma$ modulo $H$ is based on a lifting argument. Indeed, we show that for $\Sigma$ a $(\M, 0, \infty)$-minimal set spanning $\Gamma$ modulo $H$ there is a fundamental domain of finite perimeter in $\widetilde M(H)$ whose boundary projects into $\Sigma$. The result then follows from the perimeter minimising property of $D$.

\cref{ithm: comparison of different covers} is a generalisation of the first part of \cref{ithm: minimising fundamental domains and projected boundary} where we allow the cover to vary along the minimising sequence. While the strategy of the proof is the same as in \cref{ithm: minimising fundamental domains and projected boundary}, we additionally have to understand in what sense the covering spaces are converging to some limiting covering space along the sequence. This is made simpler by observing that a fundamental domain $D$ in $\widetilde M(H)$ lifts to the universal cover in a canonical way to define a generalised $H$-invariant fundamental domain (\textbf{$H$-domain} in the sequel - see \cref{def: H domains}). Lifting the minimising sequence $D_k$ to the universal cover, we obtain a sequence $F_k$ of $H_k$-domains. Convergence becomes more tractable as the ambient space is now fixed. Considering a suitable topology on the subgroups of $\pi_1(M)$, we can show that along the minimising sequence, the subgroups $H_k$ are converging to a subgroup $H$ and the sets $F_k$ are converging to an $H$-domain $F$. We then show that $F$ projects to a fundamental domain $D$ in $\widetilde M(H)$ achieving the infimum, which is what we were after.

Our compactness theorem for $H$-domains shares many similarities with the works \cite{cesaroni2022minimal, nobili2024lattice, cesaroni2025minimal} on optimal lattice tilings. Indeed, their compactness results also apply to a minimising sequence along which both a fundamental domain and a group vary. The main difference between our arguments and those in the above works is that $H$-domains are \textit{not} fundamental domains, which introduces some additional considerations. Most notably, $H$-domains only have \textit{locally} finite perimeter; and therefore one must restrict to a specific $H$-dependent subspace of the universal cover to obtain a meaningful `perimeter'. These subspaces therefore vary along the minimising sequence and hence, one must take care to prove that these subspaces `converge' appropriately.

It is interesting to compare \cref{ithm: minimising fundamental domains and projected boundary,icor: comparison of different surfaces} with the results of de Lellis-Ghiraldin-Maggi \cite[Theorem 4]{de2017direct} (see also the earlier result \cite[Main Theorem]{harrison2016existence} proved by different means). While the approaches of \cite{harrison2016existence}, \cite{de2017direct}, and the present paper are very different, we obtain similar results. In the terminology set forth above, de Lellis-Ghiraldin-Maggi prove that given a collection of homotopy classes $\mathcal{C} \subset \pi_1(M) \setminus \{e\}$, there is a closed set $\Sigma$ homotopically spanning $\Gamma$ modulo $\pi_1(M) \setminus \mathcal{C}$ which has least area among all such sets\footnote{Note that $\pi_1(M) \setminus \mathcal{C}$ is not necessarily a group, but the homotopic spanning condition still makes sense.}. This result recovers \cref{ithm: minimising fundamental domains and projected boundary} by choosing $\mathcal{C} = \pi_1(M) \setminus H$ and recovers \cref{icor: comparison of different surfaces} by choosing $\mathcal{C}$ to be the (classes of) meridians of $\Gamma$. Using the methods we develop to prove \cref{icor: comparison of different surfaces}, we prove the following version of their result.

\begin{restatable}{corollary}{sixth} \label{icor: similar to de Lellis-Ghiraldin-Maggi}
    Fix $\mathcal{C} \subset \pi_1(M) \setminus \{e\}$. Among all sets which are $(\M, 0, \infty)$-minimal with respect to $\Gamma$ and homotopically span $\Gamma$ modulo some normal $H \subset \pi_1(M) \setminus \mathcal{C}$, there is a set of least area.
\end{restatable}

This differs from \cite[Theorem 4]{de2017direct} because $H$ is assumed to be normal and we are minimising among $(\M, 0, \infty)$-minimal sets satisfying the constraint, as opposed to all closed sets. In view of the discussion in \cref{subsec: unresolve questions}, the assumption on $H$ may be removable.

We conclude by discussing \cref{iprop: projected boundary from specific covers}. The special case concerning $\widetilde M(K_2)$ was already proved in \cite{amato2017constrained} by comparing $\Sigma_H(D)$ to a minimising current mod 2. See also the related result by Guaraco-Lynch \cite{guaraco2024plateau} via analysis of the Allen-Cahn functional. We include a proof of the $\widetilde M(K_2)$ case here for completeness, as the method is exactly the same as the one used to prove the $\widetilde M(K_3)$ case (up to the boundary regularity). For definiteness, let us discuss the $\widetilde M(K_3)$ case. The key part of the result is to show that a given competitor surface $\Sigma$ (i.e. a smooth and oriented surface away from finitely many `oriented' triple junctions) defines an admissible competitor in the perimeter-minimisation problem, i.e. there is a fundamental domain in $\widetilde M(K_3)$ whose projected boundary is this surface. In view of the lifting argument discussed above, this is settled upon showing that such a surface spans $\Gamma$ modulo $K_3$. This is done via homological arguments showing that the orientability conditions of $\Sigma$ allow us to compute its spanning condition.

\subsection{Outline}
In \cref{sec: preliminaries}, we recall some useful information and notation about knot groups, covering spaces, varifolds and sets of finite perimeter. In \cref{sec: fundamental domains and H-domains}, we define fundamental domains and $H$-domains, prove some of their properties, and provide some ways to construct them from surfaces in $M$ and from other fundamental domains/ $H$-domains. These results will be used regularly in the proofs of the above results. In \cref{sec: compactness}, we prove our compactness result for sequences of $H$-domains, which will be the main ingredient in the proofs of \cref{ithm: minimising fundamental domains and projected boundary,ithm: comparison of different covers}. In \cref{sec: projected boundary of minimising fundamental domains}, we study the projected boundary of perimeter minimising fundamental domains. The results we prove here will be used to establish the properties of $\Sigma_H(D)$ collected in \cref{ithm: minimising fundamental domains and projected boundary} and \cref{icor: comparison of different surfaces}. Finally, in \cref{sec: proof of main theorems}, we use the previously proven results, together with some results from the appendices, to prove all the results stated above.

\subsection{Acknowledgments}
I would like to thank my supervisor, Marco Guaraco, for his constant support and profound insights throughout this project. I would also like to thank Stephen Lynch for the opportunity to present some aspects of this work during its preparation. This work was fully funded by the Roth Scholarship provided by the Department of Mathematics, Imperial College London.

\section{preliminaries} \label{sec: preliminaries}
\subsection{Basic notation}
We denote the open (resp. closed) ball centered at $x$ of radius $r$ by $B(x, r)$ (resp. $\cB(x, r)$). $\Haus^n$ denotes the $n$-dimensional Hausdorff measure and we use the shorthand $| \cdot |$ for the top dimensional Hausdorff measure.

\subsection{Setup} \label{subsec: setup}
We consider $\Gamma = \gamma_1 \cup ... \cup \gamma_m$ a collection of disjoint, possibly knotted, smooth closed curves in $\R^3$, possibly linked together\footnote{We need $C^2$ regularity to use Allard's boundary monotonicity and regularity theorems \cite{allard1975boundary} in the blowup analysis in \cref{iprop: projected boundary from specific covers}. In view of \cite{bourni2016allard-type}, it would be enough to assume that each $\gamma_i$ is $C^{1, \alpha}$ with $\alpha > 0$. For all the other results, assuming each $\gamma_i$ is $C^1$ would be enough since all that is needed is that $\gamma_i$ has finite length and that $\Gamma$ defines a tame link. However, we take the $\gamma_i$s to be smooth ($C^\infty$) for simplicity.}. This is the boundary data of the Plateau problem. We let $M = \R^3 \setminus \Gamma$.

\subsection{Fundamental group and first homology of \texorpdfstring{$M$}{M}} \label{subsec: knot group}
We denote the fundamental group of $M$ at $x_0$ by $\pi_1(M, x_0)$. We denote by $\overline{\alpha}$ the inverse of the path $\alpha$ and by $\alpha \cdot \beta$ the concatenation of two paths based at $x_0$. All paths are parametrised by the interval $[0,1]$.

Since $M$ is path connected, the fundamental groups at different basepoints $x_0$ and $y_0$ are isomorphic, via concatenation with a path from $x_0$ to $y_0$ and its inverse, and so we will suppress the basepoint dependency from our notation. Note, however, that this isomorphism is not canonical, as different choices of path can lead to different isomorphisms. Yet, we note that normal subgroups of $\pi_1(M)$ are canonically identified, since changing path changes the isomorphism by a conjugation operation, which normal subgroups are invariant under. Therefore, when referring to a normal subgroup, we also let the basepoint be understood from context and suppress it from our notation.

Recall from knot theory that $\pi_1(M)$ admits a \textbf{Wirtinger presentation} \cite[Chapter 11]{lickorish1997introduction}. We omit the full description of this presentation and instead retain the following facts. Given a link diagram for $\Gamma$, the generators of the presentation are small circles around each lobe in the link diagram. One should think of all these circles being joined to the basepoint via a path that initially travels `out of' the link diagram, and then towards the basepoint. Since a tame link has finitely many lobes, $\pi_1(M)$ is finitely generated. We also note that the relations of the presentation imply that, when $\Gamma$ is a single knot, $\pi_1(M)$ is the normal closure of a single Wirtinger generator.

Recall the first homology of $M$, $H_1(M)$, is isomorphic to the direct sum of $m$ factors of $\Z$ where the generators of each factor are induced by the meridian $\xi_i$  of $\gamma_i$ \cite[Theorem 1.7]{lickorish1997introduction}. The meridians $\xi_1, ..., \xi_m$ can be taken to be (sufficiently small) circles. When an orientation of $\Gamma$ is specified, we endow the $\xi_i$s with the corresponding orientation specified by the right hand rule.

\subsection{Relevant covering space theory} \label{subsec: relevant covering space theory}
For $p: \widetilde M \rightarrow M$ a path-connected covering space of $M$, we denote the associated group of \textbf{deck transformations} by $\deck(p)$. The \textbf{degree} of $p$ refers to the cardinality of $\deck(p)$. For a curve $\alpha$ starting at $x_0$, we denote its unique lift to a curve starting at $\widetilde x_0 \in p^{-1}(x_0)$ by $\lift(\alpha, \widetilde x_0)$. 

$p: \widetilde M \rightarrow M$ is \textbf{normal} if, for any pair of points $\widetilde x, \widetilde x'$ in the same fibre, there is a deck transformation taking $\widetilde x$ to $\widetilde x'$. Path-connected normal covering spaces of $M$ are classified by normal subgroups of $\pi_1(M)$. Indeed, the map that takes a normal cover $p: \widetilde M \rightarrow M$ to the normal subgroup $H = p_{*}(\pi_1(\widetilde M))$ is bijective up to covering space isomorphisms \cite[Theorem 1.38]{hatcher2005algebraic}. We introduce the notation $p: \widetilde M(H) \rightarrow M$ to denote the covering space corresponding to the subgroup $H$. We have that $\deck(p) \cong \pi_1(M) / H$ \cite[Proposition 1.39]{hatcher2005algebraic} which is induced by taking $\alpha \in \pi_1(M, x_0)$ and sending it to the deck transformation taking $\widetilde x_0$ to $\lift(\alpha, \widetilde x_0)(1)$ for all $\widetilde x_0 \in p^{-1}(x_0)$. Note that, since $\pi_1(M)$ is finitely generated (by its Wirtinger generators for example), so is $\deck(p)$ and hence $\deck(p)$ has (at most) countably many elements.

In this paper, we consider normal path-connected covering spaces $p: \widetilde M \rightarrow M$ of $M$. We will omit mention of the path-connectedness to declutter notation. We equip such a covering space with the differential structure induced by $M$, the pullback orientation and with the pullback metric. This turns $p$ into a local isometry, makes $\widetilde M$ locally Euclidean, and ensures that deck transformations are orientation preserving isometries. We equip the cover with Hausdorff measures. Finally, we define the \textbf{Euclidean radius} at $\widetilde x \in \widetilde M$, denoted by $\RE(\widetilde x)$, as half the supremum radius $R$ such that 
\begin{equation*}
    p^{-1}(B(p(\widetilde x), R)) = \bigsqcup_{g \in \deck(p)} gB(\widetilde x, R)
\end{equation*}
and $B(\widetilde x, R)$ is isometric to $B(p(\widetilde x), R)$ via $p$. Since $\RE$ is constant along fibres, we define the Euclidean radius at $x \in M$ as the Euclidean radius at $\widetilde x \in \widetilde M$ for any $\widetilde x$ in the fibre of $x$. We say a \textbf{ball is Euclidean} if its radius is less than the Euclidean radius of its centre.

\subsection{Relevant theory of varifolds}
We recall some notation for varifolds that we will use. A 2-varifold $V$ in $U \subset \R^3$ is a Radon measure on $G_2(U) = U \times G(2,3)$, the 2-Grassmannian of $U$. We denote its \textbf{weight} by $|V|$, its \textbf{support} by $\mathrm{supp}(V)$. If $V$ is a stationary varifold, then the monotonicity formula \cite[5.1(1)]{allard1972first} guarantees that the \textbf{density} of $V$ at $x \in \R^3$, $\vartheta_V(x)$, is well defined as the limit with $r \rightarrow 0$ of the mass ratios
\begin{equation*}
    \vartheta_V(x, r) := \frac{|V|(B(x, r))}{\pi r^2}.
\end{equation*}

If $S$ is a 2-rectifiable set and $\theta$ is a positive locally $\Haus^2$-integrable function, we let $v(S, \theta)$ denote the \textbf{varifold induced by $S$ with multiplicity $\theta$}. Such a varifold is called rectifiable. We say a stationary rectifiable varifold $V$ is integral if $\vartheta_V(x) \in \{1, 2, 3, ...\}$ for $|V|$-a.e. $x \in U$. We have the following result concerning the closure of a stationary rectifiable varifold of finite weight and density $\geq 1$. The argument is similar to \cite[Corollary 17.18]{maggi2012sets}.
\begin{proposition} \label{prop: density one stationary rectifiable varifolds are almost closed}
    Suppose that $V = v(S, 1)$ is a rectifiable 2-varifold that is stationary in an open set $U \subset \R^3$. Suppose further that $|V|(U) < \infty$ and the density $\vartheta_V(x) \geq 1$ at all $x \in S$. Then $\overline{S} \cap U = S \cap U$ up to sets of $\Haus^2$-measure zero.
\end{proposition}
\begin{proof}
    Fix $x \in \overline{S} \cap U$ and $r > 0$ such that $B(x, r) \subset U$. Let $(x_k)_{k \in \N}$ be a sequence in $S \cap U$ such that $x_k \rightarrow x$. By passing to a shifted subsequence if necessary, we can assume there is $0 < r_k \leq r$ such that $B(x_k, r_k) \subset B(x, r)$ with $r_k \rightarrow r$. Then, for all $k$,
    \begin{equation*}
        \vartheta_V(x_k, r_k) = \frac{\Haus^2(S \cap B(x_k, r_k))}{\pi r_k^2} \leq \frac{\Haus^2(S \cap B(x, r))}{\pi r^2} \frac{r^2}{r_k^2} = \vartheta_V(x, r) \e(k)
    \end{equation*}
    where $\e(k) = r^2 / r_k^2$. Since $\vartheta_V(x_k) \geq 1$ for each $k$, the monotonicity formula for stationary varifolds \cite[5.1(1)]{allard1972first} implies that $\vartheta_V(x_k, r_k) \geq 1$. Hence $\vartheta_V(x, r) \geq \e(k)^{-1} \rightarrow 1$ and so $\vartheta_V(x)\geq 1$. Hence, by a measure comparison theorem (\cite[Theorem 6.4]{maggi2012sets}), $\Haus^2(S \cap (\overline{S} \cap U)) \geq \Haus^2(\overline{S} \cap U)$. Since $|V|(U) < \infty$, the result follows.
\end{proof}

\subsection{Relevant theory of sets of finite perimeter} \label{subsec: relevant theory of sets of finite perimeter}
Fix a normal covering space $p: \widetilde M \rightarrow M$. We extract some facts and notation from \cite{maggi2012sets} that will be useful in our discussion. Strictly speaking, the results in \cite{maggi2012sets} are proved in a Euclidean ambient space, but with mild modifications these results carry over to our setting since $p: \widetilde M \rightarrow M \subset \R^3$ is a local isometry.

A measurable set $E$ has \textbf{locally finite perimeter} if
\begin{equation*}
    \sup \left\{ \int_E \mathrm{div}T : T \text{ is a } C_c^1 \text{ vector field}, \: \textrm{spt} T \subset K, \: \sup |T| \leq 1 \right\} < \infty
\end{equation*}
for all compact sets $K \subset \widetilde M$. $E$ has finite perimeter if this quantity is bounded independent of $K$. The Riesz representation theorem tells us that for each set of locally finite perimeter $E$ there is a vector-valued Radon measure $\mu_E = g |\mu_E|$, called the \textbf{Gauss-Green measure}, such that
\begin{align*}
    \int_E \mathrm{div}T &= \int_{\widetilde M} T \cdot \diff \mu_E \text{ for all } C_c^1 \text{ vector fields } T; \\
    \mathrm{spt} \mu_E &= \{ \widetilde x \in \widetilde M: 0 < |E \cap B(\widetilde x, r)| < |B(\widetilde x, r)| \text{ for all } r>0 \text{ sufficiently small}\}.
\end{align*}
We define the \textbf{perimeter of $E$ in a set $A$} as $P(E; A) = |\mu_E|(A)$ and note that for $A$ open,
\begin{equation*}
    P(E, A) = \sup\{ \int_E \mathrm{div} T : T \text{ is a } C_c^1 \text{ vector field}, \: \textrm{spt} T \subset A, \: \sup |T| \leq 1 \}.
\end{equation*}
De Giorgi's structure theorem \cite[Theorem 15.9]{maggi2012sets} asserts that
\begin{equation*}
    \mu_E = \nu_E \Haus^2 \resmes \partial^* E
\end{equation*}
where $\partial^* E$, the \textbf{reduced boundary of $E$}, is the $2$-rectifiable set of points $ \widetilde x \in \mathrm{spt} \mu_E$ where $\nu_E$, the \textbf{measure theoretic outward unit normal} given by
\begin{equation*}
    \nu_E(\widetilde x) := \lim_{r \rightarrow 0^+} \frac{\mu_E(B(\widetilde x, r))}{|\mu_E|(B(\widetilde x, r))},
\end{equation*}
exists and is of unit length. We therefore view $\partial^* E$ as the set where the perimeter of $E$ `lives' and note that $\overline{\partial^* E} = \mathrm{spt} \mu_E \subset \partial E$ by De Giorgi's theorem\footnote{More elementary arguments can be used to deduce the first equality - see \cite[Remark 15.3]{maggi2012sets}.} and the above characterisation of $\mathrm{spt} \mu_E$. Furthermore, we can modify $E$ by sets of $\Haus^3$-measure zero such that $\overline{\partial^* E} = \partial E$ without changing $\partial^*E$ \cite[Proposition 12.19]{maggi2012sets}, which we always do.

We define the \textbf{density of $E$ at $\widetilde x$} as
\begin{equation*}
    \Theta_E(\widetilde x) = \lim_{r \rightarrow 0^+} \frac{|E \cap B(\widetilde x, r)|}{|B(\widetilde x, r)|},
\end{equation*}
which is well defined for $\Haus^3$-a.e. $\widetilde x \in \widetilde M$. For $t \in [0,1]$, we define the \textbf{set of points of $E$-density $t$} as
\begin{equation*}
    E^{(t)} := \{ \widetilde x \in \widetilde M: \lim_{r \rightarrow 0^+}|E \cap B(\widetilde x, r)|/|B(\widetilde x, r)| = t\};
\end{equation*}
and the \textbf{essential boundary} of $E$ as $\partial^eE := \widetilde M \setminus (E^{(0)} \cup E^{(1)})$. The blow up properties of the reduced boundary \cite[Corollary 15.8]{maggi2012sets} guarantee that
\begin{equation*}
    \partial^* E \subset E^{(1/2)} \subset \partial^e E
\end{equation*}
and Federer's theorem \cite[Theorem 16.2]{maggi2012sets} asserts that
\begin{equation*}
    \Haus^2(\partial^e E \setminus \partial^*E) = 0
\end{equation*}
and so $\partial^*E \approx E^{(1/2)} \approx \partial^e E$, where $A \approx B$ if $\Haus^{2}(A \triangle B) = 0$ and $\triangle$ denotes the symmetric difference.

We say that a sequence $(E_k)_{k \in \N}$ of measurable sets \textbf{converges locally} to a measurable set $E$ if $|(E_k \triangle E) \cap K| \rightarrow 0$ as $k \rightarrow \infty$ for all compact $K \subset \widetilde M$. $(E_k)_{k \in \N}$ \textbf{converges} to $E$ if $|E_k \triangle E| \rightarrow 0$ as $k \rightarrow \infty$. The perimeter is lower semicontinuous under this convergence, i.e. if $(E_k)$ is a sequence of sets of finite perimeter converging to $E$ locally and such that for each compact $K$
\begin{equation*}
    \limsup_{k \rightarrow \infty} P(E_k, K) < \infty,
\end{equation*}
then $E$ is a set of finite perimeter and for all open $A \subset \widetilde M$
\begin{equation*}
    P(E, A) \leq \liminf_{k \rightarrow \infty}P(E_k, A).
\end{equation*}

Recall that a set of finite perimeter naturally defines a function of bounded variation via $E \mapsto \one_E$, where $\one_E$ denotes the characteristic function of $E$. This embedding provides the following compactness theorem for sets of finite perimeter.

\begin{proposition}
    Let $(E_k)$ be a sequence of sets of locally finite perimeter such that
    \begin{equation*}
        \limsup_{k \rightarrow \infty}P(E_k, K) < \infty
    \end{equation*}
    for all compact $K$. Then there exists a subsequence $k'$ and a set of locally finite perimeter $E$ such that $E_{k'} \rightarrow E$ locally and $\mu_{E_{k'}} \xrightarrow{*} \mu_E$ in the weak* topology on Radon measures.
\end{proposition}

It will be of use to us to know how sets of finite perimeter get mapped by diffeomorphisms. See \cite[Proposition 17.1]{maggi2012sets} for a proof of the following result.

\begin{proposition} \label{prop: mapping set of finite perimeter by diffeomorphisms}
    Let $f: \widetilde M \rightarrow \widetilde M$ be a diffeomorphism and let $E$ be a set of locally finite perimeter. Then $f(E)$ is a set of locally finite perimeter and $f(\partial^*E) \approx \partial^* f(E)$. If $f$ is an isometry, then $f(\partial^* E) = \partial^* f(E)$.
\end{proposition}

\section{Fundamental domains and \texorpdfstring{$H$}{H}-domains} \label{sec: fundamental domains and H-domains}
\subsection{Definitions, existence, and correspondence} \label{subsec: FD and HD definitions existence and correspondence}
Let $p: \widetilde M \rightarrow M$ be a normal covering space of $M$. 
\begin{definition} \label{def: fundamental domains}
    We say that a subset $D$ of $\widetilde M$ is a \textbf{fundamental domain} in $\widetilde M$ if
    \begin{align}
        |D \cap gD| &= 0 \text{ for all } g \in \deck(p) \setminus \{\id\} \text{, and} \label{eq: fundamental domain 1} \\
        \left|\widetilde M \setminus \bigcup_{g \in \deck(p)} gD \right| &= 0. \label{eq: fundamental domain 2}
    \end{align}
\end{definition}

This says that $D$ and its translations under $\deck(p)$ tessellate (in a measure theoretic sense) the cover $\widetilde M$. Furthermore, \eqref{eq: fundamental domain 2} tells us that $p|_D$ is surjective up to $\Haus^3$-measure zero. Together with $p: \widetilde M \rightarrow M$ being normal, \eqref{eq: fundamental domain 1} tells us roughly that $p|_D$ is injective up to $\Haus^3$-measure zero. So morally, $D$ is a copy (in a measure theoretic sense) of $M$ in $\widetilde M$.

A natural way to construct a fundamental domain is to lift the complement of a closed surface in $M$ as follows. For a closed set $\Sigma \subset M$, we define the \textbf{avoiding group} of $\Sigma$ at $x_0 \in M$, $A(\Sigma, x_0)$, as the group of loops $[\alpha] \in \pi_1(M, x_0)$ such that there is $\alpha' \in [\alpha]$ with $\alpha' \cap \Sigma = \emptyset$. This depends on the basepoint when $M \setminus \Sigma$ is disconnected. When $M \setminus \Sigma$ is connected, we sometimes suppress the basepoint from our notation if this causes no ambiguity.

\begin{proposition} \label{prop: lifting thick surface to fundamental domain}
    Let $\Sigma$ be a closed set in $M$ with $\Haus^3(\Sigma) = 0$ and suppose $H$ is a normal subgroup of $\pi_1(M)$ such that $A(\Sigma, x_0) \leq H$ for all $x_0$. Then there is an open fundamental domain $D(\Sigma, H)$ in the cover $p: \widetilde M(H) \rightarrow M$ such that $p: D(\Sigma, H) \rightarrow M \setminus \Sigma$ is an isometry and $p(\partial D) \subset \Sigma$.
\end{proposition}
\begin{proof}
    Assume first that $M \setminus \Sigma$ is path connected and fix $x_0 \in M$ and $\widetilde x_0$ in the fibre of $x_0$. Since $M \setminus \Sigma$ is path-connected, we find, for any $x \in M$, a path $\alpha_x: [0,1] \rightarrow M$ from $x_0$ to $x$ that does not intersect $\Sigma$. We claim that the map
    \begin{equation*}
        \Phi: M \setminus \Sigma \rightarrow \widetilde M(H), \quad \Phi(x) = \lift(\alpha_x, \widetilde x_0)(1)
    \end{equation*}
    is a well defined isometry onto its open (in $\widetilde M(H)$) image $D(\Sigma, H) := \Phi(M \setminus \Sigma)$ with $\Phi^{-1} = p|_D$. Set $D := D(\Sigma, H)$ to simplify notation here.

    To see that $\Phi$ is well defined we need to check that $\Phi(x)$ is independent of the choice of path from $x_0$ to $x$ in $M \setminus \Sigma$. Let $\alpha_x$ and $\alpha_x'$ be two such paths. Then $\beta = \alpha_x \cdot \overline{\alpha_x'}$, which is a closed loop based at $x_0$, does not meet $\Sigma$. Hence $[\beta] \in A(\Sigma, x_0) \leq H$ and so $\lift(\beta, \widetilde x_0)$ is closed. Since $\lift(\beta, \widetilde x_0) = \lift(\alpha_x, \widetilde x_0) \cdot \overline{\lift(\alpha_x', \widetilde x_0)}$, we conclude that $\lift(\alpha_x, \widetilde x_0)(1) = \lift(\alpha_x', \widetilde x_0)(1)$ and so $\Phi$ does not depend on the choice of path.

    It is straightforward to check that $D$ is open and that $\Phi$ is bijective with $\Phi^{-1} = p|_D$. Hence, since $p$ is a local isometry, $\Phi$ is an isometry.

    We now check that $D$ is indeed a fundamental domain. To see \eqref{eq: fundamental domain 1}, suppose that for some $g \in \deck(p)$ we can find $\widetilde x \in D \cap gD$. So there is $\widetilde y \in D$ with $g \widetilde y = \widetilde x \in D$. But $y := p(\widetilde y) = p(\widetilde x)$ and so $\widetilde y = \Phi(y) = \widetilde x$. Hence, $g = \id$. Therefore, $D \cap gD = \emptyset$ for $g \neq \id$ which implies \eqref{eq: fundamental domain 1}. To see \eqref{eq: fundamental domain 2}, suppose that
    \begin{equation*}
        \left| \widetilde M(H) \setminus \bigcup_{g \in \deck(p)} gD \right| >0
    \end{equation*}
    which, since $p$ is a local isometry, then implies that
    \begin{equation*}
        \left| p \left(\widetilde M(H) \setminus \bigcup_{g \in \deck(p)} gD \right) \right| >0.
    \end{equation*}
    Since we also have that
    \begin{equation*}
        p \left(\widetilde M(H) \setminus \bigcup_{g \in \deck(p)} gD \right) = \Sigma
    \end{equation*}
    we get that $\Sigma$ has positive $\Haus^3$-measure which is a contradiction and so we have verified \eqref{eq: fundamental domain 2}.

    It remains to check that $p(\partial D) \subset \Sigma$. We saw above that $D \cap gD = \emptyset$ for $g \in \deck(p) \setminus \{\id\}$. It follows that $\overline{D} \subset \widetilde M(H) \setminus \bigcup_{g \neq \id} gD$. Furthermore, since $D$ is open, $\partial D \cap D = \emptyset$ and so $\partial D \subset \widetilde M \setminus D$. It follows that $\partial D \subset \widetilde M(H) \setminus \bigcup_{g \in \deck(p)} gD$ and so
    \begin{equation*}
        p(\partial D) \subset p\left( \widetilde M(H) \setminus \bigcup_{g \in \deck(p)} gD \right) = \Sigma
    \end{equation*}
    as required.

    The case where $M \setminus \Sigma$ is not path connected is handled analogously with the exception that one defines $\Phi$ on each connected component of $M \setminus \Sigma$.
\end{proof}

As a quick consequence, we get that there is an open fundamental domain in every normal covering space of $M$.
\begin{corollary} \label{cor: existence of fundamental domains in arbitrary covers}
    Let $\Sigma_0$ be the cone over $\Gamma$ emanating from a point $x_0 \notin \Gamma$. $A(\Sigma_0) = \{e\}$ and so we find a corresponding fundamental domain $D(\Sigma_0, H)$ for each normal subgroup $H$.
\end{corollary}
\begin{proof}
    Take some $[\alpha] \in \pi_1(M)$ with $\alpha \cap \Sigma_0 = \emptyset$. For some large $R$, let $\alpha_R(t) = (1-R) x_0 + R \alpha(t)$ which is clearly homotopic to $\alpha$. For $R$ sufficiently large, $\Gamma \subset B(x_0, R/2)$ and so $\alpha_R$ is contractible. Therefore, $[\alpha]$ is trivial which is what was required.
\end{proof}

It will be useful to have a unified framework to treat fundamental domains of any cover as subsets of a single, fixed space. This is provided by $H$-domains which are subsets of the universal cover. To define these, recall that the group of deck transformations of the universal cover is $\pi_1(M)$.
\begin{definition} \label{def: H domains}
    Let $H$ be a normal subgroup of $\pi_1(M)$. We say a subset $F$ of the universal cover $\widetilde M(e)$ is an \textbf{$H$-domain} if
    \begin{align}
        |F \cap g F| &= 0 \text{ for all } g \notin H, \label{eq: HD 1} \\
        |F \triangle h F| &= 0 \text{ for all } h \in H, \text{ and} \label{eq: HD 2} \\
        \left|\widetilde M(e) \setminus \bigcup_{g \in \pi_1(M)} gF \right | &= 0. \label{eq: HD 3}
    \end{align}
    Note that, using \eqref{eq: HD 2}, we can equivalently rewrite \eqref{eq: HD 3} as
    \begin{equation} \label{eq: HD 4}
        \left| \widetilde M(e) \setminus \bigcup_{gH \in \pi_1(M)/H} gF \right| = 0.
    \end{equation}
\end{definition}
One should think informally of $\pi_1(M)$ decomposing `orthogonally' into $H$ and $\pi_1(M)/H$. Then, $H$-domains behave like fundamental domains in the $\pi_1(M)/H$ direction (i.e. \eqref{eq: HD 1}, \eqref{eq: HD 4}) and are invariant in the $H$ direction (i.e. \eqref{eq: HD 2}). If one loosely thinks of deck transformations as representing `translation direction' in $\widetilde M(e)$, then one can think of $H$ and $\pi_1(M)/H$ as `orthogonal sets of directions' in $\widetilde M(e)$.

To motivate this definition, we show, in \cref{prop: lifting fundamental domains to H-domains}, that a fundamental domain in $p_H: \widetilde M(H) \rightarrow M$ naturally defines an $H$-domain by lifting, and that an $H$-domain naturally defines a fundamental domain in $p_H: \widetilde M(H) \rightarrow M$ by projection. The natural map to lift a fundamental domain in $\widetilde M(H)$ to $\widetilde M(e)$ is the covering map $f_H: \widetilde M(e) \rightarrow \widetilde M(H)$ obtained by lifting the universal covering map $q: \widetilde M(e) \rightarrow M$ to $\widetilde M(H)$. Diagrammatically,
\begin{equation*}
    \begin{tikzcd}
        \widetilde M(e) \arrow[dr, "q"] \arrow[r, "f_H"] & \widetilde M(H) \arrow[d, "p_H"] \\
        & M
    \end{tikzcd}
\end{equation*}
\begin{proposition} \label{prop: lifting fundamental domains to H-domains}
    If $D$ is a fundamental domain in a normal cover $p_H: \widetilde M(H) \rightarrow M$, then $F := f_H^{-1}(D)$ is an $H$-domain. If $F$ is an $H$-domain, then $f_H(F)$ is a fundamental domain in $p_H: \widetilde M(H) \rightarrow M$.
\end{proposition}

\begin{proof}   
    We start with the first statement. To see \eqref{eq: HD 1}, let $g \in \pi_1(M)$ and suppose $|F \cap gF| > 0$. We have that
    \begin{align*}
        f_H(F \cap gF) = f_H(f_H^{-1}(D) \cap f_H^{-1}(g'D)) = D \cap g'D,
    \end{align*}
    where $g'$ is the deck transformation of $p_H: \widetilde M(H) \rightarrow M$ such that $f_H \circ g = g' \circ f_H$. Hence, $|D \cap g'D| >0$ which implies $g' = \id$ since $D$ is a fundamental domain. Therefore, $g$ is a deck transformation of $f_H: \widetilde M(e) \rightarrow \widetilde M(H)$, i.e. $g \in H$. To see \eqref{eq: HD 2}, let $g \in \pi_1(M)$ and suppose that $|F \triangle gF| >0$. We have
    \begin{align*}
        f_H(F \triangle gF) = f_H(f_H^{-1}(D) \triangle f_H^{-1}(g'D)) = D \triangle g'D,
    \end{align*}
    where again $g'$ is the deck transformation of $p_H: \widetilde M(H) \rightarrow M$ such that $f_H \circ g = g' \circ f_H$. Hence, $|D \triangle g'D| > 0$ which implies $g' \neq \id$ since $D$ is a fundamental domain. Therefore, $g$ is not a deck transformation, i.e. $g \notin H$. To see \eqref{eq: HD 3}, we compute
    \begin{align*}
        f_H \left(\widetilde M(e) \setminus \bigcup_{g \in \pi_1(M)} gF \right) = \widetilde M(H) \setminus \bigcup_{g' \in \deck(p_H)} g'D
    \end{align*}
    and hence, we contradict that $D$ is a fundamental domain unless $F$ verifies \eqref{eq: HD 3}.

    We now turn to the second statement. To see \eqref{eq: fundamental domain 1}, let $g' \in \deck(p)$ and suppose $|D \cap g'D|> 0$. Up to a set of $\Haus^3$-measure zero, we have that
    \begin{align*}
        f_H^{-1}(D \cap g'D) = F \cap gF,
    \end{align*}
    where $g$ is a deck transformation of $q: \widetilde M(e) \rightarrow M$ such that $f_H \circ g = g' \circ f_H$. Hence, $|F \cap gF| >0$ which implies that $g \in H$, i.e. $g$ is a deck transformation of $f_H: \widetilde M(e) \rightarrow \widetilde M(H)$. Therefore, $g' = \id$ as required. To see \eqref{eq: fundamental domain 2}, we compute
    \begin{align*}
        f_H^{-1}\left(\widetilde M(H) \setminus \bigcup_{g' \in \deck(p_H)} g'D \right) = \widetilde M(e) \setminus \bigcup_{gH \in \pi_1(M)/H} gF
    \end{align*}
    and hence, we contradict that $F$ is an $H$-domain unless $D$ verifies \eqref{eq: fundamental domain 2}.
\end{proof}

\subsection{Perimeter of fundamental domains}
The fundamental domains $D(\Sigma, H) \subset \widetilde M(H)$, constructed in \cref{prop: lifting thick surface to fundamental domain} from surfaces $\Sigma \subset M$ with $A(\Sigma) \leq H$, have (at least locally) finite perimeter when $\Haus^2(\Sigma)$ is finite. We prove this in \cref{prop: lifting thin surface to fundamental domain of finite perimeter}. We start by proving that the perimeter of a fundamental domain $D$ is twice the area of the projected boundary $p(\partial^* D)$. The proof follows \cite[Lemma 1.4]{congedo1991multidimensionalsegmentation}.

\begin{proposition} \label{prop: area of projection of reduced boundary}
    Let $D$ be a fundamental domain of finite perimeter in $p: \widetilde M \rightarrow M$. Let $\widetilde B, B$ be Euclidean balls in $\widetilde M, M$ respectively such that $p: \widetilde B \rightarrow B$ is an isometry and define $\I := \cup_{g \in \deck(p)} \partial^* gD$. We have 
    \begin{align}
        p( \I \cap \widetilde B) &= p(\partial^*D) \cap B \label{eq: area of projection of reduced boundary 1}\\
        P \left(D, \bigcup_{g \in \deck(p)} g \widetilde B \right) &= 2 \Haus^2(\I \cap \widetilde B) \label{eq: area of projection of reduced boundary 2}.
    \end{align}
    Via a covering argument, \eqref{eq: area of projection of reduced boundary 1} and \eqref{eq: area of projection of reduced boundary 2} imply $P(D) = 2 \Haus^2(p(\partial^*D))$.
\end{proposition}
\begin{proof}
    \textit{Step 1:} We prove \eqref{eq: area of projection of reduced boundary 1}. Since $p|_{\widetilde B}$ is an isometry, we have
    \begin{align*}
        p(\I \cap \widetilde B) = p \left( \bigcup_{g \in \deck(p)} \partial^* gD \cap \widetilde B \right) = p \left( \bigcup_{g \in \deck(p)} g \left(\partial^*D \cap g^{-1}\widetilde B \right) \right) &= \bigcup_{g \in \deck(p)} p(\partial^* D \cap g^{-1} \widetilde B) \\
        &= p\left( \partial^*D \cap \bigcup_{g \in \deck(p)} g^{-1} \widetilde B \right) \\
        &=p(\partial^* D \cap p^{-1}(B)) \\
        &= p(\partial^*D) \cap B
    \end{align*}
    where we have used \cref{prop: mapping set of finite perimeter by diffeomorphisms} in the second equality of the first line.

    \textit{Step 2:} We start the proof of \eqref{eq: area of projection of reduced boundary 2} by proving that
    \begin{equation} \label{eq: area of projection of reduced boundary 3}
        \partial^* gD \cap \widetilde B \approx \bigcup_{h \in \deck(p) \setminus \{g\}} \partial^* hD \cap \partial^* gD \cap \widetilde B,
    \end{equation}
    where $A \approx A'$ if $\Haus^2(A \triangle A') = 0$ (see \cref{subsec: relevant theory of sets of finite perimeter}). We prove this by induction. Fix some $g \in \deck(p)$ and let $\{g_i\}_{i \in \N}$ be an enumeration of $\deck(p)$ with $g_0 = g$. Let $G_{\leq N} := \{g_i : i \leq N \} \subset \deck(p)$. We introduce the notation
    \begin{align*}
        &\I(g, h, \widetilde B) = \partial^*gD \cap \partial^*hD \cap \widetilde B
    \end{align*}
    and note that $\partial^*gD \cap \widetilde B = \I(g,g, \widetilde B)$. We claim that for all $N \in \N$,
    \begin{equation} \label{eq: area of projection of reduced boundary 4}
        \I(g,g, \widetilde B) \setminus \bigcup_{h \in G_{\leq N} \setminus \{g\}} \I(g, h, \widetilde B) = \left( \bigcap_{h \in G_{\leq N} \setminus \{g\}} hD^{(0)} \right) \cap \I(g,g, \widetilde B).
    \end{equation}

    For $G_{\leq 0}$, this holds trivially. Suppose that this holds for some $N$. We have the following decomposition of the RHS of \eqref{eq: area of projection of reduced boundary 4}
    \begin{align*}
        \left( \bigcap_{h \in G_{\leq N} \setminus \{g\}} hD^{(0)} \right) \cap \I(g,g, \widetilde B) &= g_{N+1}D^{(0)} \cap\left( \bigcap_{h \in G_{\leq N} \setminus \{g\}} hD^{(0)} \right) \cap \I(g,g, \widetilde B) \\
        &\quad \cup \partial^e g_{N+1}D \cap \left( \bigcap_{h \in G_{\leq N} \setminus \{g\}} hD^{(0)} \right) \cap \I(g,g, \widetilde B) \\
        &\quad \cup g_{N+1} D^{(1)} \cap \left( \bigcap_{h \in G_{\leq N} \setminus \{g\}} hD^{(0)} \right) \cap \I(g,g, \widetilde B) \\
        &\approx \left( \bigcap_{h \in G_{\leq N+1} \setminus \{g\}} hD^{(0)} \right) \cap \I(g,g, \widetilde B) \\
        &\quad \cup \I(g, g_{N+1}, \widetilde B) \\
        &\quad \cup \emptyset.
    \end{align*}
    where we made use of Federer's theorem. Hence, \eqref{eq: area of projection of reduced boundary 4} holds with $N$ replaced with $N+1$. By induction, \eqref{eq: area of projection of reduced boundary 4} holds for all $N \in \N$. We claim that \eqref{eq: area of projection of reduced boundary 4} holds with $G_{\leq N} $ replaced with $\deck(p)$. Indeed, since $\deck(p) = \cup_{N \in \N} G_{\leq N}$ we have that
    \begin{align*}
        \I(g,g, \widetilde B) \setminus \bigcup_{h \in \deck(p) \setminus \{g\}} \I(g, h, \widetilde B) &= \I(g,g, \widetilde B) \setminus \bigcup_{N \in \N} \bigcup_{h \in G_{\leq N} \setminus\{g\}} \I(g, h, \widetilde B) \\
        &\approx \bigcap_{N \in \N} \left(\bigcap_{h \in G_{\leq N} \setminus \{g\}} hD^{(0)} \right) \cap \I(g,g, \widetilde B) \\
        &= \left( \bigcap_{h \in \deck(p) \setminus \{g\}} hD^{(0)} \right) \cap \I(g,g, \widetilde B)
    \end{align*}
    as claimed. To conclude the proof of \eqref{eq: area of projection of reduced boundary 3}, it remains to show that the set on the RHS above has $\Haus^2$-measure zero. Let $G_{\geq N} : = \{g_i \in \deck(p): i \geq N\}$. Since $D$ is a fundamental domain, observe that up to sets of $\Haus^2$-measure zero
    \begin{align*}
        \bigcap_{h \in \deck(p) \setminus \{g\}} hD^{(0)} \cap \I(g,g, \widetilde B) \subset \left\{ \Theta_{\bigcup_{k \in G_{\geq N}}kD} = 1/2 \right\} \approx  \partial^* \left( \bigcup_{k \in G_{\geq N} \setminus \{g\}} kD \right) \cap \widetilde B
    \end{align*}
    for any $N \geq 1$, and so
    \begin{align*}
        \Haus^2 \left( \bigcap_{h\ \in \deck(p) \setminus \{g\}} hD^{(0)} \cap \I(g, g, \widetilde B) \right) \leq P \left( \bigcup_{k \in G_{\geq N} \setminus \{g\} } kD ; \widetilde B \right) &\leq \sum_{k \in G_{\geq N} \setminus \{g\}} P(kD; \widetilde B) \\
        &= \sum_{k \in G_{\geq N} \setminus \{g\}}P(D; k^{-1} \widetilde B) \\
        &\rightarrow 0 \text{ as } N \rightarrow \infty
    \end{align*}
    where the convergence holds as $D$ has finite perimeter. This concludes the proof of \eqref{eq: area of projection of reduced boundary 3}.

    \textit{Step 3:} We conclude the proof of \eqref{eq: area of projection of reduced boundary 2}. Using \textit{step 2} we have that
    \begin{align*}
        \sum_{g \in \deck(p)} P(D, g \widetilde B) = \sum_{g \in \deck(p)} \Haus^2(\partial^* D \cap g \widetilde B) &= \sum_{g \in \deck(p)} \Haus^2(\partial^* g^{-1}D \cap \widetilde B) \\
        \text{(Step 2)} \qquad &= \sum_{g \in \deck(p)} \Haus^2 \left( \bigcup_{ h \in \deck(p) \setminus \{g\}} \partial^* hD \cap \partial^* gD \cap \widetilde B \right) \\
        & = 2 \Haus^2 \left( \bigcup_{\substack{g, h \in \deck(p) \\ g\neq h}} \partial^*gD \cap \partial^* hD \cap \widetilde B \right) \\
        \text{(Step 2)} \qquad &= 2 \Haus^2( \I \cap \widetilde B)
    \end{align*}
    where the additivity in the second to last line is justified since all points in the reduced boundary have density $1/2$: for $g \neq h$ and $g' \neq h'$, $\partial^* gD \cap \partial^* hD$ is disjoint from $\partial^* g'D \cap \partial^* h'D$ except when $g=g'$, $h = h'$ or $g=h'$, $h = g'$.

    \textit{Step 4:} It remains to do the covering argument. Let $\{x_i\}_{i \in \N}$ be a list of all rational points in $M$. Then $\{B(x_i, \RE(x_i))\}_{i \in \N}$ is a cover of $M$. Let $C_i = B(x_i, \RE(x_i)) \setminus \cup_{j < i} B(x_j, \RE(x_j))$ and observe that the cover $\{C_i\}_{i \in \N}$ of $M$ is pairwise disjoint and $C_i \subset B(x_i, \RE(x_i))$. By this last observation, $p^{-1}(C_i) = \bigsqcup_{g \in \deck(p)} g \widetilde C_i$ where $\widetilde C_i$ is isometric to $C_i$ via $p$. It should be clear that \eqref{eq: area of projection of reduced boundary 1}, \eqref{eq: area of projection of reduced boundary 2} still hold with $B(x, r)$ and $B(\widetilde x, r)$ replaced with $C_i$ and $\widetilde C_i$. Therefore, we have that
    \begin{equation*}
        \Haus^2(p(\partial^*D)) = \sum_{i \in \N} \Haus^2(p(\partial^*D) \cap C_i) = (1/2) \sum_{i \in \N} \sum_{g \in \deck(p)} P(D; g \widetilde C_i) = (1/2) P(D)
    \end{equation*}
    where the last equality holds because $\cup_i \cup_{g \in \deck(p)} g \widetilde C_i = \widetilde M$ and the union is disjoint.
\end{proof}

We are now ready to prove the relationship between the area of $\Sigma$ and the perimeter of the corresponding fundamental domain $D(\Sigma, H)$ (see \cref{prop: lifting thick surface to fundamental domain}).
\begin{proposition} \label{prop: lifting thin surface to fundamental domain of finite perimeter}
    Let $\Sigma$ be closed with $\Haus^2(\Sigma) < \infty$ and $A(\Sigma, x_0) \leq H$ for all $x_0 \in M$ and some normal $H \leq \pi_1(M)$. Then the fundamental domain $D := D(\Sigma, H)$ of \cref{prop: lifting thick surface to fundamental domain} has locally finite perimeter and $\Haus^2(p(\partial^*D)) \leq \Haus^2(\Sigma)$. Furthermore, it has finite perimeter at least in the following cases:
    \begin{enumerate}
        \item[(i)] If $p: \widetilde M(H) \rightarrow M$ has finite degree, then $D$ has finite perimeter and $P(D) = 2 \Haus^2(p(\partial^*D))$.
        \item[(ii)] If there is $N$ such that for every $x \in M$ there is a radius $r$ such that $B(x, r) \subset M$ and $B(x, r) \setminus \Sigma$ has at most $N$ connected components, then $D$ has finite perimeter and $P(D) = 2 \Haus^2(p(\partial^*D))$.
    \end{enumerate}
\end{proposition}
\begin{proof}
    We start by showing that $D$ has locally finite perimeter. Let $\{x_i\}_{i \in \N}$ be a list of rational points of $M$, let $B_i$ be the ball centred at $x_i$ of Euclidean radius, and let $\widetilde B_i$ be a lift of this ball in $\widetilde M$. Since $p|_{\widetilde B_i}$ is an isometry and using \cref{prop: lifting thick surface to fundamental domain} we have
    \begin{equation*}
        \Haus^2(\partial D \cap g \widetilde B_i) = \Haus^2(p(\partial D \cap g \widetilde B_i)) \leq \Haus^2( p(\partial D) \cap B_i) \leq \Haus^2(\Sigma)
    \end{equation*}
    for any $g \in \deck(p)$. Let $K$ be a compact subset of $\widetilde M$. By compactness, we find $i_1,..., i_N$ and $g_1, ..., g_N$ such that $K \subset \cup_{1 \leq j \leq N} g_j \widetilde B_{i_j}$ and so
    \begin{equation*}
        \Haus^2 (\partial D \cap K) \leq \sum_{1 \leq j \leq N} \Haus^2(\partial D \cap g_j \widetilde B_{i_j}) \leq N \Haus^2(\Sigma) < \infty
    \end{equation*}
    from which it follows that $D$ has locally finite perimeter (see \cite[Proposition 3.62]{ambrosio2000variation}). Using that $\partial^* D \subset \partial D$ and \cref{prop: lifting thick surface to fundamental domain}, we have that
    \begin{equation} \label{eq: lifting thin surface to fundamental domain of finite perimeter 1}
        p(\partial^* D) \subset \Sigma 
    \end{equation}
    from which we conclude that $\Haus^2(p(\partial^* D)) \leq \Haus^2(\Sigma)$ as required.
    
    We now prove (i). Assume that $p: \widetilde M(H) \rightarrow M$ has finite degree $s$. Recall from \textit{step 4} of the proof of \cref{prop: area of projection of reduced boundary}, that there is a pairwise disjoint countable cover $\{ g \widetilde C_i\}_{i \in \N, g \in \deck(p)}$ of $\widetilde M$ such that $\widetilde C_i$ is isometric, via $p$, to $C_i = p(\widetilde C_i)$ and $\{C_i\}_{i \in \N}$ is a pairwise disjoint countable cover of $M$. Hence,
    \begin{align*}
        P(D) = \sum_{i \in \N} \sum_{g \in \deck(p)} \Haus^2(\partial^* D \cap g \widetilde C_i) &= \sum_{i \in \N} \sum_{g \in \deck(p)} \Haus^2(p(\partial^* D \cap g \widetilde C_i)) \\
        &\leq \sum_{i \in \N} s \Haus^2(\Sigma \cap C_i) = s \Haus^2(\Sigma) < \infty,
    \end{align*}
    where we used \eqref{eq: lifting thin surface to fundamental domain of finite perimeter 1}. \cref{prop: area of projection of reduced boundary} then gives that $P(D) = 2 \Haus^2(p(\partial^* D))$.

    Finally, we prove (ii). Since $M$ is Lindelöf, we can find countably many $x_i$ in $M$ with corresponding radii $r_i$ coming from the assumptions in (ii), such that $\cup_iB(x_i, r_i)$ covers $M$. We can also ensure that $r_i \leq \RE(x_i)$. We now define $C_i = B(x_i, r_i) \setminus \cup_{j < i} B(x_j, r_j)$ and note that, as before, $\{C_i\}_i$ is a pairwise disjoint cover of $M$. Let $\widetilde C_i$ (resp. $\widetilde B_i$) be an isometric lift of $C_i$ (resp. $B_i$) .
    
    Fix $i$. Observe that the hypothesis of (ii) implies that there is $N' \leq N$ and a collection $g_1, ..., g_{N'} \in \deck(p)$ such that $|D \cap g \widetilde B_i| = 0$ for $g \notin \{g_1, ..., g_{N'} \}$. It follows that if $P(D, g \widetilde C_i) >0$, then $g \in \{g_1, ..., g_N\}$. Indeed, $P(D, g \widetilde C_i) > 0$ implies directly that $P(D, g \widetilde B_i) > 0$. Then, because $\widetilde B_i$ is open, we get that $P(D \cap g \widetilde B_i) > 0$, which in turn implies that $|D \cap g \widetilde B_i| > 0$, giving the result. In addition, we always have the upper bound
    \begin{align*}
        P(D, g \widetilde C_i) = \Haus^2(\partial^* D \cap g \widetilde C_i) = \Haus^2(p(\partial^* D \cap g \widetilde C_i)) \leq \Haus^2(p(\partial^* D) \cap C_i) \leq \Haus^2(\Sigma \cap C_i),
    \end{align*}
    where we used \eqref{eq: lifting thin surface to fundamental domain of finite perimeter 1} in the last inequality. Combining these two facts, we get that
    \begin{equation*}
        \sum_{g \in \deck(p)} P(D, g \widetilde C_i) = \sum_{1 \leq j \leq N'} P(D, g_j \widetilde C_i) \leq N' \Haus^2(\Sigma \cap C_i) \leq N \Haus^2(\Sigma \cap C_i),
    \end{equation*}
    from which we conclude that $D$ has finite perimeter:
    \begin{align*}
        P(D) = \sum_i \sum_{g \in \deck(p)} P(D, g \widetilde C_i) \leq \sum_i N \Haus^2(\Sigma \cap C_i) = N \Haus^2(\Sigma) < \infty. 
    \end{align*}
    \cref{prop: area of projection of reduced boundary} then gives that $P(D) = 2 \Haus^2(p(\partial^* D))$.
\end{proof}

We get as a corollary that fundamental domains of finite perimeter exist in all normal covers.
\begin{corollary} \label{cor: existence of fundamental domains of finite perimeter in arbitrary covers}
    For each normal subgroup $H \leq \pi_1(M)$ and a generic choice of $x_0 \in M$, the fundamental domain $D(\Sigma_0, H)$ in $\widetilde M(H)$ of \cref{cor: existence of fundamental domains in arbitrary covers} has finite perimeter.
\end{corollary}
\begin{proof}
    Since $\Gamma$ consists of $C^1$ closed curves, it is easy to verify that $\Haus^2(\Sigma_0) < \infty$. Furthermore, by application of Sard's theorem, $\Sigma_0$ satisfies the hypothesis of \cref{prop: lifting thin surface to fundamental domain of finite perimeter}(ii) for a generic choice of basepoint $x_0 \in M$. Then, applying \cref{prop: lifting thin surface to fundamental domain of finite perimeter}(ii), we get the result.
\end{proof}

Having established that fundamental domains of finite perimeter always exist, we now show that their perimeter is bounded from below away from zero. We will need the following notation: given a curve $\alpha$ and a number $\delta>0$, we define the $\delta$-tubular neighbourhood of $\alpha$, denoted $T(\alpha, \delta)$, as the set of points that may be written in the form $\alpha(t)+ a N(t) + bB(t)$, where $(\alpha', N, B)$ is the Frenet frame along $\alpha$ and $a^2 + b^2 \leq \delta^2$. We say $T(\alpha, \delta)$ is regular if every point in $T(\alpha, \delta)$ has a unique expression of the form $\alpha(t) + a N(t) + b B(t)$.
\begin{proposition} \label{prop: lower bound on perimeter}
    Suppose that $p: \widetilde M(H) \rightarrow M$ is a non-trivial normal covering space of $M$. There is $\eta > 0$ such that for any fundamental domain $D$ of finite perimeter in $\widetilde M$, we have $P(D) \geq \eta$.
\end{proposition}
\begin{proof}
    Since $p: \widetilde M(H) \rightarrow M$ is non-trivial, we can find a circular loop $\xi: [0,1] \rightarrow M$ such that $[\xi] \notin H$ (see \cref{subsec: knot group}). For future reference, let $z$ be the basepoint of this loop and let $r > 0$ be its radius. We aim to find a lower bound for the perimeter of $D$ in $p^{-1}(T(\xi, \delta))$ for some $\delta$ depending only on $\Gamma$. This will conclude the proof.

    Let $N \geq 2$ be the smallest integer such that $[\xi]^N \in H$ if it exists, and set $N = \infty$ if such an integer does not exist. For $N$ finite, we define $\xi^N: [0,N] \rightarrow M$ as $\xi$ concatenated with itself $N$ times\footnote{Note that here we adopt the convention that the domain of concatenated loops is not renormalised to $[0,1]$, e.g. $\xi \cdot \xi: [0,2] \rightarrow M$.}
    and for $N = \infty$ we define $\xi^N: \R \rightarrow M$ as $\xi^{N}(x) = \xi(x - \lfloor x \rfloor)$ where $\lfloor x \rfloor$ denotes the largest integer $\leq x$.

    \textit{Step 1:} We show that there is a subset $\{\widetilde z_i\}_{i \in I} \subset p^{-1}(z)$ such that
    \begin{equation*}
        p^{-1}(\xi) = \bigsqcup_i \lift(\xi^N, \widetilde z_i) = \bigsqcup_i \widetilde \xi^{N, i}
    \end{equation*}
    where we define $\widetilde \xi^{N, i} := \lift(\xi^N, \widetilde z_i)$. Since $[\xi] \notin H$, $[\xi]$ corresponds to a non-trivial element $g_{\xi} \in \deck(p)$ via the isomorphism $\pi_1(M)/H \rightarrow \deck(p)$ described in \cref{subsec: relevant covering space theory}. Let $G$ denote the cyclic subgroup of order $N$ generated by $g_\xi$. $\deck(p)$ decomposes into the disjoint union of the right cosets of $G$, $\{Gg_i\}_{i \in I}$. Let $\widetilde z \in p^{-1}(z)$ and let $\widetilde z_i := g_i \widetilde z$. We have that
    \begin{align*}
        p^{-1}(\xi) = \bigcup_{g \in \deck(p)} \lift(\xi, g \widetilde z) = \bigsqcup_{i \in I} \bigcup_{h \in G} \lift(\xi, h g_i \widetilde z) = \bigsqcup_{i \in I} \bigcup_{h \in G} \lift(\xi, h \widetilde z_i) = \bigsqcup_{i \in I} \lift(\xi^N, \widetilde z_i)
    \end{align*}
    as required.

    \textit{Step 2:} We define $\delta$ and decompose $p^{-1}(T(\xi, \delta))$ into convenient (nearly) non-overlapping subsets. Define $\alpha_j^N = \xi^N |_{[t_{j-1}^N, t_j^N]}$ and $\widetilde \alpha_j^{N, i} = \widetilde \xi^{N, i}|_{[t_{j-1}^N, t_j^N]}$ where $t_j^N = \kappa_N j$ and 
    \begin{equation*}
        \kappa_N := \begin{cases}
            N/(N+1) & \text{if $N$ is finite} \\
            3/4 & \text{if $N$ is infinite}
        \end{cases}
    \end{equation*}
    and note that $p(\widetilde \alpha_j^{N, i}) = \alpha_j^N$. Also define $J_N = \{1 \leq j \leq N+1 \}$ when $N$ is finite and $J_N = \Z$ when $N = \infty$. Fix $\delta > 0$ small enough such that $T( \xi, \delta)$ is a \textit{regular} $\delta$-tubular neighbourhood of $\xi$. Set $T_j^N := T(\alpha_j^N, \delta)$, $\widetilde T_j^{N, i} := T(\widetilde \alpha_j^{N, i}, \delta)$, and note that $\widetilde T_j^{N, i}$ is isometric to $T_j^N$ via $p$. We prove that 
    \begin{align}
    &p^{-1}(T(\xi, \delta)) = \bigsqcup_{i \in I} \bigcup_{j \in J_N} \widetilde T_j^{N, i} \label{eq: lower bound on perimeter -10} \\
    &|\widetilde T_j^{N, i} \cap \widetilde T_{j'}^{N, i'}| = 0 \text{ for $i, i' \in I$, $j, j' \in J_N$ and $(i, j) \neq (i', j')$} \label{eq: lower bound on perimeter -9} \\
    &|\widetilde T_j^{N, i}| = \kappa_N |T(\xi, \delta)| \label{eq: lower bound on perimeter -8}
    \end{align}

    \eqref{eq: lower bound on perimeter -10} follows directly from \textit{step 1} and the fact that $\cup_{j \in J_N} \widetilde T_j^{N, i} = T(\widetilde \xi^{N, i}, \delta)$. \eqref{eq: lower bound on perimeter -9} follows from the fact that $\{\alpha_j^N\}_{j \in J_N}$ forms a partition of $\xi^N$ only overlapping at points $\{\xi^N(t_j^N) : j \in J_N\}$ and the fact that $\widetilde \alpha_j^{N,i}$ is disjoint from $\widetilde \alpha_{j'}^{N, i'}$ for $i \neq i'$. \eqref{eq: lower bound on perimeter -8} follows from the definition of $\widetilde \alpha_j^{N,i}$.

    \textit{Step 3:} We conclude the proof by using a relative isoperimetric inequality whose proof we defer to the appendix (see \cref{prop: relative isoperimetric inequality of partial cylinders}). Since $D$ is a fundamental domain, $|D \cap p^{-1}(T(\xi, \delta))| = |T(\xi, \delta)|$. Combining this with \eqref{eq: lower bound on perimeter -10} and \eqref{eq: lower bound on perimeter -9} we get
    \begin{equation*}
        |T(\xi, \delta)| = \sum_{i \in I, j \in J_N} |D \cap \widetilde T_j^{N,i}|.
    \end{equation*}
    Let $S = \{(i, j) : |D \cap \widetilde T_j^{N, i}| \leq (1/2) |T(\xi, \delta)| \}$. It follows directly from \eqref{eq: lower bound on perimeter -8} that
    \begin{equation*}
        (1-\kappa_N) |T(\xi, \delta)| \leq \sum_{(i, j) \in S} |D \cap \widetilde T_j^{N, i}|.
    \end{equation*}
    Now for $(i, j) \in S$, since $|D \cap \widetilde T_j^{N, i}| \leq (1/2) |T(\xi, \delta)| = (1/2 \kappa_N) |\widetilde T_j^{N, i}|$, there is a constant $c = c(1/2 \kappa_N) >0$ such that we have the following relative isoperimetric inequality
    \begin{equation*}
        c |D \cap \widetilde T_j^{N, i}|^{2/3} \leq P(D, \widetilde T_j^{N, i}).
    \end{equation*}
    See \cref{prop: relative isoperimetric inequality of partial cylinders} for a proof. Combining the previous two expressions we get
    \begin{align*}
        c \cdot (1-\kappa_N)^{2/3} |T(\xi, \delta)|^{2/3} \leq c \cdot \left( \sum_{(i,j) \in S} |D \cap \widetilde T_j^{N, i}| \right)^{2/3} &\leq \sum_{(i, j) \in S} c|D \cap \widetilde T_j^{N,i}|^{2/3} \\
        &\leq P(D, \widetilde T_j^{N,i}) \\
        &\leq P(D),
    \end{align*}
    which is a lower bound on $P(D)$ depending only on $\Gamma$ and the covering space.
\end{proof}

\subsection{Perimeter of \texorpdfstring{$H$}{H}-domains}
Note that when $H$ is not trivial, an $H$-domain $F$ will not have finite perimeter. This follows from the $H$-invariance (i.e. \eqref{eq: HD 2}) of $H$-domains: if $P(F, \widetilde B)>0$ for some ball $\widetilde B$, then $P(F) \geq P(F, \bigcup_{g \in H} \widetilde B) = \sum_{g \in H} P(F, \widetilde B) = \infty$. Hence, it is natural to measure the perimeter of $F$ along the `orthogonal' direction $\pi_1(M)/H$ (see discussion following \cref{def: H domains}). The next proposition shows this is well defined.
\begin{proposition} \label{prop: perimeter of H-domain well defined}
    Let $F$ be an $H$-domain of locally finite perimeter and let $R$ be a fundamental domain in the universal cover. The quantity
    \begin{equation} \label{eq: perimeter of H-domaion well defined}
        P\left( F, \bigcup_{gH \in \pi_1(M)/H} gR \right)
    \end{equation}
    is well defined, i.e. it does not depend on the specific representative of the class $gH$.
\end{proposition}
\begin{proof}
    All that we need to check is that $P(F, gR) = P(F, ghR)$ when $h \in H$. Without loss of generality, we can consider only the case $ g = \id$ which is equivalent to $P(F, R) = P(h^{-1}F, R)$. This holds since $|F \triangle h^{-1}F| = 0$ by definition of $H$-domain.
\end{proof}

We now aim at \cref{prop: perimeter of fundamental domains and H-domains} which describes the relationship between \eqref{eq: perimeter of H-domaion well defined} and the perimeter of the corresponding fundamental domain (after one chooses a good $R$). From now on, we set $R$ to be the fundamental domain in the universal cover obtained in \cref{cor: existence of fundamental domains of finite perimeter in arbitrary covers}. We recall that $R$ is open, of finite perimeter, and $q|_R: R \rightarrow M$ is an isometry onto its image. It is easy to check that for any normal subgroup $H$, $f_H(R)$ is an open fundamental domain in $\widetilde M(H)$. Since $q|_R$ is an isometry onto its image, it follows $(f_H)|_R$ is an isometry onto its image. Hence, it is justified to abuse notation by denoting $f_H(R)$ by $R$ and letting the group $H$ be understood from context. \label{page: sheet of a cover} We call $R$ a \textbf{sheet} of the cover in which it lives. Since a sheet is an open fundamental domain, it enjoys the following key property: for any normal subgroup $H$,
\begin{equation} \label{eq: closure of open fd cover}
    \widetilde M(H) = \bigcup_{g \in \deck(p_H)} g \overline{R}.
\end{equation}

We make the following simplifying assumption: For any set $E$ of locally finite perimeter,
\begin{equation} \label{eq: tesselation assumption}
    P(E, R) = P(E, \overline{R}).
\end{equation}
This is not true as stated (e.g. take $E = R$). However, given a set of locally finite perimeter $E$, we can modify the vertex of $\Sigma_0$ in the construction of $R$ to guarantee that \eqref{eq: tesselation assumption} holds. Furthermore, given a countable collection of sets of locally finite perimeter, we can guarantee that \eqref{eq: tesselation assumption} holds for every set in the collection, by again modifying the basepoint of $\Sigma_0$. In the sequel, we will always be working with at most a countable collection of sets of locally finite perimeter, and so we can always assume \eqref{eq: tesselation assumption} holds. This has the advantage of decluttering the presentation of the results.

\begin{proposition} \label{prop: perimeter of fundamental domains and H-domains}
    Let $F$ be an $H$-domain and let $D = f_H(F)$ be the corresponding fundamental domain in $p_H: \widetilde M(H) \rightarrow M$. If $F$ or $D$ has locally finite perimeter, then so does the other. Furthermore,
    \begin{align}
        P \left(F, \bigcup_{gH \in \pi_1(M)/H} gR \right) &= P(D).
    \end{align}
\end{proposition}

\begin{proof}
    Suppose $F$ has locally finite perimeter. We have
    \begin{align*}
        P\left( F, \bigcup_{gH \in \pi_1(M)/H} gR \right) = \sum_{gH \in \pi_1(M)/H} P(F, gR) = \sum_{g' \in \deck(p_H)} P(D, g'R) &= \sum_{g' \in \deck(p_H)} P(D, g' \overline{R}) \\
        &= P(D).
    \end{align*}
    Indeed, the first equality holds because $R$ is open and a fundamental domain and so $R$ is disjoint from any $gR$, $g \neq e$. The second equality holds because $f_H$ is an isometry on $R$ and because each class $gH \in \pi_1(M)/H$ induces a unique $g' \in \deck(p_H)$ completing the diagram
    \begin{equation*}
        \begin{tikzcd}[row sep = 1em]
            \widetilde M(e) \arrow[d, "f_H"] \arrow[r, "g"] & \widetilde M(e) \arrow[d, "f_H"] \\
            \widetilde M(H) \arrow[r, "g'"] & \widetilde M(H).
        \end{tikzcd}
    \end{equation*}
    The third equality holds by \eqref{eq: tesselation assumption}. Finally, the fourth holds for the same reason as the first, together with \eqref{eq: closure of open fd cover}.

    The above computation also holds (from right to left) if we instead assume that $D$ has locally finite perimeter. This concludes the proof.
\end{proof}

\section{Compactness of \texorpdfstring{$H$}{H}-domains and the existence of minimising fundamental domains} \label{sec: compactness}
The aim of this section is to prove a compactness theorem for sequences of $H$-domains of locally finite perimeter, where the group $H$ is allowed to vary suitably along the sequence. We remind the interested reader to consult \cite{cesaroni2024periodic, cesaroni2022minimal, nobili2024lattice, cesaroni2025minimal} for similar ideas. Before stating this, we first introduce a notion of convergence for subsets of $\pi_1(M)$:

\begin{definition} \label{def: subgroup convergence}
    We say that a sequence $A_k$ of subsets of $\pi_1(M)$ \textbf{converges} to a subset $A$, denoted $A_k \rightarrow A$, if
    \begin{align}
        \lim_k A_k \geq A:& \text{ For all } a \in A, \text{ we eventually have } a \in A_k; \\
        \lim_k A_k \leq A:& \text{ For all } g \notin A, \text{ we eventually have } g \notin A_k.
    \end{align}
\end{definition}

\begin{remark}
    This convergence is induced by the Tychonoff product topology on the countable product $2^{\pi_1(M)}$. This observation gives an efficient route to \cref{prop: compactness of subgroups} via Tychonoff's compactness theorem and the metrisability of the Tychonoff topology on the countable product $2^{\pi_1(M)}$. We opt to present a more direct self-contained proof of \cref{prop: compactness of subgroups} for which the reader does not need knowledge of the above topological results.
\end{remark}

We will need the following general properties of the above convergence. The indices in (viii) may seem strange, but they are chosen to be consistent with how (viii) is used in the proof of \cref{lem: lower semicontinuity}.
\begin{proposition} \label{prop: properties of convergence}
Let $A_k, B_k$ be sequences of subsets of $\pi_1(M)$, let $A, B$ be subsets of $\pi_1(M)$.
    \begin{enumerate}
        \item[(i)] If $A_k$ converges, then the limit is unique.
        \item[(ii)] If $A_k \rightarrow A$, then $A = \cup_K \cap_{k \geq K} A_k$.
        \item[(iii)] If $A_k \rightarrow A$ and $j(k)$ is a subsequence, then $A_{j(k)} \rightarrow A$ as well.
        \item[(iv)] If $A_k \rightarrow A$ and $g \in \pi_1(M)$, then $gA_k \rightarrow gA$.
        \item[(v)] If $A_k \rightarrow A$ and $B_k \rightarrow B$, then $A_k \cup B_k \rightarrow A \cup B$.
    \end{enumerate}

In the case where we have a sequence of normal subgroups $H_k$, we also have
    \begin{enumerate}
        \item[(vi)] If $H_k \rightarrow H$, then $H$ is a normal subgroup.
        \item[(vii)] If $H_k \rightarrow H$ and $H$ has finite index, then $H_k = H$ eventually.
        \item[(viii)] If $H_k \rightarrow H$ and $(\eta_k^2), ..., (\eta_k^{j})$ are sequences of elements in $\pi_1(M)$ such that $H_k \cup \eta_k^2 H_k \cup ... \cup \eta_k^{j} H_k \rightarrow \pi_1(M)$, then $H$ has finite index.
    \end{enumerate}
\end{proposition}
\begin{proof}
    (i) If $A_k \rightarrow A, A'$ where $A \neq A'$, then without loss of generality there is $g \in A \setminus A'$. But $A_k \rightarrow A$ implies $g \in A_k$ eventually while $A_k \rightarrow A'$ implies $g \notin A_k$ eventually. This is impossible.
    
    (ii) It follows directly that $\lim_k A_k \geq A$. If $g \notin A$, then, by definition of $A$, $g$ is never eventually in $A_k$. But $A_k$ converges so $g$ must eventually not be in $A_k$. Hence we also have $\lim_k A_k \leq A$.

    (iii), (iv), and (v) are immediate.

    (vi) This is direct from the expression of the limit obtained in (ii)
    
    (vii) Since $H$ is a finite index subgroup of $\pi_1(M)$ and $\pi_1(M)$ is finitely generated, we deduce that $H$ is itself finitely generated. Therefore, since $\lim_k H_k \geq H$, we have that eventually $H_k$ contains the finitely many generators of $H$, and so $H_k \geq H$. Let $\pi_1(M)/H = \{H, g^2H, ..., g^J H \}$. We eventually have that $g^j \notin H_k$ for all $2 \leq j \leq J$. Hence, using that $H_k$ is a group, we deduce that $H_k$ eventually contains $H$ but not $g^j H$ for all $2 \leq j \leq J$. Therefore, $H_k = H$ eventually.

    (viii) If $H = \pi_1(M)$ then the result holds trivially. Assume $H \neq \pi_1(M)$ and let $g^2 \in \pi_1(M) \setminus H$. Since $H_k \cup \eta_k^2H_k \cup ... \cup \eta_k^{j} H_k \rightarrow \pi_1(M)$ and $g^2 \notin H_k$ eventually, then the pigeon hole principle provides a subsequence and a fixed $i \in \{2, ..., j\}$ such that $g^2 \in \eta_k^i H_k$ eventually. So $\eta_k^i H_k = g^2 H_k$ eventually and hence by (iv), $\eta_k^i H_k \rightarrow g^2 H$. If $H \cup g^2H = \pi_1(M)$ then the result holds. Otherwise, we iterate the above scheme. It is easy to see using (v) and $H_k \cup \eta_k^2 H_k \cup ... \cup \eta_k^{j} H_k \rightarrow \pi_1(M)$ that the induction eventually terminates with $H \cup g^2H \cup ... \cup g^sH = \pi_1(M)$ where $s \leq j$ and $g^i \in \pi_1(M) \setminus (H \cup g^2H \cup ... \cup g^{i-1}H)$. Therefore $H$ has finite index.
\end{proof}

The above notion of convergence of subsets enjoys the following nice compactness.
\begin{proposition} \label{prop: compactness of subgroups}
    Let $A_k$ be a sequence of subsets of $\pi_1(M)$. Then there is a subset $A$ of $\pi_1(M)$ such that, after possibly passing to a subsequence, $A_k \rightarrow A$.
\end{proposition}
\begin{proof}
    We start with the simple observation that for any sequence $A_k$ of subsets, $\lim_k A_k \geq \cup_K \cap_{k \geq K} A_k$. 

    Enumerate $\pi_1(M)$ and write $g \leq h$ (resp. $g < h$) if $g \in \pi_1(M)$ appears before (resp. strictly before) $h$ in the enumeration. Let $G = \cup_K \cap_{k \geq K} A_k$. As per the above observation, we have $\lim_k A_k \geq G$. If $\lim_k A_k \leq G$, we are done. If $\lim_k A_k \nleq G$, let $g_1$ be the first element in the enumeration of $\pi_1(M)$ such that $g_1 \notin G$ and there is a subsequence $j_1 \subset \N$ such that $g_1 \in A_{j_1(k)}$ for all $k$. Set $G_1 = \cup_k \cap_{k \geq K} A_{j_1(k)} \supset G \cup \{g_1\}$.
    
    We have that $\lim_k A_{j_1(k)} \geq G_1$. If $\lim_k A_{j_1(k)} \leq G_1$ we are done. If $\lim_k A_{j_1(k)} \nleq G_1$, then as before let $g_2$ be the first element in the enumeration of $\pi_1(M)$ such that $g_2 \notin G_1$ and there is a subsequence $j_2 \subset j_1$ such that $g_2 \in A_{j_2(k)}$ for all $k$. Set $G_2 = \cup_K \cap_{k \geq K} A_{j_2(k)} \supset G_1 \cup \{g_2\}$ and note that $g_1 < g_2$.

    Continue with this process. If there is $i$ such that $\lim_k A_{j_i(k)} \leq G_i$, then we are done. If there is no such $i$, let $G = \cup_K \cap_{k \geq K} A_{j_k(k)}$. We have that $\lim_k A_{j_k(k)} \geq G$. Suppose for a contradiction that $\lim_k A_{j_k(k)} \nleq G$, so there is $g \notin G$ and a subsequence $j(k)\subset j_k(k)$ such that $g \in A_{j(k)}$ for all $k$. Since $g_i < g_{i+1}$ for all $i$, there is $i$ such that $g_i \leq g \leq g_{i+1}$. Hence either $g \in G_i$, or $g \notin G_i$ and $g \notin A_{j_i(k)}$ for $k$ large enough, or $g \notin G_i$ and $g \in A_{j_{i+1}(k)}$ for all $k$. In the first case, since $G \supset G_i$ for all $i$, we get $g \in G$ which is a contradiction. In the second case, since $j(k)$ is eventually a subsequence of $j_i(k)$, we eventually get $g \notin A_{j(k)}$ which is a contradiction. In the third case, by construction of $G_{i+1}$, we get $g \in G_{i+1}$ which, as in the first case, is a contradiction. Therefore, $\lim_k A_{j_k(k)} \leq G$ as required.
\end{proof}

We now state and prove the compactness and lower semicontinuity theorem. As before, $R$ is an open fundamental domain in $\widetilde M(e)$.
\begin{theorem} \label{thm: compactness of H-domains}
    Let $H_k$ be a sequence of normal subgroups of $\pi_1(M)$ such that $H_k \rightarrow H$. Let $F_k$ be a sequence of $H_k$-domains of locally finite perimeter such that
    \begin{equation} \label{eq: uniform bound for H-domains}
        \sup_k P \left(F_k, \bigcup_{gH_k \in \pi_1(M)/H_k} gR \right) \leq C_0.
    \end{equation}
    Then there exists an $H$-domain $F$ of locally finite perimeter such that, after possibly passing to a subsequence,
    \begin{equation} \label{eq: lowersemicontinuity for H-domains}
        P \left(F, \bigcup_{gH \in \pi_1(M)/H} gR \right) \leq \liminf_{k \rightarrow \infty} P \left( F_k, \bigcup_{gH_k \in \pi_1(M)/H_k} gR\right).
    \end{equation}
\end{theorem}
We will divide the proof into two lemmas: the first proves the existence of $F$, and the second shows the lower semicontinuity property of $F$. We recall that the proofs share many similarities with the proofs in \cite{novaga2022isoperimetricclusters, cesaroni2024periodic, nobili2024lattice, cesaroni2025minimal}, particularly \cites[Theorem 3.2]{cesaroni2024periodic}[Theorem 3.3]{cesaroni2022minimal}. However, our geometric setting differs from the one explored in the above papers. Technically, this translates into two main differences: in our case, $M$ is not bounded and $F_k$ is not a fundamental domain. While the first is easily managed, the second poses additional difficulties, especially in the proof of \cref{lem: lower semicontinuity} since both $F_k$ and the set in which the perimeter of $F_k$ is measured change along the subsequence.

We will need the following definition. Given a sequence of normal subgroups $H_k \leq \pi_1(M)$ and a sequence of deck transformations $(g_k)$, we say that $(g_k)$ \textbf{goes to infinity transversally to $H_k$}, written $(g_k) \rightarrow \infty \perp H_k$, if there does not exist $g \in \pi_1(M)$ such that $g_k \in gH_k$ infinitely often. For example, if $H_k = H$ is a constant sequence and $\eta_1H, \eta_2H, ...$ is an enumeration of $\pi_1(M)/H$, then $g_k := \eta_k$ goes to infinity transversally to $H_k$.

\begin{lemma} \label{lem: concentration compactness}
    Let $H_k$ be a sequence of normal subgroups of $\pi_1(M)$ such that $H_k \rightarrow H$. Let $F_k$ be a sequence of $H_k$-domains of locally finite perimeter satisfying \eqref{eq: uniform bound for H-domains}. Then, after possibly passing to a subsequence, we can find sequences of deck transformations $(\lambda_k^1), (\lambda_k^2), ...$ with $(\lambda_k^l)^{-1}(\lambda_k^{l'}) \rightarrow \infty \perp H_k$ for $l \neq l'$, and sets $F^1, F^2, ...$ given by 
    \begin{align}
        (\lambda_k^l)^{-1}F_k &\rightarrow F^l,
    \end{align}
    where the convergence is locally as sets of finite perimeter, such that $F := \cup_l F^l$ is an $H$-domain.
\end{lemma}
\begin{proof}
    Let $\{B^i\}$ be a countable cover of $M$ by Euclidean balls and for each $i$, let $\widetilde B^i$ be an isometric lift of $B^i$ to $\widetilde M(e)$. For each $i$ and $k$, let $g^{i, 1}_k H_k, g^{i, 2}_k H_k, ...$ be an enumeration of $\pi_1(M)/H_k$ that is decreasing in the sense that
    \begin{equation} \label{eq: decreasing enumeration}
        |(g^{i, j}_k)^{-1}F_k \cap \widetilde B^i| \geq |(g^{i, j+1}_k)^{-1} F_k \cap \widetilde B^i|.
    \end{equation}
    This enumeration is well defined. Indeed, first note that \eqref{eq: decreasing enumeration} does not depend on the choice of representative of the class $g^{i, j}_k H_k$ because, by definition of $H_k$-domains, $|(g')^{-1}F_k \cap \widetilde B^i| = |g^{-1} F_k \cap \widetilde B^i|$ for any $g' \in gH_k$. Second, we have $|(g^{i, j}_k)^{-1} F_k \cap \widetilde B^i| \leq |\widetilde B^i| < \infty$ and so we can indeed order $\pi_1(M)/H_k$ as described.

    Note that properties \eqref{eq: HD 1} and \eqref{eq: HD 4} of $H_k$-domains directly imply that for any $i$,
    \begin{align}
        |F_k^{i, j} \cap F_k^{i, j'}| &= 0 \text{ for } j \neq j', \text{ and} \label{eq: prelimit properties of Fkij 1}\\
        |\widetilde B^i \setminus \bigcup_j F_k^{i, j}| &= 0 \label{eq: prelimit properties of Fkij 2}
    \end{align}
    
    Set the notation $F^{i, j}_k := (g^{i, j}_k)^{-1}F_k$. Intuitively, each $F^{i, j}_k$ `tracks' a piece of $F_k$, that is perhaps `drifting towards infinity' along the orthogonal direction $\pi_1(M) / H_k$ (see \cref{def: H domains} for more on this notion of orthogonal direction). Note that it is possible that many of the $F_k^{i, j}$s are tracking the same piece of $F_k$ - we address this in \textit{step 3} below. Up to passing to a diagonal subsequence, the standard compactness and lower semicontinuity (see \cref{subsec: relevant theory of sets of finite perimeter}) gives us sets of locally finite perimeter $F^{i, j}$ such that
    \begin{equation*}
        F^{i, j}_k \rightarrow F^{i, j}
    \end{equation*}
    locally as sets of finite perimeter, and for any open set $A$,
    \begin{equation} \label{eq: standard lowersemi}
        P(F^{i, j}, A) \leq \liminf_{k \rightarrow \infty}P(F^{i, j}_k, A).
    \end{equation}
    The first two steps of the proof below are concerned with showing what properties of the sequences $F_k^{i, j}$ are carried onto the limits $F^{i, j}$.

    \textit{Step 1:} We start by showing that \eqref{eq: prelimit properties of Fkij 1} and \eqref{eq: prelimit properties of Fkij 2} are preserved in the limit, i.e. for fixed $i$,
    \begin{align}
        |F^{i, j} \cap F^{i, j'}| &= 0 \text{ for } j \neq j', \text{ and } \label{eq: no overlap} \\
        |\widetilde B^i \setminus \bigcup_j F^{i, j}| &= 0. \label{eq: no mass loss at infinity}
    \end{align}
    Equation \eqref{eq: no overlap} is obtained directly by taking the limit $k \rightarrow \infty$ in \eqref{eq: prelimit properties of Fkij 1}. In view of \eqref{eq: no overlap}, \eqref{eq: no mass loss at infinity} is equivalent to
    \begin{equation} \label{eq: no mass loss at infinity 2}
        |\widetilde B^i| = \sum_j |F^{i, j} \cap \widetilde B^i|.
    \end{equation}
    We now focus on proving this equality. By combining \eqref{eq: prelimit properties of Fkij 1} and \eqref{eq: prelimit properties of Fkij 2}, we get
    \begin{align} \label{eq: cover of ball with no overlap}
        |\widetilde B^i| = \sum_j |F^{i, j}_k \cap \widetilde B^i|
    \end{align}
    and hence
    \begin{align} \label{eq: sum partition + isoperimetry}
        \sum_{j=1}^J|F^{i,j}_k \cap \widetilde B^i| \leq |\widetilde B^i| \leq \sum_{j=1}^J |F^{i, j}_k \cap \widetilde B^i| + \sum_{j=J+1}^{\infty} |F^{i, j}_k \cap \widetilde B^i|^{1/3} P(F^{i, j}_k \cap \widetilde B^i)
    \end{align}
    where we use the non-sharp isoperimetric inequality $|E|^{2/3} \leq P(E)$ (see \cite[Proposition 12.35]{maggi2012sets}) on the right hand side. Combining \eqref{eq: cover of ball with no overlap} with \eqref{eq: decreasing enumeration} implies that $|F^{i, j}_k \cap \widetilde B^i| \leq |\widetilde B^i|/J$ for all $j \geq J+1$. In addition, up to slightly increasing the radius of $\widetilde B^i$,
    \begin{align*}
        \sum_{j = J+1}^\infty P(F^{i, j}_k \cap \widetilde B^i) = \Haus^2(\partial \widetilde B^{i}) + \sum_{j = J+1}^{\infty} P(F^{i, j}_k, \widetilde B^i) &\leq \Haus^2(\partial \widetilde B^i) + \sum_{gH_k \in \pi_1(M)/H_k} P(F_k, g\widetilde B^i) \\
        &\leq \Haus^2(\partial \widetilde B^i) + C_0
    \end{align*}
    where $C_0$ is the constant from \eqref{eq: uniform bound for H-domains}. Applying these estimates in \eqref{eq: sum partition + isoperimetry} and then sending $k \rightarrow \infty$ we get
    \begin{equation*}
        \sum_{j=1}^J|F^{i,j} \cap \widetilde B^i| \leq |\widetilde B^i| \leq \sum_{j=1}^J |F^{i, j} \cap \widetilde B^i| + \frac{|\widetilde B^i|^{1/3}}{J^{1/3}} (\Haus^2(\partial \widetilde B^i) + C_0)
    \end{equation*}
    Now sending $J \rightarrow \infty$ we get $|\widetilde B^i| = \sum_j |F^{i, j} \cap \widetilde B^i|$ which is \eqref{eq: no mass loss at infinity 2}.

    \textit{Step 2:} We now prove that the first two properties of $H_k$-domains are carried to the limit. More precisely, by making use of
    $H_k \rightarrow H$, we show that each $F^{i,j}$ satisfies the first two properties of $H$-domains:
    \begin{align}
        |F^{i, j} \cap g F^{i, j}| &= 0 \text{ for all } g \notin H \label{eq: Fij satisfies HD1}\\
        |F^{i, j} \triangle gF^{i, j}| &= 0 \text{ for all } g \in H. \label{eq: Fij satisfies HD2}
    \end{align}
    To see \eqref{eq: Fij satisfies HD1}, fix some $g \notin H$. Since $H_k \rightarrow H$, we know that $g \notin H_k$ for $k$ sufficiently large. For such large $k$, since normal subgroups are fixed by conjugation, we get $(g_k^{i, j}) g (g_k^{i, j})^{-1} \notin H_k$. The definition of $H_k$-domains then gives that $0 = |F_k \cap (g_k^{i, j}) g (g_k^{i, j})^{-1} F_k| = |F_k^{i, j} \cap g F_k^{i, j}|$. Taking $k \rightarrow \infty$ we conclude $|F^{i, j} \cap g F^{i, j}|=0$ as required. To see \eqref{eq: Fij satisfies HD2}, fix some $g \in H$. Since $H_k \rightarrow H$, we know that $g \in H_k$ for $k$ sufficiently large. For such large $k$, following the same structure as the proof of \eqref{eq: Fij satisfies HD1}, we get $(g_k^{i, j})g(g_k^{i, j})^{-1} \in H_k$, and so $0 = |F_k^{i, j} \triangle g F_k^{i, j}|$, and so $|F^{i, j} \triangle g F^{i, j}| = 0$ as required.

    \textit{Step 3:} It is of course possible that many different $F^{i, j}_k$s are tracking the same piece of $F_k$, and hence different $F^{i, j}$s are the same up to a translation by some deck transformation. For each piece of $F_k$, we now select \textit{one} $F^{i, j}_k$ tracking it.
    
    For two sequences $(g_k), (g_k') \subset \pi_1(M)$, we define the \textbf{equivalence relation} $(g_k) \sim (g_k')$ if there are finitely many $a_1, ..., a_N \in \pi_1(M)$ such that for all $k$, $g_k^{-1}g_k' \in \cup_i a_i H_k$. Observe that if $(g_k) \nsim (g_k')$, then either $(g_k)^{-1}(g_k') \rightarrow \infty \perp H_k$ or there is a subsequence $k' \subset k$ such that $(g_{k'}) \sim(g_{k'}')$. Hence, after passing to a diagonal subsequence in $k$ if necessary, we have the dichotomy: $(g_k^{i, j}) \sim (g_k^{i', j'})$ or $(g_k^{i, j})^{-1}(g_k^{i', j'}) \rightarrow \infty \perp H_k$. Indeed, enumerate the pairs of pairs $(i, j), (i', j')$ with $(i, j) \neq (i', j')$. Let $(i_0, j_0), (i_0', j_0')$ be the first in the enumeration. We have that either $(g_k^{i_0, j_0}) \sim (g_k^{i_0', j_0'})$, or $(g_k^{i_0, j_0})^{-1}(g_k^{i_0', j_0'}) \rightarrow \infty \perp H_k$, or that there exists a subsequence $k' \subset k$ such that $(g_{k'}^{i_0, j_0}) \sim (g_{k'}^{i_0', j_0'})$. Passing to the subsequence $k'$ if necessary, we can assume that either $(g_k^{i_0, j_0}) \sim (g_k^{i_0', j_0'})$ or $(g_k^{i_0, j_0})^{-1}(g_k^{i_0', j_0'}) \rightarrow \infty \perp H_k$. Now repeat this procedure with the second pair of pairs in the enumeration, again passing to a subsequence if necessary. We now have the dichotomy for the first two pairs of pairs in the enumeration. Continuing this process, we can define a diagonal subsequence $k^*$ along which, for any $(i, j), (i', j')$,  either $(g_{k^*}^{i, j}) \sim (g_{k^*}^{i', j'})$ or $(g_{k^*}^{i, j})^{-1}(g_{k^*}^{i', j'}) \rightarrow \infty \perp H_{k^*}$ as required.
    
    We now prove that equivalent sequences track the same piece of $F_k$, and inequivalent sequences track different pieces of $F_k$:
    \begin{align}
        \text{If $(g_k^{i, j}) \sim (g_k^{i', j'})$,} &\text{ then $|F^{i, j} \triangle g F^{i', j'}| = 0$ for some $g \in \pi_1(M)$.} \label{eq: sim sequence track same piece}\\
        \text{ If $(g_k^{i, j})^{-1} (g_k^{i', j'}) \rightarrow \infty \perp H_k$,} &\text{ then $|F^{i, j} \cap g F^{i', j'}| = 0$ for all $g \in \pi_1(M)$.} \label{eq: unsim sequences track different pieces}
    \end{align}
    To see \eqref{eq: sim sequence track same piece}, start by  noticing that $(g_k^{i, j}) \sim (g_k^{i', j'})$ implies that there must be some $g \in \pi_1(M)$ and some subsequence $k'$ such that $(g_{k'}^{i, j})^{-1} (g_{k'}^{i', j'}) \in g H_{k'}$ along the subsequence. Rearranging using that $H_{k'}$ is normal, we obtain $(g_{k'}^{i, j}) g (g_{k'}^{i', j'})^{-1} \in H_{k'}$, which together with the definition of $H_{k'}$-domain gives $0 = |F_k \triangle (g_{k'}^{i, j})g(g_{k'}^{i', j'})^{-1} F_k|=|F_{k'}^{i, j} \triangle g F_{k'}^{i', j'}|$. Taking the limit $k \rightarrow \infty$ we obtain $|F^{i, j} \triangle g F^{i', j'}| = 0$ as required. To see \eqref{eq: unsim sequences track different pieces}, fix some $g \in \pi_1(M)$. Since $(g_k^{i, j})^{-1} (g_k^{i', j'}) \rightarrow \infty \perp H_k$, we must have that $(g_k^{i, j})^{-1} (g_k^{i', j'}) \notin gH_k$ for all $k$ sufficiently large. For such large $k$, following the structure of the proof of the previous case, we obtain that $(g_k^{i, j}) g (g_k^{i', j'})^{-1} \notin H_k$, and so by definition of $H_k$-domains, $0 = |F_k \cap (g_k^{i, j})g(g_k^{i', j'})^{-1} F_k| = |F_k^{i, j} \cap g F_k^{i', j'}|$, and so $|F^{i, j} \cap g F^{i', j'}| = 0$ as required.

    To achieve the objective of this step, we therefore choose a representative of each equivalence class under $\sim$. More precisely, let $S$ be the set whose elements are all the sequences $(g_k^{i, j})$ with different indices $(i, j)$. Let $(\lambda_k^1), (\lambda_k^2), ...$ be representatives of the equivalence classes $S / \sim$, let $F_k^l := (\lambda_k^l)^{-1} F_k$, and let $F^l$ be the limit of $F_k^l$ as sets of locally finite perimeter. We have that $(\lambda_k^l)^{-1}(\lambda_k^{l'}) \rightarrow \infty \perp H_k$ and hence, by \eqref{eq: unsim sequences track different pieces}, that
    \begin{equation} \label{eq: Fls dont intersect}
        \text{$|F^l \cap g F^{l'}| = 0$ for all $g \in \pi_1(M)$}
    \end{equation}
    for $l \neq l'$.
    
    \textit{Step 4:} We conclude the proof by showing that $F := \cup_l F^l$ is an $H$-domain. We start by showing \eqref{eq: HD 1} holds. Fix $g \notin H$. By \eqref{eq: Fij satisfies HD1} $|F^l \cap gF^l| = 0$ for all $l$. By \eqref{eq: Fls dont intersect}, $|F^l \cap gF^{l'}| = 0$ for all pairs $l \neq l'$. Therefore, $|F \cap gF| \leq \sum_{l, l'} |F^l \cap gF^{l'}| = 0$ as required. We now show \eqref{eq: HD 2}. Let $g \in H$. By \eqref{eq: Fij satisfies HD2}, $|F^l \triangle g F^l| = 0$ for all $l$. Since $F \triangle gF \subset \cup_l (F^l \triangle gF^l)$, we get that $|F \triangle gF| \leq \sum_l |F^l \triangle g F^l| = 0$ as required. Finally we prove the last property of $H$-domains \eqref{eq: HD 3}. Suppose for a contradiction that \eqref{eq: HD 3} does not hold. This implies we can find a set $E$ of positive measure such that
    \begin{equation} \label{eq: F missing mass}
        \left| E \cap \left( \bigcup_{g \in \pi_1(M)} gF \right) \right| = 0.
    \end{equation}
    By \eqref{eq: no mass loss at infinity}, there is $(i, j)$ and $g \in \pi_1(M)$ such that $|E \cap gF^{i, j}| > 0$. Let $(\lambda_k^l)$ be the representative of the equivalence class containing $(g_k^{i, j})$, i.e. $(\lambda_k^l) \sim (g_k^{i, j})$. By \eqref{eq: sim sequence track same piece}, there is some $g' \in \pi_1(M)$ such that $|F^{i, j} \triangle g' F^l| = 0$, or equivalently, $|g F^{i, j} \triangle g g'F^l| = 0$. Hence, $|E \cap g g' F^l| > 0$ which contradicts \eqref{eq: F missing mass}. This proves \eqref{eq: HD 3}.
\end{proof}

Having proven the existence of a `limiting' $H$-domain, we now show that the perimeter is lower semicontinuous with respect to this limit. Recall that $R$ is an open fundamental domain in $\widetilde M(e)$.
\begin{lemma} \label{lem: lower semicontinuity}
    In the setting of \cref{lem: concentration compactness}, we have that
    \begin{equation}
        P \left( F, \bigcup_{gH \in \pi_1(M)/H} gR \right) \leq \liminf_{k \rightarrow \infty} P \left( F_k, \bigcup_{gH_k \in \pi_1(M)/H_k} gR \right).
    \end{equation}
\end{lemma}
\begin{proof}
    The two key ingredients in this proof are that $(\lambda_k^l)^{-1} (\lambda_k^{l'}) \rightarrow \infty \perp H_k$ for $l \neq l'$ and that $H_k \rightarrow H$. We divide the proof into two cases: $H$ has finite index and $H$ does not have finite index.

    \textit{Case 1: $H$ has finite index.} By \cref{prop: properties of convergence}(vii), we have that $H_k = H$ eventually. We deduce from this that the collection $(\lambda_k^1), (\lambda_k^2), ...$ actually only contains $(\lambda_k^1)$; and in particular $F = F^1$. Indeed, if there were two distinct $(\lambda_k^l), (\lambda_k^{l'})$ (i.e. $l \neq l'$), then eventually $(\lambda_k^l)^{-1}(\lambda_k^{l'}) \rightarrow \infty \perp H$ which implies that $\pi_1(M)/H$ is infinite - a contradiction. We can then compute
    \begin{align*}
        \liminf_k P \left(F_k, \bigcup_{gH_k \in \pi_1(M)/H_k} gR \right) &= \liminf_k P \left(F_k, \bigcup_{gH \in \pi_1(M)/H} gR \right) \\
        &= \liminf_k P \left(F_k, \lambda_k^1 \left(\bigcup_{gH \in \pi_1(M)/H} gR \right) \right) \\
        &= \liminf_k P \left(F_k^1, \bigcup_{gH \in \pi_1(M)/H} gR \right) \geq P \left(F^1, \bigcup_{gH \in \pi_1(M)/H} gR \right),
    \end{align*}
    where the last inequality is the standard lower semicontinuity of perimeter along the convergence $F_k^1 \rightarrow F^1$ (see \eqref{eq: standard lowersemi}). Since $F = F^1$, we have just computed the desired lower semicontinuity.

    \textit{Case 2: $H$ does not have finite index.} Let $\eta^1, \eta^2, ...$ be an enumeration of $\pi_1(M)$ with $\eta^1 = e$. We define the total order $\leq$ on $\pi_1(M)$ by $\eta^j \leq \eta^{j'}$ if and only if $j \leq j'$. Let $\eta_k^1H_k, \eta_k^2 H_k, ...$ (resp. $\eta_\infty^1H, \eta_\infty^2H, ...$) be an enumeration of $\pi_1(M)/H_k$ (resp. $\pi_1(M)/H$). By changing $\eta_k^j$ (resp. $\eta_\infty^j$) to another element in the class $\eta_k^jH_k$ (resp. $\eta_\infty^jH$) if necessary, we can assume that $\eta_k^j = \min \{\eta^{j'} : \eta^{j'} \in \eta_k^j H_k \}$ (resp. $\eta_\infty^j = \min \{ \eta^{j'} : \eta^{j'} \in \eta_\infty^j H \}$). Furthermore, by reordering the cosets themselves if necessary, we can assume that $\eta_k^j \leq \eta_k^{j'}$ (resp. $\eta_\infty^j \leq \eta_\infty^{j'}$) for $j \leq j'$. Note that, since $\eta^1 = e$, we always have that $\eta_k^1 H_k = H_k$. It is direct to see that this enumeration enjoys the following properties:
    \begin{align}
        &\text{For each } k \text{, we have that } \eta^j \in \eta_k^{j'(k)}H_k \text{ for some } j'(k) \leq j. \label{eq: enumeration property 1} \\
        &\text{If } \eta^j \leq \eta_k^{j'} \text{, then } \eta^j \in \eta_k^1H_k \cup ... \cup \eta_k^{j'} H_k. \label{eq: enumeration property 2}
    \end{align}

    The proof of the lower semicontinuity will follow from the following two results. First, for any $L, J' \in \N$, we have
    \begin{align}
        P \left( F_k, \bigcup_{gH_k \in \pi_1(M)/H_k} gR\right) \geq \sum_{l=1}^L P \left(F_k^l, \bigcup_{j=1}^{J'} \eta_k^j R \right) \label{eq: first result for lsc}
    \end{align}
    for all $k$ sufficiently large (depending on $L, J'$). Second, for any $J \in \N$ there is $J' \in \N$ such that for any $l$
    \begin{align}
        P \left( F_k^l, \bigcup_{j=1}^{J'} \eta_k^j R\right) \geq P \left( F_k^l, \bigcup_{j=1}^J \eta_\infty^j R \right). \label{eq: second result for lsc}
    \end{align}
    for all $k$ sufficiently large (depending on $J$). Indeed, using these two results, we get that for any $L, J \in \N$ and $J'$ given by \eqref{eq: second result for lsc},
    \begin{align*}
        I := \liminf_k P \left( F_k, \bigcup_{gH_k \in \pi_1(M)/H_k} gR \right) &\geq \sum_{l=1}^L \liminf_k P \left( F_k^l,\bigcup_{j=1}^{J'} \eta_k^j R \right) \\
        &\geq \sum_{l=1}^L \liminf_k P \left( F_k^l, \bigcup_{j=1}^J \eta_\infty^j R \right) \\
        &\geq \sum_{l=1}^L P \left(F^l, \bigcup_{j=1}^J \eta_\infty^j R \right)
    \end{align*}
    where the last inequality uses the standard lower semicontinuity of perimeter along the convergence $F_k^l \rightarrow F^l$ (see \eqref{eq: standard lowersemi}). Taking $J \rightarrow \infty$ and then $L \rightarrow \infty$ we get the required lower semicontinuity:
    \begin{align*}
        I \geq \sum_{l=1}^\infty P \left(F^l, \bigcup_{j=1}^\infty \eta_\infty^j R \right) = \sum_{l=1}^\infty P \left( F^l, \bigcup_{gH \in \pi_1(M)/H} gR \right) \geq P \left( F, \bigcup_{gH \in \pi_1(M)/H} gR \right). 
    \end{align*}

    We turn to the proof of \eqref{eq: first result for lsc}. We claim that for distinct $l, l' \in \{1, ..., L\}$ and any $j_1, j_2 \in \{1,..., J'\}$ we have that $\lambda_k^l \eta_k^{j_1} H_k \neq \lambda_k^{l'} \eta_k^{j_2} H_k$ for $k$ sufficiently large (depending on $l, l', j_1, j_2$). Indeed, suppose for a contradiction that we find $l \neq l'$, $j_1, j_2$, and a subsequence in $k$ along which $\lambda_k^l \eta_k ^{j_1} H_k = \lambda_k^{l'} \eta_k^{j_2} H_k$. This implies that $(\lambda_k^l)^{-1} (\lambda_k^{l'}) \in \eta_k^{j_1} (\eta_k^{j_2})^{-1} H_k$ along the subsequence. 

    Now observe that for each $j$, there exists $j'$ large enough such that $\eta_k^j \leq \eta^{j'}$ for all sufficiently large $k$ (depending on $j$). To see this holds, suppose for a contradiction that there is $j$ and a subsequence of $k$ along which $\eta_k^j \geq \eta^k$. In view of \eqref{eq: enumeration property 2}, this implies that $\eta_k^1H_k \cup \eta_k^2 H_k \cup ... \cup \eta_k^j H_k = H_k \cup \eta_k^2H_k \cup ... \cup \eta_k^j H_k \rightarrow \pi_1(M)$. But in this case, \cref{prop: properties of convergence}(viii) gives that $H$ has finite index, which is a contradiction.

    Using this observation, we find $j_1', j_2'$ such that $\eta_k^{j_1} \leq \eta^{j_1'}$ and $\eta_k^{j_2} \leq \eta^{j_2'}$ for large enough $k$. Hence, by applying the pigeonhole principle, we find $j_1'' \leq j_1'$, $j_2'' \leq j_2'$ and a further subsequence of $k$ such that $\eta_k^{j_1} = \eta^{j_1''}$ and $\eta_k^{j_2} = \eta^{j_2''}$. Along this subsequence, we therefore have that $(\lambda_k^l)^{-1} (\lambda_k^{l'}) \in \eta^{j_1''} (\eta^{j_2''})^{-1} H_k$ which contradicts that $(\lambda_k^l)^{-1} (\lambda_k^{l'}) \rightarrow \infty \perp H_k$. This proves the claim.

    The above claim proves that for each $k$ large enough, the map $\{1, ..., L \} \times \{1,...,J'\} \rightarrow \pi_1(M)/H_k$ given by $(l, j) \mapsto \lambda_k^l \eta_k^j H_k$ is injective. So, for $k$ large enough,
    \begin{align*}
        P \left( F_k, \bigcup_{gH_k \in \pi_1(M)/H_k} gR \right) = \sum_{gH_k \in \pi_1(M)/H_k} P(F_k, gR) &\geq \sum_{\substack{1 \leq l \leq L \\ 1 \leq j \leq J'}} P(F_k, \lambda_k^l \eta_k^j R) \\
        &= \sum_{l=1}^L P \left( F_k, \lambda_k^l \bigcup_{j=1}^{J'} \eta_k^j R \right) \\
        &= \sum_{l=1}^L P \left( F_k^l, \bigcup_{j=1}^{J'} \eta_k^j R \right)
    \end{align*}
    where the measure additivity is thanks to $R$ being an open fundamental domain. This proves \eqref{eq: first result for lsc}.

    It remains to prove \eqref{eq: second result for lsc}. We first claim that there is $J' \in \N$ such that, for all $j \in \{1, ..., J \}$ and all $k$, we have $\eta_\infty^j H_k \in \{\eta_k^1H_k, ..., \eta_k^{J'}H_k \}$. Indeed, $\eta^j_\infty = \eta^{j'}$ for some $j'$ and, by \eqref{eq: enumeration property 1}, $\eta^{j'} \in \eta_k^{j''(k)}H_k$ where $j''(k) \leq j'$ for all $k$. Hence $\eta_\infty^j H_k = \eta_k^{j''(k)} H_k \in \{\eta_k^1 H_k , ..., \eta_k^{j'} H_k \}$ for all $k$. Setting $J'$ to be the largest $j'$ induced by some $j \in \{1,...,J\}$ gives the result.

    We now claim that for any distinct $j, j' \in \{1,...,J\}$ we have that $\eta_\infty^j H_k \neq \eta_\infty^{j'} H_k$ for $k$ sufficiently large (depending on $J$). To see this, first recall that, by definition, $(\eta_\infty^j)^{-1} (\eta_\infty^{j'}) \notin H$.  So, since $H_k \rightarrow H$, we have $(\eta_\infty^j)^{-1}(\eta_\infty^{j'}) \notin H_k$ eventually, or equivalently, $\eta_\infty^jH_k \neq \eta_\infty^{j'}H_k$ eventually. It follows that for $k$ sufficiently large, $\eta_\infty^j H_k \neq \eta_\infty^{j'} H_k$ for all distinct $j, j' \in \{1,...,J\}$. This proves the claim.

    The above two claims prove that for each $k$ large enough, the map $\{1,...,J\} \rightarrow \{\eta_k^1 H_k, ..., \eta_k^{J'} H_k \}$ given by $j \mapsto \eta_\infty^j H_k$ is well defined and injective. So, for $k$ large enough and any $l$,
    \begin{align*}
        P \left( F_k^l, \bigcup_{j=1}^{J'} \eta_k^j R \right) = \sum_{j=1}^{J'} P(F_k^l, \eta_k^j R) \geq \sum_{j=1}^J P(F_k^l, \eta_\infty^j R) = P \left( F_k^l, \bigcup_{j=1}^J \eta_\infty^j R \right)
    \end{align*}
    where the measure additivity is thanks to $R$ being an open fundamental domain. This proves \eqref{eq: second result for lsc}.
\end{proof}

We have two corollaries to \cref{thm: compactness of H-domains}. Let $I(H)$ be the infimum of perimeter among all finite perimeter fundamental domains in $\widetilde M(H)$. By \cref{prop: lower bound on perimeter}, we have $I(H)>0$ for $H \neq \pi_1(M)$.
\begin{corollary} \label{cor: existence of perimeter minimising fundamental domains}
    Fix a non-trivial normal covering space $p: \widetilde M(H) \rightarrow M$. There exists a fundamental domain in $\widetilde M(H)$ minimising perimeter among all fundamental domains of finite perimeter in $\widetilde M(H)$.
\end{corollary}
\begin{proof}
    Let $D_k$ be a minimising sequence of fundamental domains in $\widetilde M(H)$. Let $F_k = f_H^{-1}(D_k)$ be the corresponding $H$-domains in $\widetilde M(e)$ (see \cref{prop: lifting fundamental domains to H-domains}). Using \cref{prop: perimeter of fundamental domains and H-domains}, we see that $F_k$ satisfies the assumption of \cref{thm: compactness of H-domains} with $H_k = H$ and so we find an $H$-domain such that, after possibly passing to a subsequence,
    \begin{equation*}
        P \left( F, \bigcup_{gH \in \pi_1(M)/H} gR \right) \leq \liminf_{k \rightarrow \infty} P \left(F_k, \bigcup_{gH \in \pi_1(M)/H} gR \right) =  \liminf_{k \rightarrow \infty} P(D_k) = I(H)
    \end{equation*}
    Then, using again \cref{prop: lifting fundamental domains to H-domains,prop: perimeter of fundamental domains and H-domains}, $D:= f_H(F)$ is a fundamental domain of finite perimeter in $\widetilde M(H)$ such that $P(D) = I(H)$ as required.
\end{proof}

We can also minimise across covers using \cref{thm: compactness of H-domains}.
\begin{corollary} \label{cor: existence of perimeter minimising fundamental domain across covers}
    Let $\Pi$ be a non-empty collection of normal subgroups of $\pi_1(M)$ that is closed with respect to the convergence of \cref{def: subgroup convergence}. Then there is $H \in \Pi$ such that $I(H) \leq I(H')$ for all $H' \in \Pi$.
\end{corollary}
\begin{proof}
    Let $I(\Pi)$ be the infimum of perimeter among all finite perimeter fundamental domains in any $\widetilde M(H)$ with $H \in \Pi$. Let $D_k$ be a minimising sequence, i.e. $P(D_k) \rightarrow I(\Pi)$, and let $\widetilde M(H_k)$ be the cover in which $D_k$ lives. Let $F_k := f_{H_k}^{-1}(D_k)$ be the corresponding $H_k$-domains (see \cref{prop: lifting fundamental domains to H-domains}). By \cref{prop: compactness of subgroups}, after possibly passing to a subsequence, there is a normal subgroup $H$ such that $H_k \rightarrow H$ and furthermore, since $\Pi$ is closed, $H \in \Pi$. By \cref{prop: perimeter of fundamental domains and H-domains}, $F_k$ satisfies the uniform bound \eqref{eq: uniform bound for H-domains}. By \cref{thm: compactness of H-domains}, there is a $H$-domain $F$ of locally finite perimeter such that, after possibly passing to a subsequence,
    \begin{equation*}
        P \left( F, \bigcup_{gH \in \pi_1(M) /H} gR \right) \leq \liminf_{k \rightarrow \infty} P \left(F_k, \bigcup_{gH_k \in \pi_1(M) / H_k} gR \right) = \liminf_{k \rightarrow \infty}P(D_k) = I( \Pi).
    \end{equation*}
    Then, using again \cref{prop: lifting fundamental domains to H-domains,prop: perimeter of fundamental domains and H-domains}, $D := f_H(F)$ is a fundamental domain of finite perimeter in $\widetilde M(H)$ such that $P(D) = I(\Pi)$ as required.
\end{proof}

\section{Properties of the projected boundary of minimising fundamental domains} \label{sec: projected boundary of minimising fundamental domains}
Having established the existence of minimising fundamental domains in all normal covers $p: \widetilde M(H) \rightarrow M$, we now study their projected boundary. Let $D$ be a minimising fundamental domain in $p: \widetilde M(H) \rightarrow M$. Recall from the introduction that $\Sigma_H(D) := \overline{p(\partial^*D)}$ is the \textbf{projected boundary} of $D$. Since $H$ is fixed in this section, we use the shorthand $\Sigma(D) := \Sigma_H(D)$. We show in this section that $\Sigma(D)$ homotopically spans $\Gamma$ modulo $H$, and $\Sigma(D)$ is a $(\M, 0, \infty)$-minimal set with respect to $\Gamma$ when $\widetilde M(H)$ is non-trivial.

We start by noting that the homotopic spanning condition can be reformulated as follows: for $H$ a subgroup of $\pi_1(M)$, a set $\Sigma \subset \R^3$ \textbf{homotopically spans $\Gamma$ modulo $H$} if $A(\Sigma, x_0) \leq H$ for all $x_0 \in M$ (recall $A(\Sigma, x_0)$ is the avoiding group of $\Sigma$ at $x_0$ defined in \cref{subsec: FD and HD definitions existence and correspondence}).

\begin{proposition} \label{prop: projected boundary spans}
    Let $D$ be a minimising fundamental domain in the normal covering space $p: \widetilde M(H) \rightarrow M$. Then $\Sigma(D)$ homotopically spans $\Gamma$ modulo $H$.
\end{proposition}
\begin{proof}
    Let $\Sigma := \Sigma(D)$, let $x_0 \in M$, and let $\widetilde x_0$ be an element of the fiber of $x_0$. Suppose for a contradiction that we can find a loop $\alpha$ based at $x_0$ with $\alpha \cap \Sigma = \emptyset$ such that $\widetilde \alpha_g := \lift(\alpha, g \widetilde x_0)$ is not a closed loop for a(ny) $g \in \deck(p)$. Since $\alpha$ is compact and $\Sigma$ is closed we can find $\delta > 0$ such that the closed $\delta$-tubular neighbourhood of $\alpha$, $\overline{T}(\alpha, \delta)$, is disjoint from $\Sigma$. Then $\overline{T}(\widetilde \alpha_g, \delta)$ is disjoint from $p^{-1}(\Sigma) \supset \partial D$\footnote{Recall that we modify $D$ by $\Haus^3$-measure zero so that $\overline{\partial^* D} = \partial D$, so the inclusion follows directly from $p(\overline{\partial^* D}) \subset \overline{p(\partial^* D)}$.} for all $g$. By property \eqref{eq: fundamental domain 2} of $D$, there is some $g \in \deck(p)$ such that $|D \cap \overline{T}(\widetilde \alpha_g, \delta)| > 0$. Hence, $D \cap \overline{T}(\widetilde \alpha_g, \delta) \subset \mathrm{int}D$. Hence, $\overline{T}(\widetilde \alpha_g, \delta) \subset \mathrm{int}D$. In particular, there is $\e >0$ such that $B(\widetilde \alpha_g(0), \e), B(\widetilde \alpha_g(1), \e) \subset D$ which is a contradiction to property \eqref{eq: fundamental domain 1} of $D$.
\end{proof}

We now turn to the $(\M, 0, \infty)$-minimal property of $\Sigma(D)$. We recall a precise definition below. To this end, we define the \textbf{support} of a closed set $\Sigma$ of finite $\Haus^2$-measure, denoted by $\mathrm{spt}(\Sigma)$, as the support of the measure $\Haus^2 \resmes \Sigma$. More explicitly, $\mathrm{spt}(\Sigma) = \{x \in \Sigma: \Haus^2(\Sigma \cap B(x, r)) > 0 \text{ for all } r>0 \}$. Recall that the support is a closed set. We say a closed set of finite $\Haus^2$-measure is \textbf{reduced} if it equals its support.
\begin{definition} \label{def: Almgren minimal sets}
    A non-empty closed and reduced set $\Sigma \subset \R^3$ of finite $\Haus^2$-measure is $(\M, 0, \infty)$-minimal with respect to $\Gamma$ if
    \begin{equation}
        \Haus^2(\Sigma \cap B) \leq \Haus^2(\varphi(\Sigma \cap B))
    \end{equation}
    whenever $\varphi: \R^3 \rightarrow \R^3$ is Lipschitz, $\varphi = \id$ outside an open ball $B$, $\varphi(B) \subset B$ and $\mathrm{dist} (B, \Gamma) > 0$.
\end{definition}
$\Sigma(D)$ is closed by definition and non-empty when $\widetilde M(H)$ is non-trivial by \cref{prop: lower bound on perimeter}. Hence, it just remains to show that $\Sigma(D)$ has finite area, is reduced, and that local interior Lipschitz deformations of $\Sigma(D)$ do not decrease its area. We prove these facts sequentially in \cref{prop: projected boundary is reduced,prop: closure of projected boundary is projected boundary for minimisers,prop: projected boundary is M 0 infty}. We reiterate that these arguments are common in the literature (see for instance \parencites[Corollary 1 and Question(a)]{choe1989existence}[Theorem 10.2]{brakke1995soap}). We include them for the convenience of the reader.

\begin{proposition} \label{prop: closure of projected boundary is projected boundary for minimisers}.
    Let $D$ be a minimising fundamental domain in the normal covering space $p: \widetilde M \rightarrow M$. Then $\Sigma(D) = p(\partial^*D)$ up to a set of $\Haus^2$-measure zero. In particular $\Haus^2(\Sigma(D)) < \infty$.
\end{proposition}
\begin{proof}
    As discussed in \cref{subsec: relevant theory of sets of finite perimeter}, the rectifiable 2-varifold $v(\partial^* D, 1)$ has density 1 at all $\widetilde x \in \partial^*D$. Since $p(\partial^* D)$ is rectifiable, it follows that the rectifiable varifold $V = v(p(\partial^*D), 1)$ satisfies $\vartheta_V(x) \geq 1$ for all $x \in p(\partial^* D)$\footnote{In fact, in view of \cref{prop: area of projection of reduced boundary}, $\vartheta_V(x) = 1$ for $\Haus^2$-a.e. $x \in p(\partial^* D)$, but we do not need this.}. By \cref{prop: area of projection of reduced boundary}, we also know that $|V|(M) < \infty$.
    
    We now show that $V$ is stationary in $M$. 
    For $t \in [0,1]$, let $\psi_t: M \rightarrow M$ be a smooth one-parameter family of diffeomorphisms with $\psi_0 = \id$ and $\psi_t = \id$ outside $B \subset \subset M$ for all $t$. By lifting the homotopy $\psi_t \circ p$ we find a smooth one-parameter family of diffeomorphisms $\widetilde \psi_t: \widetilde M \rightarrow \widetilde M$ with $\widetilde \psi_0 = \id$, $\widetilde \psi_t = \id$ outside $p^{-1}(B)$ for all $t$, and $p\circ \widetilde \psi_t = \psi_t \circ p$. Let $D_{\psi_t} = \widetilde \psi_t(D)$ which is a fundamental domain of finite perimeter with $\partial^* D_{\psi_t} \approx \widetilde \psi_t(\partial^* D)$ (see \cref{prop: mapping set of finite perimeter by diffeomorphisms}). Because $D$ minimises perimeter, $P(D) \leq P(D_{\psi_t})$. Then, \cref{prop: area of projection of reduced boundary} implies that
    \begin{equation*}
        \Haus^2(p(\partial^* D)) \leq \Haus^2(p(\partial^* D_{\psi_t})) = \Haus^2(p(\widetilde \psi_t(\partial^* D))) = \Haus^2(\psi_t(p(\partial^* D)))    
    \end{equation*}
    from which it follows that $V$ is stationary in $M$.

    We can now apply \cref{prop: density one stationary rectifiable varifolds are almost closed} to conclude that $\overline{p(\partial^* D)} \cap M \approx p(\partial^* D) \cap M$. Furthermore, since $\Haus^2(\Gamma) = 0$, we conclude $\Sigma(D) = \overline{p(\partial^* D)} \approx p(\partial^* D)$ as required.
\end{proof}

\begin{proposition} \label{prop: projected boundary is reduced}
    Let $D$ be a minimising fundamental domain in the normal covering space $p: \widetilde M \rightarrow M$. Then $\Sigma(D)$ is reduced.
\end{proposition}
\begin{proof}
    We know that $\mathrm{spt}(\Sigma) \subset \Sigma$ and so we just need to show that $\Sigma \subset \mathrm{spt}(\Sigma)$. Since $\mathrm{spt(\Sigma)}$ is closed, it is sufficient to show that $p(\partial^* D) \subset \mathrm{spt}(\Sigma)$. Take $x \in p(\partial^* D)$ and let $\widetilde x \in \partial^*D$ be such that $p(\widetilde x) = x$. Since $D$ blows up to a half space around $\widetilde x$, we have that $\Haus^2(\partial^* D \cap B(\widetilde x, r))>0$ for all $r>0$ (see \cite[Corollary 15.8]{maggi2012sets}). For small enough $r$, we have that $p|_{B(\widetilde x, r)}$ is an isometry and so
    \begin{equation*}
        \Haus^2(\partial^* D \cap B(\widetilde x, r)) = \Haus^2(p(\partial^*D \cap B(\widetilde x, r))) \leq \Haus^2(p(\partial^*D) \cap B(x, r)) \leq \Haus^2(\Sigma \cap B(x, r)).
    \end{equation*}
    Therefore, for all small $r$, $\Haus^2(\Sigma \cap B(x, r)) > 0$. Hence, $x \in \mathrm{spt}(\Sigma)$ as required.
\end{proof}

\begin{proposition} \label{prop: projected boundary is M 0 infty}
    Let $D$ be a minimising fundamental domain in a normal covering space
    $p: \widetilde M \rightarrow M$. Let $B \subset M$ be a ball at positive distance from $\Gamma$ and let $\varphi: \R^3 \rightarrow \R^3$ be a Lipschitz map such that $\varphi(B) \subset B$, $\varphi = \id$ outside $B$. Then $\Haus^2(\Sigma(D)) \leq \Haus^2(\varphi(\Sigma(D)))$.
\end{proposition}
\begin{proof}
   We construct a fundamental domain of finite perimeter $D_{\varphi}$ in $\widetilde M$ such that $p(\partial^*D_{\varphi}) \subset \varphi(p(\partial^* D))$ up to sets of $\Haus^2$-measure zero. We then get the result as follows:
    \begin{align*}
        \Haus^2(\Sigma(D)) = \Haus^2(p(\partial^*D)) = (1/2)P(D) \leq (1/2)P(D_{\varphi}) = \Haus^2(p(\partial^*D_{\varphi})) &\leq \Haus^2(\varphi(p(\partial^*D))) \\
        &\leq \Haus^2(\varphi(\Sigma(D)))
    \end{align*}
    where we used \cref{prop: closure of projected boundary is projected boundary for minimisers}, \cref{prop: area of projection of reduced boundary}, minimality of $D$, \cref{prop: area of projection of reduced boundary} again, and the defining property of $D_\varphi$ in that order.

    We start with some preliminaries before constructing $D_{\varphi}$ in \textit{step 1}. $p^{-1}(B)$ can be written as $\cup_{g \in \deck(p)} g \widetilde B$ where $p: \widetilde B \rightarrow B$ is an isometry. $\varphi \circ p: \widetilde M \rightarrow M$ is homotopic to $\id \circ p: \widetilde M \rightarrow M$ via $(\widetilde x, t) \mapsto (1-t)p(\widetilde x) + t \varphi \circ p(\widetilde x)$. Let $\widetilde \varphi: \widetilde M \rightarrow \widetilde M$ be the endpoint of the unique homotopy starting at $\id: \widetilde M \rightarrow \widetilde M$ and lifting the homotopy $\widetilde M \times [0,1] \rightarrow M$, $(\widetilde x, t) \mapsto (1-t)p(\widetilde x) + t \varphi \circ p(\widetilde x)$. We have that $p \circ \widetilde \varphi = \varphi \circ p$. Furthermore, it is not hard to see that $\widetilde \varphi = \id$ outside $p^{-1}(B)$ and $\widetilde \varphi$ maps $g \widetilde B$ to itself for each $g \in \deck(p)$. In particular, it follows that $\widetilde \varphi$ is Lipschitz, that $g^{-1} \circ \widetilde \varphi \circ g = \widetilde \varphi$ for any $g \in \deck(p)$, and that $\widetilde \varphi$ is proper since $\varphi$ is too\footnote{$\varphi$ is proper since, for $K \subset \R^3$ compact, $\varphi^{-1}(K) = \varphi^{-1}(K \cap \overline B) \cup \varphi^{-1}(K \cap (\R^3 \setminus B))$ both of which are compact: the former as a closed subset of the compact set $\overline{B}$ and the latter since $\varphi$ is the identity away from $B$.}. Define
    \begin{equation*}
        \widetilde \varphi_\# \one_D: \widetilde M  \rightarrow [-\infty, \infty]; \quad \widetilde \varphi_\# \one_D (\widetilde y) = \sum_{\widetilde x \in \widetilde \varphi^{-1}(\widetilde y)} \one_D(\widetilde x) \sigma(\widetilde x),
    \end{equation*}
    where $\sigma(\widetilde x) = \mathrm{sign} \det (\nabla \widetilde \varphi(\widetilde x))$. $\widetilde \varphi_\# \one_D$ is well defined away from the set of $\Haus^3$-measure zero $X \cup \widetilde \varphi(Y \cup Z)$. $X$ is the set of points $\widetilde y$ where $\widetilde \varphi^{-1}(\widetilde y)$ is not finite. This set is contained in $\deck(p) \cdot \widetilde B$ as $\widetilde \varphi$ is the identity elsewhere, and using the area formula \cite[Theorem 2.71]{ambrosio2000variation} on each ball $g \widetilde B$ we conclude that $\Haus^3(X) = 0$. $Y$ is the set of points where $\widetilde \varphi$ is not differentiable. By Rademacher's theorem, $0 = \Haus^3(Y) = \Haus^3(\widetilde \varphi(Y))$. $Z$ is the set of points where $\widetilde \varphi$ is differentiable but $\det(\nabla \widetilde \varphi) = 0$. $\Haus^3(\widetilde \varphi(Z)) = 0$ by a Sard type theorem \cite[Lemma 2.73]{ambrosio2000variation}.

    We define the degree of $\widetilde \varphi$ as
    \begin{equation*}
        \deg(\widetilde \varphi) (\widetilde y) = \sum_{\widetilde x \in \widetilde \varphi^{-1}(\widetilde y)} \sigma(\widetilde x)
    \end{equation*}
    which is well defined $\Haus^3$-a.e. as $\deg(\widetilde \varphi) = \widetilde \varphi_\# \one_{\widetilde M}$. By \cite[Corollary 4.1.26]{federer1969geometric}, $\deg(\widetilde \varphi)$ is $\Haus^3$ almost constant and since $\deg(\widetilde \varphi) = 1$ outside $\deck(p) \cdot \widetilde B$, we have that $\deg(\widetilde \varphi) = 1$ $\Haus^3$-a.e.. With this and since $D$ is a fundamental domain, we compute the following important identity: for $\Haus^3$-a.e. $\widetilde y \in \widetilde M$
    \begin{align} \label{eq: Almgren minimal projection 1}
    \begin{split}
        \sum_{g \in \deck(p)} \widetilde \varphi_\# \one_D \circ g^{-1}(\widetilde y) &= \sum_{g \in \deck(p)} \sum_{\widetilde x \in \widetilde \varphi^{-1}(g^{-1}(\widetilde y))} \one_D(\widetilde x) \sigma(\widetilde x) \\
        &= \sum_{g \in \deck(p)} \sum_{\widetilde x \in g^{-1}(\widetilde \varphi^{-1}(\widetilde y))} \one_D(\widetilde x) \sigma (\widetilde x) \\
        &= \sum_{g \in \deck(p)} \sum_{\widetilde x \in \widetilde \varphi^{-1}(\widetilde y)} \one_{gD}(\widetilde x) \sigma(\widetilde x) \\
        &= \sum_{\widetilde x \in \widetilde \varphi^{-1}(\widetilde y)} \sigma(\widetilde x) \sum_{g \in \deck(p)} \one_{gD}(\widetilde x) = \sum_{\widetilde x \in \widetilde \varphi^{-1}(\widetilde y)} \sigma(\widetilde x) = 1.
        \end{split}
    \end{align}
    where in the second line we use the fact $g^{-1} \circ \widetilde \varphi \circ g = \widetilde \varphi$ mentioned above, and in the final line we used that $\widetilde \varphi^{-1}(\widetilde y)$ is finite for $\Haus^3$ almost every $\widetilde y$.
    
    While $\widetilde \varphi_\# \one_D$ is not a characteristic function, it should be seen as a  way to push $D$ forward by $\widetilde \varphi$ in a way that accounts for orientation when $\widetilde \varphi$ generates `folds'. Loyal to this observation, \cite[Theorem 3.16]{ambrosio2000variation} guarantees that, since $\widetilde \varphi$ is proper and Lipschitz, $\widetilde \varphi_\# \one_D \in BV(\widetilde M)$ and that
    \begin{equation} \label{eq: gradient of pushforward is pushforward of gradient}
        |\nabla \widetilde \varphi_\# \one_D| \leq (\mathrm{Lip}(\widetilde \varphi))^2 \widetilde \varphi_\# |\nabla \one_D|,
    \end{equation}
    where $BV(\widetilde M)$ denotes the functions of bounded variation on $\widetilde M$, $\nabla()$ is the vector-valued Radon measure representing the distributional derivative of $()$, $|\nabla()|$ denotes the total variation measure of $\nabla()$, and $\widetilde \varphi_\# |\nabla \one_D|$ denotes the measure pushforward $\widetilde \varphi_\# |\nabla \one_D|(A) = |\nabla \one_D|(\widetilde \varphi^{-1}(A))$.\footnote{Strictly speaking \cite[Theorem 3.16]{ambrosio2000variation} requires that $\one_D$ be in $L^1(\widetilde M)$ which is not the case. However, because $\widetilde \varphi$ is the identity outside $p^{-1}(B)$, one can simply apply \cite[Theorem 3.16]{ambrosio2000variation} on each ball $g \widetilde B$ to get the global estimate.}
    
    For ease of notation, let $f = \widetilde \varphi_\# \one_D$. We consider the two sets $\{f \geq 1\}$, $\{f \leq -1 \}$ and note that, by the coarea formula for BV functions \cite[Theorem 3.40]{ambrosio2000variation}, they both have finite perimeter and
    \begin{equation} \label{eq: (f>1) and (f <1) perimeter contained in jump set of f}
        |\mu_{\{f \geq 1\}}| + |\mu_{\{f \leq -1\}}| \leq |\nabla \widetilde \varphi_\# \one_D|.
    \end{equation}
    Observe that
    \begin{equation} \label{eq: (f <-1) contained in tilde B and translates}
        \deck(p) \cdot \{f \leq -1\} \subset \deck(p) \cdot \widetilde B
    \end{equation}
    where the shorthand $\deck(p) \cdot A$ denotes
    \begin{equation*}
        \deck(p) \cdot A := \bigcup_{g \in \deck(p)} gA.
    \end{equation*}
    We will use this shorthand repeatedly in the proof of this proposition.

    \textit{Step 1:} We construct $D_{\varphi}$ and show it is a fundamental domain. It follows from \eqref{eq: Almgren minimal projection 1} that
    \begin{equation} \label{eq: (f>1) covers tilde M}
        |\widetilde M \setminus \deck(p) \cdot \{f \geq 1\}| = 0
    \end{equation}
    and so it remains to take care of the intersections of $\{f \geq 1\}$ with its translations $g\{f \geq 1\}$. Observe that up to sets of $\Haus^3$-measure zero, we have, thanks again to \eqref{eq: Almgren minimal projection 1}, that for $g \neq \id$
    \begin{equation} \label{eq: intersections of (f > 1) contained in (f < -1)}
        \{f \geq 1\} \cap g\{f \geq 1\} \subset \deck(p) \cdot \{f \leq -1\}
    \end{equation}

    Let $D_\varphi = \left(\{f \geq 1\} \setminus (\deck(p) \cdot \{f \leq -1\}) \right) \cup (\widetilde B \cap (\deck(p) \cdot \{f \leq -1\}))$. We now prove $D_\varphi$ is a fundamental domain. We have that, up to sets of $\Haus^3$-measure zero,
    \begin{align*}
        \bigcup_{h \in \deck(p)} h D_\varphi &= \bigcup_{h \in \deck(p)} h(\{f \geq 1\} \setminus \deck(p) \cdot \{f \leq -1\}) \bigcup_{h \in \deck(p)} h \widetilde B \cap (\deck(p) \cdot \{f \leq -1\}) \\
        &= \left( \bigcup_{h \in \deck(p)} h\{f \geq 1\} \setminus \deck(p) \cdot \{f \leq -1\} \right) \cup ((\deck(p) \cdot \widetilde B) \cap (\deck(p) \cdot \{f \leq -1\})) \\
        &=((\deck(p) \cdot \{f \geq 1\}) \setminus (\deck(p) \cdot \{f \leq -1\})) \cup (\deck(p) \cdot \{f \leq -1\}) \\
        &=\deck(p) \cdot \{f \geq 1\}
    \end{align*}
    where we used \eqref{eq: (f <-1) contained in tilde B and translates} on the third line, and \eqref{eq: (f>1) covers tilde M} on the final line. Therefore, $|\widetilde M \setminus \deck(p) \cdot D_\varphi|= 0$ which proves \eqref{eq: fundamental domain 2}. To see \eqref{eq: fundamental domain 1}, let $g \in \deck(p) \setminus \{\id\}$ and compute that, up to sets of $\Haus^3$-measure zero,
    \begin{align*}
        D_\varphi \cap g D_\varphi &\subset ((\{f \geq 1\} \cap g\{f \geq 1\}) \setminus (\deck(p) \cdot \{f \leq -1\})) \cup (\widetilde B \cap g \widetilde B) \\
        &\subset \emptyset \cup \emptyset
    \end{align*}
    where we used \eqref{eq: intersections of (f > 1) contained in (f < -1)}. This proves \eqref{eq: fundamental domain 1}.
    
    \textit{Step 2:} We show $D_\varphi$ has finite perimeter. For ease of notation, let $S = \{f \geq 1\}$, $I = \deck(p) \cdot \{f \leq -1\}$ so that $D_\varphi = (S \setminus I) \cup (\widetilde B \cap I)$. Recall that $\{f \geq 1\}$ and $\{f \leq -1\}$ have finite perimeter, and so $S$ has finite perimeter and $I$ has finite perimeter in each ball $g \widetilde B$. By applying the standard rules for distributing perimeter over set operations, we have
    \begin{equation*}
        P(D_{\varphi}) \leq P(S) + P(I, S^{(1)}) + P(I \cap \widetilde B).
    \end{equation*}
    Therefore, it only remains to show that $P(I, S^{(1)})$ is finite too. We claim that
    \begin{align}
        &\partial^* I \cap \widetilde B \supset \bigcup_{g \in \deck(p)} \partial^*I \cap \widetilde B \cap gS^{(1)}, \label{111111} \\
        &\left( \partial ^*I \cap \widetilde B \cap g S^{(1)} \right) \cap \left(\partial^* I \cap \widetilde B \cap g'S^{(1)}\right) = \emptyset \text{ for } g \neq g', \label{22222}
    \end{align}
    where \eqref{22222} holds up to sets of $\Haus^2$-measure zero. If we have these, then
    \begin{align*}
        P(I, S^{(1)}) = \Haus^2(\partial^*I \cap S^{(1)}) = \Haus^2 \left(\partial^*I \cap S^{(1)} \cap \bigcup_{g \in \deck(p)} g \widetilde B \right) &= \sum_{g \in  \deck(p)} \Haus^2 \left( \partial^* I \cap S^{(1)} \cap g \widetilde B \right) \\
        &= \sum_{g \in \deck(p)} \Haus^2 \left( \partial^*I \cap g S^{(1)} \cap \widetilde B \right) \\
         \text{ By \eqref{22222}} \qquad &= \Haus^2 \left(\bigcup_{g \in \deck(p)} \partial^* I \cap g S^{(1)} \cap \widetilde B \right)\\
        \text{ By \eqref{111111}} \qquad &\leq \Haus^2(\partial^* I \cap \widetilde B) = P(I, \widetilde B)
    \end{align*}
    and hence, $P(I, S^{(1)})$ is finite as required.
    
    That \eqref{111111} holds is obvious. We now prove \eqref{22222}. Suppose the intersection is non-empty and let $\widetilde x$ be in it. For any $\e > 0$ we can find $r > 0$ such that
    \begin{equation*}
        |gS^{(1)} \cap B(\widetilde x, r)|, |g' S^{(1)} \cap B(\widetilde x, r)| \geq (1  - \e) |B(\widetilde x, r)|.
    \end{equation*}
    Hence, there is a subset $E$ of $B(\widetilde x, r)$ with $|E| \geq (1- 2\e) |B(\widetilde x, r)|$ such that $f \circ g^{-1} + f \circ (g')^{-1} \geq 2$ on $E$. Hence, by \eqref{eq: Almgren minimal projection 1}, on a set $E'$ with $|E' \triangle E| = 0$, we have that
    \begin{equation*}
        \sum_{h \in \deck(p) \setminus \{g, g'\}} f \circ h^{-1} \leq -1
    \end{equation*}
    on $E'$. Observe that $E' \subset I$ and so $|I \cap B(\widetilde x, r)| \geq (1-2\e)|B(\widetilde x, r)|$. Taking $\e \rightarrow 0$, it follows that $\widetilde x \in I^{(1)}$. So $\widetilde x \in \partial^*I \cap I^{(1)}$ which is impossible up to a set of $\Haus^2$-measure zero.

    \textit{Step 3:} We conclude the proof by proving the key property of $D_\varphi$:  $p(\partial^* D_\varphi) \subset \varphi (p(\partial^*D))$ up to a set of $\Haus^2$-measure zero. We have, up to sets of $\Haus^2$-measure zero,
    \begin{equation} \label{eq: reduced boundary of D_varphi}
        \partial^*D_\varphi \subset \partial^*S \cup \partial^* I.
    \end{equation}
    Since for any Borel set $A$, $P(I, A) \leq \sum_{g \in \deck(p)}P(g \{f \leq -1\}, A)$, we get, up to sets of $\Haus^2$-measure zero
    \begin{equation*}
        \partial^* I \subset \bigcup_{g \in \deck(p)} \partial^* g \{f \leq -1\}
    \end{equation*}
    by plugging in $A = \widetilde M \setminus \cup_{g \in \deck(p)} \partial^* g\{f \leq -1\}$. Hence, up to sets of $\Haus^2$-measure zero, \eqref{eq: reduced boundary of D_varphi} expands to
    \begin{equation*}
        \partial^* D_\varphi \subset \partial^* \{f \geq 1\} \bigcup_{g \in \deck(p)} \partial^* g\{f \leq -1\}
    \end{equation*}
    and so we get, again up to $\Haus^2$-measure zero,
    \begin{equation*}
        p(\partial^*D_\varphi) \subset p(\partial^*\{f \geq 1\} \cup \partial^*\{f \leq -1\}).
    \end{equation*}
    Furthermore, using \eqref{eq: gradient of pushforward is pushforward of gradient} and \eqref{eq: (f>1) and (f <1) perimeter contained in jump set of f}, we have, for a Borel set $A$,
    \begin{align*}
        \Haus^2((\partial^*\{f \geq 1\} \cup \partial^* \{f \leq -1\}) \cap A) \leq \mu_{\{f \geq 1\}}(A) + \mu_{\{f \leq -1\}}(A) &\leq |\nabla \widetilde \varphi_\# \one_D|(A) \\
        &\leq (\mathrm{Lip}(\widetilde \varphi))^2 \widetilde \varphi_\#|\nabla \one_D|(A) \\
        &= (\mathrm{Lip}(\widetilde \varphi))^2 \Haus^2(\partial^*D \cap \widetilde \varphi^{-1}(A)).
    \end{align*}
    Setting $A = \widetilde M \setminus \widetilde \varphi(\partial^*D)$ we discover that
    \begin{equation*}
        \Haus^2((\partial^* \{f \geq 1\}  \cup \partial^* \{f \leq -1\}) \setminus \widetilde \varphi(\partial^*D)) \leq (\mathrm{Lip}(\widetilde \varphi))^2 \Haus^2(\emptyset) = 0
    \end{equation*}
    and so $p(\partial^* D_\varphi) \subset p(\widetilde \varphi(\partial^* D)) = \varphi(p(\partial^*D))$ up to a set of $\Haus^2$-measure zero.
\end{proof}

\section{Proof of main results} \label{sec: proof of main theorems}
Using the results from the previous sections, we now prove the results stated in the introduction. For convenience, we also recall the statements here.

\first*
\begin{proof}
    \cref{cor: existence of perimeter minimising fundamental domains} gives us the existence of a perimeter minimising fundamental domain $D$ in $\widetilde M(H)$. Let $\Sigma(D) := \overline{p(\partial^* D)}$. \cref{prop: lower bound on perimeter,prop: area of projection of reduced boundary,prop: closure of projected boundary is projected boundary for minimisers} give that $\Haus^2(\Sigma(D)) = \Haus^2(p(\partial^*D)) = P(D) > 0$. \cref{prop: projected boundary spans} gives that $\Sigma(D)$ homotopically spans $\Gamma$ modulo $H$. \cref{prop: lower bound on perimeter,prop: projected boundary is reduced,prop: closure of projected boundary is projected boundary for minimisers,prop: projected boundary is M 0 infty} give that $\Sigma(D)$ is $(\M, 0, \infty)$-minimal with respect to $\Gamma$.

    Let $\Sigma'$ be a $(\M, 0, \infty)$-minimal set with respect to $\Gamma$ that homotopically spans $\Gamma$ modulo $H$. By the regularity results of Taylor \cite{taylor1976structure}, $\Sigma'$ satisfies the assumptions of \cref{prop: lifting thin surface to fundamental domain of finite perimeter}(ii). Hence, $D(\Sigma', H)$ is a fundamental domain of finite perimeter in $\widetilde M(H)$ with $P(D(\Sigma', H)) \leq 2 \Haus^2(\Sigma')$. We therefore have
    \begin{equation*}
        \Haus^2(\Sigma') \geq (1/2) P(D(\Sigma', H)) \geq (1/2) P(D)  = \Haus^2(\Sigma(D))
    \end{equation*}
    where the second inequality is the minimality of $D$, and the equality is \cref{prop: closure of projected boundary is projected boundary for minimisers}. This concludes the proof.
\end{proof}

\second*
\begin{proof}
    To get the first result, in view of \cref{cor: existence of perimeter minimising fundamental domain across covers}, we need to check that being fully spanning is closed under convergence of subgroups. Let $H_k \rightarrow H$, where $H_k$ is a sequence of fully spanning subgroups. Suppose that $H$ contains a meridian $a$ of $\pi_1(M)$. Since $\lim_k H_k \geq H$ (see \cref{def: subgroup convergence}), we have by definition that $a \in H_k$ for $k$ large enough. But then $H_k$ is not fully spanning for large $k$, which is a contradiction. Therefore, $H$ is fully spanning as required.
\end{proof}

\third*
\begin{proof}
    Let $D$ and $p: \widetilde M(H) \rightarrow M$ be the fundamental domain and the cover of \cref{ithm: comparison of different covers}, and set $\Sigma = \Sigma_H(D)$. $\Sigma$ is $(\M, 0, \infty)$-minimal with respect to $\Gamma$ and homotopically spans $\Gamma$ modulo $H$ by \cref{ithm: minimising fundamental domains and projected boundary}. Therefore, it only remains to check the minimising property of $\Sigma$. Let $\Sigma'$ be a $(\M, 0, \infty)$-minimal set with respect to $\Gamma$ that homotopically spans $\Gamma$ modulo some normal fully spanning $H'$. By the regularity results of Taylor \cite{taylor1976structure}, $\Sigma'$ satisfies the assumptions of \cref{prop: lifting thin surface to fundamental domain of finite perimeter}(ii). Hence, $D(\Sigma', H')$ is a fundamental domain of finite perimeter in $\widetilde M(H')$ with $P(D(\Sigma', H')) \leq 2 \Haus^2(\Sigma')$. We have
    \begin{equation*}
        \Haus^2(\Sigma') \geq (1/2)P(D(\Sigma', H')) \geq (1/2) P(D) = \Haus^2(\Sigma)
    \end{equation*}
    where the second inequality is the minimality of $D$, and the final equality is \cref{prop: closure of projected boundary is projected boundary for minimisers}. This proves the corollary.
\end{proof}

We now turn to \cref{iprop: projected boundary from specific covers}. We start with some preliminaries. Recall that for this proposition we assume that $\Gamma$ has a single connected component, i.e. $\Gamma = \gamma_1$. Hence, we have
\begin{equation*}
    H_1(\R^3 \setminus \Gamma; \Z_q) \cong \Z_q
\end{equation*}
where the generator is the homology class carried by the loop $\xi_1$ defined in \cref{subsec: setup}. We denote this class by $\langle \xi_1 \rangle$.

For any $q \in \Z$, a closed loop $\alpha$ defines a class in $H_1(M; \Z_q)$, say $a \langle \xi_1 \rangle$. We define the $q$-linking number of $\alpha$ with $\Gamma$, denoted $\link_q(\alpha, \Gamma)$, as $a \in \Z_q$. More formally, $\link_q(\cdot, \Gamma): \pi_1(M) \rightarrow \Z_q$ is the map $\pi_1(M) \rightarrow \Z \rightarrow \Z_q$ given by the composition of the abelianisation $\pi_1(M) \rightarrow \Z$ and the homomorphism sending $1 \in \Z \mapsto 1 \in \Z_q$. These linking numbers readily define normal subgroups of $\pi_1(M)$ given by
\begin{equation} \label{eq: q subgroups}
    K_q = \{[\alpha] \in \pi_1(M): \link_q(\alpha, \Gamma) = 0 \}.
\end{equation}
These satisfy that $\pi_1(M) / K_q \cong \Z_q$ by the first isomorphism theorem. The idea of defining these subgroups comes from \cite{amato2017constrained}.

\fourth*

\begin{proof}[Idea of proof]
    Since this result is tangential to the main content of the paper, we only provide an idea of the proof to illustrate the main ideas and forgo the technical details.

    We write $\Sigma(D) := \Sigma_H(D)$ and let the subgroup $H = K_2, K_3$ be understood from context. In either of the two cases, suppose there is an interior point $x$ of $\Sigma(D)$ at which $\Sigma(D)$ is not smooth. Since $\Sigma(D)$ is $(\M, 0, \infty)$-minimal, Taylor's regularity theorem \cite{taylor1976structure} gives that, for $B$ a small enough ball around $x$, $B \setminus \Sigma(D)$ has three or four connected components depending on whether $x$ is a triple junction or tetrahedral singularity point. This means that the cover must have at least three or four deck transformations. In the case of $\widetilde M(K_2)$, both these possibilities give a contradiction since $\widetilde M(K_2)$ is a double cover. Hence $\Sigma(D)$ is smooth. $\widetilde M(K_3)$ is a triple cover and so cannot have four deck transformations. Hence, $B \setminus \Sigma(D)$ cannot have four connected components, and therefore $\Sigma(D)$ has no interior tetrahedral singularities.

    We now discuss the boundary regularity of $\Sigma := \Sigma(D)$ in the case $H = K_2$. Fix a point on $\Gamma$, and assume for simplicity this point is 0. Since $\Sigma(D)$ is smooth and stable in the interior, a standard blowup argument combined with curvature estimates \cite{schoen1984estimates}, \cite[Lemma 2.4 and Theorem 2.10]{colding2011course} gives that a blowup sequence $\Sigma_k$ of $\Sigma$ converges smoothly to a union of half planes, denoted $\Sigma_\infty$, meeting along $T_0 \Gamma$. $\Sigma_\infty \cap B(0,1)$ has to be the projection of the boundary of a perimeter minimising fundamental domain in the double cover of $B(0,1) \setminus T_0 \Gamma$. It is then easy to check by a simple cut and paste argument that any $\Sigma_\infty$ consisting of more than two distinct half plane would not be the projection of the boundary of a perimeter minimiser (see \cref{fig: cut and paste}). Furthermore, the boundary of a fundamental domain in the double cover of $B(0,1) \setminus T_0 \Gamma$ cannot project to two half planes. Hence $\Sigma_\infty$ is a single half plane. After checking that the convergence of $\Sigma_k$ to $\Sigma_\infty$ occurs without multiplicity, one concludes by \cite[\nopp 4]{allard1975boundary} that $x$ is a smooth boundary point of $\Sigma$. Therefore $\Sigma$ has smooth boundary.

    It remains to check the minimising property of $\Sigma(D)$. We carry this out in the case $H = K_3$ as the case $H= K_2$ follows the same strategy and is simpler. Take a surface $\Sigma'$ as described in the proposition, that is, $\Sigma'$ is compact with smooth boundary $\Gamma$, and has smooth and oriented interior away from finitely many smooth triple junction line singularities, where the orientation induced on the triple junction by each smooth piece agrees. Such a surface can be triangulated by a 2-dimensional simplicial complex $K$ whose boundary is $\Gamma$ together with any of the triple junction lines (see \cite[Proposition 4.5 and ensuing discussion]{gallier2013guide}). Furthermore, because of the orientability condition on the triple junction, we can guarantee that the boundary of $K$ has multiplicity 3 on the triple junction lines, whereas it has multiplicity one on $\Gamma$. Hence $\Sigma' \in H_2(\R^3, \Gamma; \Z_3)$, where $H_2(\R^3, \Gamma; \Z_3)$ denotes the second homology of $\R^3$ relative to $\Gamma$ in $\Z_3$ coefficients. By intersection theory (or more precisely, by Poincaré Duality - see \cites[Theorem 8.3]{bredon1993topology}{hatcher2005algebraic}), it is now easy to compute the avoiding group of $\Sigma'$: $\alpha \cap \Sigma' = \emptyset$ implies that $\link_3(\alpha, \Gamma) = 0$. Therefore, $A(\Sigma') \leq \{\alpha: \link_3(\alpha, \Gamma) = 0 \} = K_3$. By \cref{prop: lifting thin surface to fundamental domain of finite perimeter}, there is a fundamental domain $D(\Sigma', K_3)$ in $\widetilde M(K_3)$ with $P(D(\Sigma', K_3)) \leq 2 \Haus^2(\Sigma')$. Therefore, by the minimality of $D$ and \cref{prop: closure of projected boundary is projected boundary for minimisers}, $\Haus^2(\Sigma(D)) \leq \Haus^2(\Sigma')$.
\end{proof}

\begin{figure}
    \centering
    \resizebox{0.3\textwidth}{!}{%
        \begin{tikzpicture}

            \useasboundingbox (-5, -5) rectangle (5, 5);
            \draw (-5, -5) rectangle (5, 5); 
        
            \draw[thick] (-5, 0) -- (5, 0);
            \draw[thin] (0, -5) -- (0, 5);
        
            \coordinate (TopLeft)  at (-2.5,  2.5);
            \coordinate (TopRight) at ( 2.5,  2.5);
            \coordinate (BotLeft)  at (-2.5, -2.5);
            \coordinate (BotRight) at ( 2.5, -2.5);

            \tikzset{shade_style/.style={fill=blue!20, opacity=0.6}}
            \tikzset{yellow_shade/.style={fill=blue!20, opacity=0.6}}
            \tikzset{white_shade/.style={fill=white!100, opacity=1}}

            \fill[yellow_shade] (0.15, -4.85) rectangle (4.85, -0.15);
            \fill[white_shade] (BotRight) -- (0.15, -2.5) -- (0.15, -0.15) -- (2.5, -0.15) -- cycle;
            \fill[white_shade] (BotRight) -- (0.14, -3.87) -- (0.14, -4.87) -- (4.87, -4.87) -- (4.87, -3.87) -- cycle;
            \fill[white_shade] (BotRight) -- ++(90:1.3) -- ++(300:2.2517) -- cycle;

            \fill[shade_style] (TopLeft) -- (-4.85, 2.5) -- (-4.85, 4.85) -- (-2.5, 4.85) -- cycle;
            \fill[shade_style] (TopLeft) -- (-4.86, 1.13) -- (-4.86, 0.13) -- (-0.13, 0.13) -- (-0.13, 1.13) -- cycle;
            \fill[shade_style] (TopRight) -- (2.5, 4.85) -- (4.85, 4.85) -- (4.86, 1.13) -- cycle;
            \fill[shade_style] (TopRight) -- (0.15, 2.5) -- (0.13, 2.5) -- (0.13, 1.13) -- cycle;

            \fill[shade_style] (BotLeft) -- (-4.85, -2.5) -- (-4.85, -0.15) -- (-2.5, -0.15) -- cycle;
            \fill[shade_style] (BotLeft) -- (-4.86, -3.87) -- (-4.86, -4.87) -- (-0.13, -4.87) -- (-0.13, -3.87) -- cycle;
            \fill[shade_style] (BotLeft) -- ++(90:1.3) -- ++(300:2.2517) -- cycle;

            \foreach \coord in {TopLeft, TopRight, BotLeft, BotRight} {
                \fill (\coord) circle (2pt);
                \draw[dashed] (\coord) -- ++(-2.35, 0);
            }

            \draw (TopLeft) -- ++(90:2.35);      
            \draw (TopLeft) -- ++(210:2.73);     
            \draw (TopLeft) -- ++(330:2.73);     
        
            \draw (TopRight) -- ++(90:2.35);     
            \draw (TopRight) -- ++(210:2.73);    
            \draw (TopRight) -- ++(330:2.73);

            \draw (BotLeft) ++(90:1.3) -- ++(90:1.05);
            \draw (BotLeft) ++(330:1.3) -- ++(330:1.43);
            \draw (BotLeft) -- ++(210:2.73);
            \draw (BotLeft) ++(90:1.3) -- ++(300:2.2517);
        
            \draw (BotRight) ++(90:1.3) -- ++(90:1.05);
            \draw (BotRight) ++(330:1.3) -- ++(330:1.43);
            \draw (BotRight) -- ++(210:2.73);
            \draw (BotRight) ++(90:1.3) -- ++(300:2.2517);
        
        \end{tikzpicture}
    }
    \caption{This is a simplified depiction of the cut and paste argument mentioned in the proof of \cref{iprop: projected boundary from specific covers}, where we only depict a single slice of $B(0,1) \setminus T_0 \Gamma$ transverse to $T_0 \Gamma$. Each row is a depiction of the double cover of $B(0,1) \setminus T_0 \Gamma$ obtained by gluing across the dashed plane. In the top row, we see a fundamental domain (shaded in blue) whose projected boundary consists of three half planes. The bottom row depicts a fundamental domain in the double cover that has less perimeter than the conical fundamental domain of the top row. This `shows' that a conical perimeter minimising fundamental domain cannot have a projected boundary consisting of more than two half planes.}
    \label{fig: cut and paste}
\end{figure}
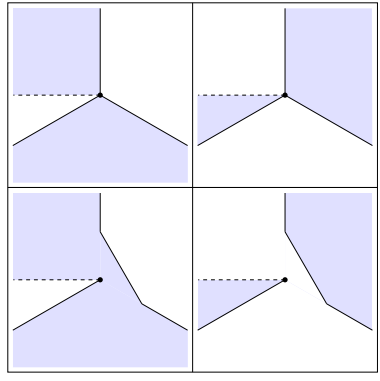

\fifth*

\begin{proof}
    From the Wirtinger presentation of a knot group and the fact $\Gamma$ has a single component, we know that any Wirtinger generator normally generates $\pi_1(M)$.

    Suppose for a contradiction that there is a perimeter minimising fundamental domain $D$ in $\widetilde M(H)$ such that $\Sigma_H(D)$ does not fully wet $\Gamma$. By \cref{prop: projected boundary spans}, $\Sigma_H(D)$ homotopically spans $\Gamma$ modulo $H$. This means that there is a Wirtinger generator in $H$. Since $H$ is normal, it contains the subgroup normally generated by this Wirtinger generator, and hence $H = \pi_1(M)$ by the observation at the start of the proof. This contradicts that $H$ is proper.
\end{proof}

\sixth*
\begin{proof}
    The proof follows that of \cref{ithm: comparison of different covers} and \cref{icor: comparison of different surfaces}. Note $\Lambda := \{H \triangleleft \pi_1(M) : H \subset \pi_1(M) \setminus \mathcal{C} \}$ is non-empty because $e \notin \mathcal{C}$. In view of \cref{cor: existence of perimeter minimising fundamental domain across covers}, upon checking that $\Lambda$ is closed under the convergence of \cref{def: subgroup convergence}, we get the existence of a normal subgroup $H \subset \pi_1(M) \setminus \mathcal{C}$ and a fundamental domain $D$ in $\widetilde M(H)$ such that $D$ has least perimeter among any fundamental domain in any cover $\widetilde M(H')$ with $H' \subset \pi_1(M) \setminus \mathcal{C}$. The proof is then concluded exactly as in \cref{icor: comparison of different surfaces}.

    Therefore, it remains to prove that $\Lambda$ is closed. Let $H_k$ be a sequence in $\Lambda$ with $H_k \rightarrow H$. Suppose $H \not\subset \pi_1(M) \setminus \mathcal{C}$ so that there is $h \in H \cap \mathcal{C}$. Then, by \cref{def: subgroup convergence}, $h \in H_k$ for sufficiently large $k$ which contradicts that $H_k \in \Lambda$. Therefore, $H \in \Lambda$ as required.
\end{proof}

\appendix
\section{Relative isoperimetric inequality on solid tori}
Let $\xi: [0,2 \pi r] \rightarrow \R^3$ be the arclength parametrisation of a circle of radius $r$, let $\alpha = \xi|_{[s_1, s_2]}$ with $[s_1, s_2] \subsetneq[0,1]$, and let $T$ be a $\delta$-regular tubular neighbourhood of $\alpha$, i.e. every point in $T$ may be written uniquely as $\alpha(t) + a N(t) + b B(t)$, where $(\alpha', N, B)$ is the Frenet frame along $\alpha$ and $a^2 + b^2 \leq \delta^2$. We prove the following relative isoperimetric inequality by adapting \cite[Proposition 12.37]{maggi2012sets}.
\begin{proposition} \label{prop: relative isoperimetric inequality of partial cylinders}
    Let $E$ be a set of finite perimeter in $\R^3$ with $|E \cap T| \leq \sigma |T|$ for some $\sigma \in (0,1)$. Then there exists $c = c(\sigma) > 0$ such that
    \begin{equation*}
        c|E \cap T|^{2/3} \leq P(E; T^{(1)}).
    \end{equation*}
\end{proposition}
\begin{proof}
    Let $t_* = (s_1 + s_2)/2$ and let $l = s_2 - s_1$. Define the following vector fields on $T$:
    \begin{align*}
        X(t, a, b) &= (t- t_*) \alpha'(t), \\
        Y(t, a, b) &= a N(t) +b B(t).
    \end{align*}
    A computation shows that $\mathrm{div}(X) = r/(r-a)$ and $\mathrm{div}(Y) = 2-a/(r-a)$. Hence, $\mathrm{div} (X+Y) = 3$. By the divergence theorem,
    \begin{align*}
        3|E \cap T| &= \int_{\R^3} (X+Y) \cdot \diff\mu_{E \cap T} \\
        &= \int_{T^{(1)}} (X+Y) \cdot \nu_E \diff |\mu_E| + \int_{E^{(1)}} (X+Y) \cdot \nu_T \diff |\mu_T| + \int_{\{\nu_E = \nu_T\}} (X+Y) \cdot \nu_T \diff \Haus^2 \\
        &\geq -(l/2 + \delta)  P(E; T^{(1)}) + \min\{l/2, \delta\} P(T; E^{(1)}) + \min\{l/2, \delta\} \Haus^2(\{\nu_E = \nu_T\})
    \end{align*}
    where, in the second equality, we used
    \begin{equation*}
        \mu_{E \cap T} = \mu_E \resmes T^{(1)} + \mu_T \resmes E^{(1)} + \nu_T \Haus^2 \resmes \{\nu_E = \nu_T\}
    \end{equation*}
    which can be found in \cite[Theorem 16.3]{maggi2012sets}. Rearranging, adding $\min\{\pi r, \delta\} P(E; T^{(1)})$ to both sides, and dividing by $|E \cap T|^{2/3}$, we get
    \begin{equation} \label{eq: lower bound on perimeter 1}
        3|E\cap T|^{1/3} + A_1 \frac{P(E; T^{(1)})}{|E \cap T|^{2/3}} \geq A_2 \frac{P(E \cap T)}{|E \cap T|^{2/3}}
    \end{equation}
    where $A_1 = (l/2 + \delta) + \min\{l/2, \delta\}$ and $A_2 = \min\{l/2, \delta \}$. Suppose for a contradiction that there is a sequence $E_k$ of sets of finite perimeter in $\R^3$ with $|E_k \cap T| \leq \sigma |T|$ such that
    \begin{equation} \label{eq: lower bound on perimeter 2}
        (1/k)|E_k \cap T|^{2/3} \geq P(E_k; T^{(1)}).
    \end{equation}

    We deduce that $P(E_k; T^{(1)}) \rightarrow 0$ as $k \rightarrow \infty$. This implies that $|E_k \cap T| \rightarrow 0$ as $k \rightarrow \infty$. Indeed, $E_k \cap T$ has uniformly bounded perimeter and is contained in $T$, and so we can find a set of finite perimeter $F \subset T$ such that, after passing to a subsequence, $E_k \cap T \rightarrow F$. Furthermore, by lower semicontinuity of perimeter and the fact that $T^{(1)}$ is open, $P(F; T^{(1)}) = 0$. Hence, since $|F| = \lim_{k \rightarrow \infty} |E_k \cap T| \leq \sigma |T| < |T|$, we must have that $|F| = 0$ and so $|E_k \cap T| \rightarrow 0$.

    $|E_k \cap T| \rightarrow 0$ and \eqref{eq: lower bound on perimeter 2} respectively imply that the first term and second term on the LHS of \eqref{eq: lower bound on perimeter 1} (with $E$ replaced with $E_k$) go to zero as $k \rightarrow \infty$. By the classical isoperimetric inequality, the term on the RHS of \eqref{eq: lower bound on perimeter 1} (with $E$ replaced with $E_k$) is bounded away from zero uniformly in $k$, which contradicts the behaviour of the LHS. This proves the relative isoperimetric inequality.
\end{proof}

\printbibliography

@book{hatcher2005algebraic,
author = {Hatcher, Allen},
address = {Cambridge},
booktitle = {Algebraic topology},
isbn = {9780521795401},
language = {eng},
publisher = {Cambridge University Press},
title = {Algebraic topology },
year = {2001},
}

@article{parks1992soap,
author = {Parks, Harold R.},
title = {Soap-film-like minimal surfaces spanning knots},
volume = {2},
year = {1992},
issn = {1050-6926},
journal = {The Journal of geometric analysis},
language = {eng},
number = {3},
pages = {267-290},
}

@article{taylor1976structure,
number = {3},
pages = {489-539},
publisher = {Princeton University},
title = {The Structure of Singularities in Soap-Bubble-Like and Soap-Film-Like Minimal Surfaces},
volume = {103},
year = {1976},
language = {eng},
author = {Taylor, Jean E.},
issn = {0003-486X},
journal = {Annals of mathematics},
}

@article{drachman1998soap,
  title={Soap films bounded by non-closed curves},
  author={Drachman, Jordan and White, Brian},
  journal={The Journal of Geometric Analysis},
  volume={8},
  pages={239--250},
  year={1998},
  publisher={Springer}
}

@article{brakke1995soap,
number = {4},
pages = {445-514},
title = {Soap films and covering spaces},
volume = {5},
year = {1995},
author = {Brakke, Kenneth A.},
issn = {1050-6926},
journal = {The Journal of geometric analysis},
}

@article{allard1972first,
author = {Allard, William K.},
title = {On the First Variation of a Varifold},
volume = {95},
year = {1972},
issn = {0003-486X},
journal = {Annals of mathematics},
language = {eng},
number = {3},
pages = {417-491},
publisher = {Princeton University},
}

@book{plateau1873statique,
  title={Statique exp{\'e}rimentale et th{\'e}orique des liquides soumis aux seules forces mol{\'e}culaires: Tome premier},
  author={Plateau, Joseph Antoine Ferdinand},
  volume={2},
  year={1873},
  publisher={Gauthier-Villars}
}

@article{almgren1975elliptic,
author = {Almgren, F. J.},
title = {Existence and regularity almost everywhere of solutions to elliptic variational problems with constraints},
volume = {81},
year = {1975},
issn = {0273-0979},
journal = {Bulletin (new series) of the American Mathematical Society},
language = {eng},
number = {1},
pages = {151-154},
}

@book{lickorish1997introduction,
author = {Lickorish, W. B. Raymond.},
series = {Graduate texts in mathematics ; 175},
title = {An introduction to knot theory },
year = {1997},
address = {New York ;},
booktitle = {An introduction to knot theory},
isbn = {038798254X},
language = {eng},
lccn = {97016660},
publisher = {Springer},
}

@article{amato2017constrained,
author = {Amato, Stefano and Bellettini, Giovanni and Paolini, Maurizio},
title = {Constrained BV functions on covering spaces for minimal networks and Plateau’s type problems},
volume = {10},
year = {2017},
address = {Berlin},
copyright = {Copyright Walter de Gruyter GmbH Jan 2017},
issn = {1864-8258},
journal = {Advances in calculus of variations},
language = {eng},
number = {1},
pages = {25-47},
publisher = {De Gruyter},
}

@book{ambrosio2000variation,
author = {Ambrosio, Luigi. and Fusco, Nicola. and Pallara, Diego.},
series = {Oxford mathematical monographs},
title = {Functions of bounded variation and free discontinuity problems / },
year = {2000},
address = {Oxford ;},
booktitle = {Functions of bounded variation and free discontinuity problems /},
isbn = {0198502451},
language = {eng},
lccn = {99046602},
publisher = {Clarendon Press},
}

@book{maggi2012sets,
publisher = {Cambridge University Press},
series = {Cambridge Studies in Advanced Mathematics ; 135},
title = {Sets of Finite Perimeter and Geometric Variational Problems : An Introduction to Geometric Measure Theory },
year = {2012},
language = {eng},
author = {Maggi, Francesco},
address = {Cambridge},
booktitle = {Sets of Finite Perimeter and Geometric Variational Problems : An Introduction to Geometric Measure Theory},
isbn = {9781139108133},
}

@book{bredon1993topology,
author = {Bredon, Glen E.},
address = {New York ;},
booktitle = {Topology and geometry},
isbn = {0387979263},
language = {eng},
lccn = {lc92031618},
publisher = {Springer-Verlag},
series = {Graduate texts in mathematics ; 139},
title = {Topology and geometry },
year = {1993},
}

@book{gallier2013guide,
author = {Gallier, Jean. and Xu, Dianna.},
address = {Berlin, Heidelberg},
booktitle = {A Guide to the Classification Theorem for Compact Surfaces},
edition = {1st ed. 2013.},
isbn = {3-642-34364-3},
language = {eng},
publisher = {Springer Berlin Heidelberg},
series = {Geometry and Computing},
title = {A Guide to the Classification Theorem for Compact Surfaces },
year = {2013},
}

@book{federer1969geometric,
author = {Federer, Herbert.},
address = {Berlin, New York},
booktitle = {Geometric measure theory},
isbn = {3540606564},
language = {eng},
lccn = {lc95049023},
publisher = {Springer},
series = {Classics in mathematics},
title = {Geometric measure theory },
year = {1996},
}

@article{novaga2022isoperimetricclusters,
title={Isoperimetric clusters in homogeneous spaces via concentration compactness},
author={Novaga, Matteo and Paolini, Emanuele and Stepanov, Eugene and Tortorelli, Vincenzo Maria},
journal={The Journal of Geometric Analysis},
volume={32},
number={11},
pages={263},
year={2022},
publisher={Springer}
}

@article{allard1975boundary,
author = {Allard, William K.},
issn = {0003-486X},
journal = {Annals of mathematics},
language = {eng},
number = {3},
pages = {418-446},
publisher = {Princeton University},
title = {On the First Variation of a Varifold: Boundary Behavior},
volume = {101},
year = {1975},
}

@article{bourni2016allard-type,
author = {Bourni, Theodora},
issn = {1864-8258},
journal = {Advances in calculus of variations},
language = {eng},
number = {2},
pages = {143-161},
title = {Allard-type boundary regularity for $C^{1,\alpha}$ boundaries},
volume = {9},
year = {2016},
}

@article{congedo1991multidimensionalsegmentation,
author = {Congedo, G. and Tamanini, I.},
address = {Paris},
copyright = {1991 1991 L'Association Publications de l'Institut Henri Poincaré. Published by Elsevier B.V.},
issn = {0294-1449},
journal = {Annales de l'Institut Henri Poincaré. Analyse non linéaire},
language = {eng},
number = {2},
pages = {175-195},
publisher = {Elsevier Masson SAS},
title = {On the existence of solutions to a problem in multidimensional segmentation},
volume = {8},
year = {1991},
}

@book{colding2011course,
author = {Colding, Tobias H. and Minicozzi, William P.},
address = {Providence, R.I},
booktitle = {A course in minimal surfaces},
isbn = {9780821853238},
language = {eng},
lccn = {2010044373},
publisher = {American Mathematical Society},
series = {Graduate studies in mathematics ; v. 121},
title = {A course in minimal surfaces },
year = {2011},
}

@inbook{schoen1984estimates,
title = {Estimates for Stable Minimal Surfaces in Three Dimensional Manifolds},
booktitle = {Seminar On Minimal Submanifolds. (AM-103), Volume 103},
author = {Richard Schoen},
editor = {Enrico Bombieri},
publisher = {Princeton University Press},
address = {Princeton},
pages = {111--126},
doi = {doi:10.1515/9781400881437-006},
isbn = {9781400881437},
year = {1984},
}

@article{de2017direct,
  title={A direct approach to Plateau's problem},
  author={De Lellis, Camillo and Ghiraldin, Francesco and Maggi, Francesco},
  journal={Journal of the European Mathematical Society},
  volume={19},
  number={8},
  pages={2219--2240},
  year={2017}
}

@article{harrison2016existence,
  title={Existence and soap film regularity of solutions to Plateau’s problem},
  author={Harrison, Jenny and Pugh, Harrison},
  journal={Advances in Calculus of Variations},
  volume={9},
  number={4},
  pages={357--394},
  year={2016},
  publisher={De Gruyter}
}

@article{harrison2015plateau,
  title={Plateau's Problem: What's Next},
  author={Harrison, Jenny and Pugh, Harrison},
  journal={arXiv preprint arXiv:1509.03797},
  year={2015}
}

@article{cesaroni2024periodic,
author = {Cesaroni, Annalisa and Novaga, Matteo},
copyright = {2024 The Author(s)},
issn = {0362-546X},
journal = {Nonlinear analysis},
language = {eng},
pages = {113522-},
publisher = {Elsevier Ltd},
title = {Periodic partitions with minimal perimeter},
volume = {243},
year = {2024},
}

@article{choe1989existence,
  title={On the existence and regularity of fundamental domains with least boundary area},
  author={Choe, Jaigyoung},
  journal={Journal of Differential Geometry},
  volume={29},
  number={3},
  pages={623--663},
  year={1989},
  publisher={Lehigh University}
}

@article{guaraco2024plateau,
author = {Guaraco, Marco A. M. and Lynch, Stephen},
address = {Berlin/Heidelberg},
copyright = {The Author(s) 2024},
issn = {0944-2669},
journal = {Calculus of variations and partial differential equations},
language = {eng},
number = {5},
publisher = {Springer Berlin Heidelberg},
title = {Plateau’s problem via the Allen–Cahn functional},
volume = {63},
year = {2024},
}

@article{bellettini2018covers,
issn = {2353-3382},
journal = {Geometric flows},
language = {eng},
number = {1},
pages = {57-75},
publisher = {De Gruyter Open},
title = {Covers, soap films and BV functions},
volume = {3},
year = {2018},
author = {Bellettini, Giovanni and Paolini, Maurizio and Pasquarelli, Franco and Scianna, Giuseppe},
}

@article{bellettini2018triple,
issn = {1463-9963},
journal = {Interfaces and free boundaries},
language = {eng},
number = {3},
pages = {407-436},
publisher = {European Mathematical Society Publishing House},
title = {Triple covers and a non-simply connected surface spanning an elongated tetrahedron and beating the cone},
volume = {20},
year = {2018},
author = {Bellettini, Giovanni and Paolini, Maurizio and Pasquarelli, Franco},
address = {Zuerich, Switzerland},
copyright = {European Mathematical Society},
}

@book{almgren1976existence,
  title={Existence and regularity almost everywhere of solutions to elliptic variational problems with constraints},
  author={Almgren, Frederick J},
  volume={165},
  year={1976},
  publisher={American Mathematical Soc.}
}

@online{Brakkeweb,
  author = {Brakke, Kenneth},
  title = {Soap Films on Knots},
  url = {https://kenbrakke.com/knots/},
}

@inproceedings{cesaroni2022minimal,
  title={Minimal periodic foams with equal cells},
  author={Cesaroni, Annalisa and Novaga, Matteo},
  booktitle={INdAM Workshop: Anisotropic Isoperimetric Problems \& Related Topics},
  pages={15--24},
  year={2022},
  organization={Springer}
}

@article{nobili2024lattice,
  title={Lattice tilings with minimal perimeter and unequal volumes},
  author={Nobili, Francesco and Novaga, Matteo},
  journal={Calculus of Variations and Partial Differential Equations},
  volume={63},
  number={9},
  pages={246},
  year={2024},
  publisher={Springer}
}

@article{cesaroni2025minimal,
  title={Minimal periodic foams with fixed inradius},
  author={Cesaroni, Annalisa and Novaga, Matteo},
  journal={Mathematika},
  volume={71},
  number={2},
  pages={e70020},
  year={2025},
  publisher={Wiley Online Library}
}
\end{document}